\documentclass[reqno]{amsart}

\usepackage{amsmath,amssymb,amsfonts,amsthm}
\providecommand{\texorpdfstring}[2]{#1}
\usepackage{graphicx}
\usepackage{hyperref}
\usepackage{xurl}
\hypersetup{hidelinks}
\usepackage{cleveref}
\usepackage{xcolor}
\usepackage{longtable,array,booktabs}
\usepackage[shortlabels]{enumitem}
\usepackage[numbers,sort&compress]{natbib}
\usepackage{pgfplots}
\newif\ifshowchanges
\showchangesfalse 
\newif\ifshowlabels
\showlabelsfalse
\ifshowlabels
  \usepackage[notref,notcite]{showkeys}
\fi
\pgfplotsset{compat=1.16}
\definecolor{s1}{HTML}{2A78D6}\definecolor{s2}{HTML}{EB6834}\definecolor{s3}{HTML}{1BAF7A}\definecolor{s4}{HTML}{EDA100}
\definecolor{p1}{HTML}{86B6EF}\definecolor{p2}{HTML}{5598E7}\definecolor{p3}{HTML}{2A78D6}\definecolor{p4}{HTML}{1C5CAB}\definecolor{p5}{HTML}{0D366B}

\numberwithin{equation}{section}

\newtheorem{theorem}{Theorem}[section]
\newtheorem{lemma}[theorem]{Lemma}
\newtheorem{proposition}[theorem]{Proposition}
\newtheorem{corollary}[theorem]{Corollary}
\newtheorem{conjecture}[theorem]{Conjecture}
\theoremstyle{definition}
\newtheorem{definition}[theorem]{Definition}
\newtheorem{remark}[theorem]{Remark}
\newtheorem{example}[theorem]{Example}
\theoremstyle{plain}
\newtheorem*{mainthm}{Main Theorem}

\newcommand{\R}{\mathbb{R}}
\newcommand{\norm}[1]{\left\lVert #1 \right\rVert}

\newcommand{\D}{\mathrm D}

\newcommand{\Dx}{\mathsf{D}}            
\newcommand{\K}{\mathcal{K}}            
\newcommand{\gb}{\bar g}
\newcommand{\ub}{\bar u}
\newcommand{\Tap}{T_{\mathrm{apex}}}
\newcommand{\aR}{\alpha_{\mathrm R}}
\newcommand{\bR}{\beta_{\mathrm R}}
\newcommand{\kR}{k_{\mathrm R}}
\newcommand{\cR}{c_{\mathrm R}}
\newcommand{\eR}{\varepsilon_{\mathrm R}}
\newcommand{\aZ}{\alpha_{\mathrm Z}}
\newcommand{\bZ}{\beta_{\mathrm Z}}
\newcommand{\kZ}{k_{\mathrm Z}}
\newcommand{\cZ}{c_{\mathrm Z}}
\newcommand{\eZ}{\varepsilon_{\mathrm Z}}
\newcommand{\CD}{C^1_{\Dx}}
\newcommand{\CDz}{C^1_{\Dx,0}}

\title[Two $(1+1)$D systems from 2D Boussinesq and 3D Euler]{2D inviscid Boussinesq equations and 3D axisymmetric Euler equations: (1) A unification ($Em$), (2) Finite-time blow-up of unified $(1+1)$D models on the horizontal ray ($Rm$) and on the vertical ray ($Zm$)}
\author{Yaoming Shi}
\address{California, USA}
\email{ymshi@protonmail.com}
\urladdr{https://orcid.org/0009-0008-7384-0947 (ORCID)}
\date{November 27, 2026}


\begin{document}

\begin{abstract}
We unify the 2D inviscid Boussinesq ($m=1$) and 3D axisymmetric Euler ($m=2$) equations in Hou--Li type variables into one system $(Em)$ and restrict it to the horizontal ray, $(Rm)$, and the vertical ray, $(Zm)$, of the meridian plane. On the axes the angular pressure terms produce $P_{\theta\theta}$, which we close by $x^{-2}P_{\theta\theta}=-f\,P_{xx}$; consistency at the apex $x=0$ forces $f(t,0)\in\{0,2\}$, and we take $f\equiv2$. Then $(Rm)$ and $(Zm)$ are nonlocal transport systems with a one-sided Hardy-type operator of order $k$, whose apex values satisfy one closed ODE system related to a Constantin--Lax--Majda (CLM) type equation. We prove local well-posedness, show that smoothness is lost by the apex blow-up time $\Tap$, and that the energy $E_m$ is finite if and only if $k>(m+3)/2$ or a flux at infinity vanishes. In explicit windows, smooth compactly supported data lose smoothness no later than $\Tap$ while $E_m$ stays finite; a self-similar analysis and computations for $(R2)$ show apex blow-up at $\Tap$ with bounded energy. Unlike in the CLM model, the swirl drives the blow-up. With two closure functions the reduction is exact and the apex dynamics contains one free function, the anisotropy of the pressure Hessian at the origin; if it stays in an explicit window, $(E1)$ and $(E2)$ blow up in finite time. Companion papers study this anisotropy and the scale-invariant case.
\end{abstract}

\maketitle

\tableofcontents

\section{Introduction}\label{sec:intro}

Whether smooth solutions of the 3D incompressible Euler equations can develop a singularity in finite time has long been a central open problem in fluid dynamics \cite{BKM984,MB2002,C2007,DE2023}. The 3D axisymmetric Euler equations with swirl and the 2D inviscid Boussinesq equations share the same structure away from the symmetry axis \cite{MB2002}, and both have been studied intensively, numerically \cite{LH14a,LH14b}, through one-dimensional models \cite{CLM1985,DeGregorio1990,OSW,CKY2015,CHKLSY2017,CHH21,CHH22}, through blow-up criteria along particle trajectories \cite{Chae2008,Chae2010}, and by rigorous blow-up results for solutions of limited regularity \cite{EJ2019,EJ2020,Elgindi21,CH2021} and for smooth data with boundary \cite{CH2022,CH23,CH25}. Very recently, Alp\"oge, Buckmaster and Coiculescu \cite{ABC2026,Coiculescu2026} announced a finite-time singularity of the 3D Euler equations from smooth odd data on $\mathbb T^3$, through a multiscale construction that is not axisymmetric; related constructions on $\R^3$ appear in \cite{OpenAI2026,Zheng2026}. We compare these works with the present one in Appendices~\ref{app:ABC} and~\ref{app:OpenAI}.

In earlier versions of this preprint the author introduced a unified system $(Em)$, written in Hou--Li type variables \cite{HLi2006}, which contains the 2D Boussinesq equations for $m=1$ and the 3D axisymmetric Euler equations for $m=2$, together with two $(1+1)$-dimensional systems obtained by restricting $(Em)$ to the symmetry axes of the meridian plane. The restriction uses polar coordinates $(x,\theta)$ and a ridge-flat ansatz. In that reduction the terms $\tfrac{a}{x^2}\sin\theta\,\tfrac{P_\theta}{\cos\theta}$ and $\tfrac{b}{x^2}\cos\theta\,\tfrac{P_\theta}{\sin\theta}$ in \eqref{eq:ZERO-1234} were set to zero on the axes. This is not correct: by L'H\^opital's rule,
\[
\Bigl(\frac{P_\theta}{\sin\theta}\Bigr)_{\theta\to0}=P_{\theta\theta}(0),\qquad
\Bigl(\frac{P_\theta}{\cos\theta}\Bigr)_{\theta\to\pi/2}=-P_{\theta\theta}(\pi/2),
\]
so the ray equations contain the second angular derivative of the pressure, and the pressure Poisson equation restricted to the ray is not independent (Section~\ref{sec:ray-derivation}). The ray systems are therefore not closed. We close them by the ansatz
\[
\tfrac{1}{x^2}P_{\theta\theta}=-f(t,x)\,P_{xx},
\]
motivated by the exact Yan-type solution of Section~\ref{Yan-exact-solution}. Consistency of the two rays at the apex $x=0$ forces $f(t,0)\in\{0,2\}$ (Remark~\ref{rem:f-roots}); the root $f(t,0)=0$ amounts to the reduction of the earlier versions. The exact Yan-type solution singles out the root $2$ (Proposition~\ref{prop:Yan-closure}), and throughout the analysis we take
\[
f(t,x)\equiv2 .
\]
This gives the system $(Rm)$ for $(u,v)$ on the horizontal ray \eqref{eq:ridge-dynamics-UV} and the system $(Zm)$ for $(\ub,\gb)$ on the vertical ray \eqref{eq:ridge-dynamics-UV-90}, with the energies \eqref{eq:unified-at-theta-0-b-Rm} and \eqref{eq:unified-at-theta-0-b-Zm}. All results of this revision concern $(Rm)$ and $(Zm)$ with the corrected reduction; they replace the $(1+1)$-dimensional systems of the earlier versions, and the linear and nonlinear stability framework of those versions has been removed.

\subsection{The question} We look for smooth data of finite energy $E_m$ such that the solution of $(Rm)$ or $(Zm)$ blows up at the apex $x=0$ in finite time while $E_m$ stays finite.

\subsection{Main results} Write $\Dx=x\partial_x$. The scaling parameters $a,b>0$ enter through the constants
\[
\kR=\frac{b-am}{2b-a},\qquad \kZ=\frac{b-am}{b-2am},\qquad
A=\frac{am^2+b}{b-am},\qquad B=\frac{b}{b-am}.
\]
\begin{enumerate}[label=(\Alph*)]
\item \emph{Structure} (Section~\ref{sec:structure}). For $f\equiv2$ the time derivative enters $(Rm)$ and $(Zm)$ through $\alpha+\beta\Dx$, with $\alpha/\beta=k$. Inverting this operator turns both systems into nonlocal transport systems (Proposition~\ref{prop:transport}); the nonlocal operator $\K_k$ is a one-sided Hardy average, directed towards the apex if $k>0$ and towards infinity if $k<0$.
\item \emph{Apex dynamics} (Section~\ref{sec:apex}). The apex values of both systems satisfy the same ODE system $\lambda v'=Av^2-Bu^2$, $\lambda u'=\tfrac{m^2}{2}uv$, consistently with the apex relation \eqref{eq:apex-relation}. It has a first integral, and we classify all its solutions (Theorem~\ref{thm:apexclass}). With the normalization $\lambda=m^2(q+1)$, $b/a=m^2(mq+2)/(m^2q-2)$ it becomes the CLM-$q$ equation $v_{tt}=vv_t-Qv^3$ (Proposition~\ref{prop:CLMq}); for $m=2$, $q=2$ the solutions are explicit (Example~\ref{ex:CLM2}), while for $m=1$ they are never elementary (Remark~\ref{rem:m1-elliptic}). On the separatrix the blow-up rate is $2\lambda/(m^2(T-t))$, the rate of the Yan-type solution, and for one value of its parameter the Yan-type solution is exactly the separatrix solution (Proposition~\ref{prop:Yan-closure}).
\item \emph{Well-posedness and an upper bound for the lifespan} (Section~\ref{sec:lwp}). Both systems are locally well posed in a class of functions with bounded Euler derivatives, and the maximal time satisfies $T^*\le\Tap$ (Theorem~\ref{thm:lwp}, Corollary~\ref{cor:Tstar}).
\item \emph{Exact cores} (Section~\ref{sec:core}). For $k>0$, data that are constant near the apex give a solution that is constant, and equal to the apex ODE solution, on a core bounded by a characteristic. The core shrinks to the apex along $(Rm)$ and expands along $(Zm)$, and we compute the core energy (Theorem~\ref{thm:core}).
\item \emph{Energy} (Section~\ref{sec:energy}). For $k>0$ and compactly supported data, $v\sim L(t)x^{-k}$ at infinity with an explicit flux $L(t)$; hence $E_m(t)<\infty$ if and only if $k>(m+3)/2$ or $L(t)=0$ (Theorem~\ref{thm:farfield}). For $k<0$ compact support is preserved. Combining (B)--(E) we obtain finite-energy singularity formation in explicit parameter windows (Theorem~\ref{thm:main}).
\item \emph{Inner layer} (Section~\ref{sec:selfsim}). The linearization around the apex solution is diagonalized by powers $x^n$ (Proposition~\ref{prop:modes}). This leads to a self-similar ansatz with the scale $x\sim(\Tap-t)^{\gamma_2}$ and a profile system of three ODEs (Lemma~\ref{lem:profile}), and to Conjecture~\ref{conj:selfsim} on blow-up at the apex at $\Tap$ with bounded energy.
\item \emph{Choice of parameters} (Section~\ref{sec:family}). The normalization $\lambda=m^2(q+1)$ covers exactly $b>am$ (Proposition~\ref{prop:family-range}); along it $(Rm)$ has infinite energy, $(Zm)$ becomes local at $b=2am$ (Proposition~\ref{prop:Zlocal}), and $(Zm)$ with $b/(2m)<a<b/m$ is numerically the cleanest finite-energy case. No ratio $b/a$ suits both rays.
\item \emph{From the apex to $(E1)$, $(E2)$} (Section~\ref{sec:exact}). With two closure functions the reduction is exact, and the exact apex dynamics has one free function, the anisotropy of the pressure Hessian at the origin. Yan-type solutions realize the apex blow-up exactly with infinite energy (Proposition~\ref{prop:Yan-apex}), and if the anisotropy stays in a window, then $(Em)$ itself blows up in finite time (Theorem~\ref{thm:conditional}). This is a conditional result: the window is a hypothesis on the solution for its whole lifetime, and whether finite-energy data keep the anisotropy in it is open (Section~\ref{subsec:remains} and the companion paper \cite{ShiII}).
\item \emph{Physical mechanism} (Section~\ref{sec:mechanism}). The apex blow-up is the opposite of the Constantin--Lax--Majda blow-up. In CLM the singularity forms where $u=0$ and is driven by the inertia of the strain, which the swirl opposes. In Theorem~\ref{thm:main} it requires $u\ne0$: the swirl (the buoyancy for $m=1$) creates a low-pressure core at the origin whose suction exceeds the centrifugal force, this drives the radial inflow and the axial outflow, and the inflow spins up the swirl (Proposition~\ref{prop:pressure-origin}, Lemma~\ref{lem:swirl-dominance}). Without swirl, compressive strain decays, consistently with the global regularity of 2D Euler and of swirl-free axisymmetric Euler.
\item \emph{Companion paper.} In a companion paper \cite{ShiII} we study the anisotropy of the pressure Hessian at the origin. We split the Hessian into a swirl part and a meridional part, and express the swirl part by an explicit angular moment of $u^2$. We give a unified formulation of $(Em)$ in the meridian plane with simulations of $(E2)$, linearize about the exact apex blow-up, and study data of limited regularity. For smooth localized data the swirl-driven regime turns out to be transient, because axial stretching drives the swirl towards a columnar profile. For swirl-free data the anisotropy is decided by the angular distribution of the vorticity, and formally the window holds exactly when the vorticity vanishes at the axis to an order less than $\frac13$.
\end{enumerate}

\begin{mainthm}[short form of Theorem~\ref{thm:main}]\label{mainthm}
Let $m=1$ and $2b<a<3b$, or $m=2$ and $2b<a<8b$. Let $(v_0,u_0)$ be smooth and compactly supported in $[0,\infty)$, with $u_0(0)\ne0$ and $v_0(0)>-\sqrt{\kappa}\,|u_0(0)|$, where $\kappa=2B/(2A-m^2)>0$. Then the solution of $(Rm)$ has finite energy $E_m(t)$ for $0\le t<T^*$, and $T^*\le\Tap<\infty$, where $\Tap$ is the explicit blow-up time of the apex ODE. If the data are constant near the apex, the solution is explicit on a core $\{x\le X(t)\}$ that shrinks to the apex like $(\Tap-t)^{2/m^2}$, and the energy of the core tends to zero.
\end{mainthm}

The analogous statement for $(Zm)$ holds for $m=1$, $b/3<a<b/2$ and $m=2$, $3b/16<a<b/4$, but there the core expands and its energy tends to infinity at $\Tap$.

\subsection{Numerical results} We solve $(Rm)$ and $(Zm)$ in Lagrangian logarithmic coordinates with an exact quadrature of $\K_k$ (Section~\ref{sec:num}). For $(R2)$ in its window $2b<a<8b$ and smooth decaying data, the solution develops a singularity at the apex: the singular time agrees with $\Tap$ to a relative accuracy of $10^{-4}$--$10^{-5}$, the singular point lies within $10^{-7}$ of the apex, and $E_2$ stays bounded. The inner layer near the apex follows the self-similar profile, with the predicted exponents (for $a=6b$ the layer moves like $(\Tap-t)^{1.551}$, predicted $(\Tap-t)^{1.550}$), and the remaining gap between $T^*$ and $\Tap$ scales with the size of rounding errors exactly as predicted by a short-wave instability (Section~\ref{sec:selfsim}). The same apex blow-up is observed for $a\ge8b$, where $E_2$ is infinite. For $(R1)$ in its window a gradient singularity forms away from the apex well before $\Tap$. For $(Zm)$ in its windows a steep front forms at the edge of the expanding core, and the computations do not resolve whether it becomes singular before $\Tap$. In the parameter regimes where the dynamics is not of this kind, the energy is infinite (Remark~\ref{rem:regimes}).

\subsection{Note on this revision}
Compared with earlier versions of this preprint:
\begin{enumerate}[label=(\roman*)]
\item the limits of $P_\theta/\sin\theta$ and $P_\theta/\cos\theta$ on the symmetry axes are now computed correctly; they produce $P_{\theta\theta}$, which is closed by the ansatz \eqref{eq:Pthetatheta-at-theta-0}, \eqref{eq:Pthetatheta-at-theta-90} (Section~\ref{sec:ray-derivation});
\item the $(1+1)$-dimensional systems of the earlier versions are replaced by $(Rm)$ and $(Zm)$, and all results on them are new (Sections~\ref{sec:structure}--\ref{sec:family});
\item the linear and nonlinear stability framework of the earlier versions has been removed;
\item Section~\ref{sec:exact} relates the apex dynamics to the full equations $(E1)$ and $(E2)$;
\item the energies of the ray systems include the Jacobian of the polar coordinates, so that their weight is $x^{m+2}$ (Section~\ref{sec:ray-derivation}); this determines the energy threshold $k>(m+3)/2$ and the windows \eqref{eq:windows};
\item Section~\ref{sec:mechanism} explains the physical mechanism of the apex blow-up and shows that it is the opposite of the Constantin--Lax--Majda blow-up of the earlier versions;
\item Appendix~\ref{app:ABC} compares our results with the construction of Alp\"oge, Buckmaster and Coiculescu \cite{ABC2026}, and Appendix~\ref{app:OpenAI} with the constructions of \cite{OpenAI2026,Zheng2026};
\item Section~\ref{subsec:Chae} relates our criteria to those of Chae \cite{Chae2008,Chae2010} and extends the one-sided criterion to arbitrary particle trajectories (Proposition~\ref{prop:trajectory});
\item the pressure Hessian at the origin, the unified formulation with simulations of $(E2)$, the linearizations, data of limited regularity, and the discussion of the announcements of blow-up for the Navier--Stokes equations \cite{AB2026NS,OpenAI2026NS}, which were included in intermediate drafts of this revision, now form a companion paper \cite{ShiII};
\item the apex dynamics of scale-invariant solutions, which depends on the direction at the origin and was studied in intermediate drafts of this revision, is now the subject of a separate paper \cite{ShiIII}.
\end{enumerate}

\subsection{Organization} Section~\ref{sec:derivation} introduces $(Em)$ and its pressure equation. Section~\ref{sec:ray-derivation} derives $(Rm)$ and $(Zm)$ with the corrected axis limits, their energies and their apex dynamics. Section~\ref{Yan-exact-solution} recalls the Yan-type exact solution, tests the closure on it, and gives its first-order expansion. Sections~\ref{sec:structure}--\ref{sec:energy} contain the analysis described above, Section~\ref{sec:num} the numerical study, Section~\ref{sec:selfsim} the analysis of the inner layer, Section~\ref{sec:family} the choice of parameters, Section~\ref{sec:exact} the relation with the full equations, Section~\ref{sec:mechanism} the physical mechanism of the apex blow-up and its comparison with the CLM blow-up, Section~\ref{sec:compare} a comparison with related blow-up mechanisms, and Section~\ref{sec:concl} the conclusion and open problems. Appendix~\ref{app:ABC} compares our results with the recent construction of Alp\"oge, Buckmaster and Coiculescu and shows that a one-sided bound on the radial pressure Hessian at the origin suffices for blow-up (Proposition~\ref{prop:one-sided}). Appendix~\ref{app:OpenAI} discusses the constructions of \cite{OpenAI2026,Zheng2026} and the ideas they suggest, including a Lagrangian form of the apex dynamics (Lemma~\ref{lem:Lagrangian-apex}). Scale-invariant solutions, for which the anisotropy at the origin is no longer free, are treated in \cite{ShiIII}. The pressure Hessian at the origin, a unified formulation with simulations, data of limited regularity and the announcements of blow-up for the Navier--Stokes equations are treated in the companion paper \cite{ShiII}.

\section{2D inviscid Boussinesq equations and the 3D axisymmetric Euler equations}\label{sec:derivation}
\subsection{Velocity--pressure formulation and Hou--Li type variables}\label{sec:E2-derivation}
In this section we simultaneously convert the velocity--pressure form of the 2D inviscid Boussinesq equations (see, for example, Wu~\citep{Wu2012}, Elgindi--Jeong~\citep{EJ2020}, and Kiselev--Park--Yao~\citep{KPY2022}) and the 3D axisymmetric Euler equations (see, for example, Chae--Lee~\citep{CL2002}, Drazin--Riley~\citep{DR2006}, Chen--Fang--Zhang~\citep{CFZ2015}) into a new formulation in terms of vorticity building blocks. In this form, the structure of the vortex-stretching and convection terms becomes transparent, which makes the study of the apex blow-up mechanism in compressed coordinates more tractable.

In the velocity-pressure form, the 2D inviscid Boussinesq equations for velocity $\boldsymbol{u}=u_2\boldsymbol{e}_2+u_3\boldsymbol{e}_3$, pressure $P$ and buoyancy scalar $\vartheta$ in $(x_2,x_3)\in\R^2$ and the 3D axisymmetric Euler equations on the semimeridian plane $(r\geq 0, z\in\R)$ are given by

\begin{minipage}{0.40\textwidth}
\begin{equation}\label{eq:2DBoussinesq}
	\left\{\begin{aligned}
		&\text{2D inviscid Boussinesq}\\
		&\tfrac{\mathrm{\tilde{D}}}{\mathrm{D}t}\vartheta=0,\\
		&\tfrac{\mathrm{\tilde{D}}}{\mathrm{D}t}u_2
		=-\partial_2P +\vartheta,\\
		&\tfrac{\mathrm{\tilde{D}}}{\mathrm{D}t} u_3=-\partial_3P,\\
		&\partial_2u_2\,+ \partial_3u_3 =0,\\
		&\tfrac{\mathrm{\tilde{D}}}{\mathrm{D}t}=\partial_t+u_2 \partial_2+ 
		u_3 \partial_3,\\
	\end{aligned}\right.
\end{equation}
\end{minipage}
\begin{minipage}{0.45\textwidth}
	\begin{equation}\label{eq:3DEuler}
		\left\{
		\begin{aligned}
			&\text{3D axisymmetric Euler}\\
			&\tfrac{\mathrm{\tilde{D}}}{\mathrm{D}t}(rv^\phi)^2
			=0,\\
			&\tfrac{\mathrm{\tilde{D}}}{\mathrm{D}t}v^r
			=-\partial_r P+\tfrac{1}{r^3}(rv^\phi)^2,\\
			&\tfrac{\mathrm{\tilde{D}}}{\mathrm{D}t} v^z
			=-\partial_z P,\\
			&\partial_r(rv^r)+\partial_z (rv^z)=0,\\
			&\tfrac{\mathrm{\tilde{D}}}{\mathrm{D}t}:=\partial_t+v^r \partial_r+
			v^z \partial_z.
		\end{aligned}
		\right.
	\end{equation}
\end{minipage}\\

The Poisson pressure equations are given by\\
\begin{minipage}{0.38\textwidth}
\begin{equation}\label{eq:PPE-Boussinesq}
	\left\{
	\begin{aligned}
		&\text{PPE for Boussinesq}\\
		&-\bigl(\partial_2^2p+\partial_3^2p\bigr)\\
		&=(\partial_2 u_2)^2+(\partial_3 u_3)^2\\
		&+2(\partial_3 u_2)\partial_2 u_3-\partial_2\vartheta
	\end{aligned}
	\right.
\end{equation}
\end{minipage}
\begin{minipage}{0.58\textwidth}
	\begin{equation}\label{eq:PPE-Euler}
		\left\{
		\begin{aligned}
			&\text{PPE for Euler}\\
			&-\bigl(\partial_r^2p+\tfrac1r\partial_r p+\partial_z^2p\bigr)\\
			&=(\partial_r v^r)^2+\left(\tfrac{1}{r}v^r\right)^2+(\partial_z v^z)^2\\
			&+2(\partial_r v^z)\partial_z v^r-\tfrac{1}{r}v^\phi\partial_r v^\phi
		\end{aligned}
		\right.
	\end{equation}
\end{minipage}\\

The energies of the 2D Boussinesq equations and of the 3D axisymmetric Euler equations are, respectively,\\
\noindent\begin{minipage}{\textwidth}
	\begin{equation}\label{eq:energy-Boussinesq}
		\begin{aligned}
			&E_1(t)=\tfrac12\int_{-\infty}^{\infty}\int_{-\infty}^\infty \Bigl((u_2)^2+(u_3)^2\Bigr)\mathrm{d}x_2\mathrm{d}x_3.
		\end{aligned}
	\end{equation}
\end{minipage}\\
\noindent\begin{minipage}{\textwidth}
	\begin{equation}\label{eq:Energy-Euler}
		\begin{aligned}
			&E_2(t)=\int_{-\infty}^{\infty}\int_0^\infty \Bigl((v^r)^2+(v^\phi)^2+(v^z)^2\Bigr)r\mathrm{d}r\mathrm{d}z.
		\end{aligned}
	\end{equation}
\end{minipage}\\
We omitted the unimportant factor $1/2$ in \eqref{eq:Energy-Euler} for future convenience.

It is well-known (see, e.g., Majda and Bertozzi~\cite{MB2002}) that the 2D Boussinesq equations~\eqref{eq:2DBoussinesq} have properties similar to those of the 3D axisymmetric Euler equations~\eqref{eq:3DEuler}, at least away from the symmetry axis ($r=0$). Indeed, comparing~\eqref{eq:2DBoussinesq} with~\eqref{eq:3DEuler}, we see that the buoyancy force $\vartheta$ in the Boussinesq equation plays the role of the axisymmetric Euler centrifugal force $(rv^\phi)^2/r^3$. Equivalently, after division by the axis coordinate, the Boussinesq quantity $u^2=\vartheta/x_2$ corresponds to the Euler quantity $(v^\phi/r)^2$. The real difference between the two systems only emerges near the axis of symmetry, where the factors of $r$ can conceivably change the nature of the dynamics.

In this section we make the analogy precise by removing the caveat ``away from the symmetry axis.''

For the 2D inviscid Boussinesq equations, we assume that $u_2$ is odd in $x_2$ and even in $x_3$, that $u_3$ is even in $x_2$ and odd in $x_3$, that $P$ is even in $(x_2,x_3)$, and that $\vartheta$ is odd in $x_2$ and even in $x_3$. Define the Hou--Li~\citep{HLi2006} type variables 
\begin{equation}\label{eq:uvg-def-B}
	\{v,g,u^2,p\}:=\left\{-\frac{u_2}{x_2},\frac{u_3}{x_3},\frac{\vartheta}{x_2},P\right\}.
\end{equation}
Then~\eqref{eq:2DBoussinesq} can be converted to the $(m=1)$ version $(E1)$ of system $(Em)$ \eqref{eq:Em} (with $(r,z)=(x_2,x_3)$ for notational convenience).

For the 3D axisymmetric Euler equations, we assume that $\left(v^{\phi},v^{r}\right)$ are odd in $r$ and even in $z$, while $v^z$ is even in $r$ and odd in $z$, and $P$ is even in $(r,z)$. Define the Hou--Li~\citep{HLi2006} type variables by
\begin{equation}\label{eq:uvg-def}
	\{v,u,g,p\}:=\left\{-\frac{v^r}{r},\frac{v^{\phi}}{r},\frac{v^z}{z},P\right\}.
\end{equation}
Then, as shown in Shi~\citep{Shi2026B}, \eqref{eq:3DEuler} can be converted to the $(m=2)$ version $(E2)$ of system $(Em)$ \eqref{eq:Em}:

\begin{equation}\label{eq:Em}
	\left\{
	\begin{aligned}
		&\tfrac{\mathrm{D}}{\mathrm{D}t}u
				=\tfrac12m^2uv,\qquad\qquad t\in[0,T), (r,z)\in\R^2\\
		&\tfrac{\mathrm{D}}{\mathrm{D}t}v
				=v^2-u^2+\tfrac{1}{r}p_r\\
		&\tfrac{\mathrm{D}}{\mathrm{D}t}g
		=-g^2\,\,\,\,\,-\tfrac{1}{z}p_z\\
		&z\partial_z g-r\partial_r v+g-mv=0,\\
		&\tfrac{\mathrm{D}}{\mathrm{D}t}:=\partial_t-v r\partial_r+g z\partial_z.
	\end{aligned}
	\right.
\end{equation}

The corresponding Poisson pressure equations (PPE), \eqref{eq:PPE-Boussinesq} and \eqref{eq:PPE-Euler}, can also be unified:
\begin{equation}\label{eq:PPE-m}
	\left\{
	\begin{aligned}
		-&\bigl(\partial_r^2p+\tfrac1r(m-1)\partial_r p+\partial_z^2p\bigr)\\
		=&((zg)_z)^2+((rv)_r)^2-2rzv_zg_r-u(u+2ru_r)+(m-1)(v^2-u^2).
	\end{aligned}
	\right.
\end{equation}

\begin{remark}\label{similarity}
	We call \eqref{eq:Em} the unified system (Em). Thus (E1) stands for 2D inviscid Boussinesq equations and (E2) stands for 3D axisymmetric Euler equations. They differ only by a numerical coefficient $m$ in two places. The author conjectures that (E3) corresponds to the 4D axisymmetric Euler equations; see \cite{ShiIII}.
\end{remark}

\smallskip
\begin{remark}\label{rem:symetric-in-rz}
	From the inspection of \eqref{eq:Em}, we notice that if the initial conditions for $\{u,v,g,p\}$ are symmetric in $(r,z)$, then the PDE system preserves these symmetry properties. In this sense, we regard~\eqref{eq:Em} as being defined on $\R^2$. Throughout this paper, we assume that the initial conditions for $\{u,v,g,p\}$ are symmetric in $(r,z)$.
\end{remark}

\begin{remark}\label{rem:building-blocks}
	We call $\{u,v,g\}$ the building blocks of vorticity, because their physical dimensions agree with those of the vorticity $\boldsymbol{\omega}=\nabla \times\boldsymbol{v}$.
	Also the quadratic vortex stretching terms are greatly simplified: $(uv,\ v^2-u^2,\ -g^2)$.
\end{remark}
\subsection[CLM]{Constantin--Lax--Majda equations}
\begin{remark}\label{rem:CLM-analogue}
	We regard \eqref{eq:Em} as a two-dimensional Eulerian analogue of the Constantin--Lax--Majda equations \citep{CLM1985}. In the case $(p=0,m=2)$, the first two equations in \eqref{eq:Em} reduce to the Constantin--Lax--Majda system after the identification $\tfrac{\mathrm{D}}{\mathrm{D}t}=\tfrac{\partial}{\partial t}$ and $u(t,r,\cdot)=\tfrac12\omega(t,r)$, $v(t,r,\cdot)=\tfrac12H(\omega)(t,r)$, or $u(t,\cdot,z)=\tfrac12\omega(t,z)$, $v(t,\cdot,z)=\tfrac12H(\omega)(t,z)$.\\
	
	\noindent\begin{minipage}{\textwidth} 
		\begin{equation}\label{eq:CLM2}
			\left\{\begin{aligned}
				&u_t=
				2vu,\qquad\qquad x\in\mathbb{R}\\
				&v_t
				=v^2-u^2.\\
			\end{aligned}\right.
		\end{equation}
	\end{minipage}\\
	
	In the Constantin--Lax--Majda equations, $v=\tfrac12H(\omega)$ is a function of $u=\tfrac12\omega$. In \eqref{eq:Em}, $v$ is independent of $u$.
\end{remark}

We now present the explicit finite-time blow-up solutions of the Constantin--Lax--Majda system as a benchmark for further comparison.

Constantin, Lax, and Majda converted \eqref{eq:CLM2} into the scalar complex ODE with dependent variable $z(t,x)=v(t,x)+i\,u(t,x)$ and found the explicit solution:\\

\noindent\begin{minipage}{\textwidth} 
	\begin{equation}\label{NSBOuOddx4C}
		\left\{\begin{aligned}
			z_t(t,x)&=z^2(t,x),\qquad x\in\mathbb{R},\\
			z(t,x)&=\tfrac{1}{f(x)+ig(x)-t}.
		\end{aligned}\right.
	\end{equation}
\end{minipage}\\

Substituting the initial data into \eqref{NSBOuOddx4C} yields the following result.

\begin{theorem}[Constantin--Lax--Majda explicit formula]\label{thm:CLM}
	
	Suppose $u_0(x)=u(0,x)$ is a smooth function that decays sufficiently rapidly as
	$|x|\to\infty$, and let $v_0(x)=v(0,x)=H(u_0)(x)$. Then the solution to the model vorticity system~\eqref{eq:CLM2} is explicitly given by
	
	\noindent\begin{minipage}{\textwidth} 
		\begin{equation}\label{eq:CLM-2-u-v}
			\left\{\begin{aligned}
				&u(t,x)=\frac{u_0(x)}{\left[1-t v_0(x)\right]^2+t^2u_0^2(x)},\\
				&v(t,x)=\frac{v_0(x)
					\left[1-t v_0(x)\right]-tu_0^2(x)}{\left[1-t v_0(x)\right]^2+t^2u_0^2(x)}.\\
			\end{aligned}\right.
		\end{equation}
	\end{minipage}
\end{theorem}

\begin{theorem}[Constantin--Lax--Majda breakdown criterion]\label{thm:CLM-2}
	The smooth solution to the CLM system~\eqref{eq:CLM2} blows up in finite time if and only if the set
	\[
	Z:=\{x\in\mathbb{R}: u_0(x)=0 \text{ and } v_0(x)>0\}
	\]
	is nonempty. If $M:=\max_{x\in Z} v_0(x)$ and $\bar x\in Z$ satisfies $v_0(\bar x)=M$, then the earliest blow-up time is
	\[
	T=\frac{1}{M},
	\]
	and $v(t,\bar x)\to +\infty$ as $t\uparrow T$. Moreover, at such a blow-up point one has $u(t,\bar x)\equiv 0$ for all $t$, so the singularity is carried by the $v$-component.
\end{theorem}
If $u^2$ is solved from \eqref{eq:CLM2}(2) and substituted into \eqref{eq:CLM2}(1), then one obtains a second order ODE for $v(t,x)$. Setting $\tau=6t$, $v(t,x)=V(\tau,x)$, and $u(t,x)=U(\tau,x)$, we can convert this ODE to the $q=2$ version of the following ODE (CLM-$q$):
\begin{equation}\label{eq:CLM-q}
	\left\{
	\begin{aligned}
		V_{\tau\tau}&=VV_\tau-\tfrac{q}{2(q+1)^2}V^3,\\
		V(0)&=V_0,\\
		V_\tau(0)&=\tfrac{1}{2(q+1)}\bigl(V_0^2-U_0^2\bigr)
	\end{aligned}\right.
\end{equation}
	
\section{The ray systems \texorpdfstring{$(Rm)$}{(Rm)} and \texorpdfstring{$(Zm)$}{(Zm)}}\label{sec:ray-derivation}
	We now introduce a stream function $\tilde \psi$ and augment~\eqref{eq:Em} with two additional relations, obtaining system~\eqref{E2}. To reserve the symbols $(u,v,g,p)$ for the ray variables introduced below, we place tildes on the original unknowns. Thus system (E$m$) in~\eqref{E2} consists of five dependent variables $(\tilde u,\tilde v,\tilde g,\tilde p,\tilde \psi)$, viewed as even functions of $(r,z)$ on the meridian plane, together with seven equations; the last one defines $\tfrac{\D}{\D t}$. For future convenience, we also introduce $(t,r,z)$-scaling parameters $(\lambda>0,\sqrt{a}>0,\sqrt{b}>0)$.
	
\begin{equation}\label{E2}
\left\{\begin{aligned}
0&=\tfrac{\D}{\D t}\,\tilde u
- \tfrac12m^2\,\tilde u\,\tilde v, \qquad(t,r,z)\in [0,T)\times\R^2\\
0&=\tfrac{\D}{\D t}\,\tilde v
 -\tilde v^{2}+\tilde u^{2}-\tfrac{a}{r}\tilde p_{r}, \\
0&=\tfrac{\D}{\D t}\,\tilde g
 +\tilde g^{2}+\tfrac{b}{z}\,\,\tilde p_{z},\\
0&=z\partial_z\tilde g-r\partial_r\tilde v+\tilde g-m\tilde v, \\
0&=\tilde v-\tilde\psi-z\partial_z\tilde\psi, \\
0&=\tilde g-m\tilde\psi-r\partial_r\tilde\psi, \\
\tfrac{\D}{\D t}:&=\lambda\partial_t-\tilde
v\,r\partial_r+\tilde g\,z\partial_z.
\end{aligned}\right.
\end{equation}

\begin{remark}[No redundancy]\label{rem:no-redundancy}	
	Substituting \eqref{E2}(5) and \eqref{E2}(6) into \eqref{E2}(4) yields an identity, so no redundancy is introduced.
\end{remark}

Under the scaling change from $(t,r,z)$ to $(\lambda t, \sqrt{a}r, \sqrt{b}z)$, the unified PPE \eqref{eq:PPE-m} becomes:
\begin{equation}\label{eq:PPE-m-b}
	\left\{
	\begin{aligned}
		-&\bigl(a\tilde p_{rr}+\tfrac {a}{r}(m-1)\tilde p_r+b\tilde p_{zz}\bigr)\\
		=&((z\tilde g)_z)^2+((r\tilde v)_r)^2-2rz\tilde v_z\tilde g_r-\tilde u(\tilde u+2r\tilde u_r)+(m-1)(\tilde v^2-\tilde u^2).
	\end{aligned}
	\right.
\end{equation}

\subsection{Polar coordinates \texorpdfstring{$(x=\sqrt{r^2+z^2},\ \theta=\arctan(z/r))$}{(x, theta)}}
We use polar coordinates on the meridian plane:
\begin{equation}\label{eq:polarRxi}
	r=x \cos\left(\theta\right), \quad z=x \sin\left(\theta\right).
\end{equation}	
\begin{remark}\label{rem:polar}
	The polar coordinates $(x,\theta)$ on the meridian plane $(r,z)$ are also the spherical coordinates $(x,\theta,\phi)$ (with north pole at $\theta=\pi/2$) for 3D axisymmetric functions in $\R^3$.
	The full symmetry geometry contains the four distinguished axis directions
	\[
	\theta=0,\quad \theta=\pm{\pi}/{2},\quad \theta=\pi.
	\]
	For the analysis below we restrict to the first-quadrant wedge
	\[
	x\geq 0,\qquad \theta\in\left[0,{\pi}/{2}\right].
	\]
	The other three quadrants can be treated in the same way by the corresponding symmetry extension.
\end{remark}

We write the solutions as
\begin{equation}\label{eq:background-0}
	\begin{aligned}
		\tilde u&=U(t,x,\theta),\quad
		\tilde v=V(t,x,\theta),\quad\tilde g=G(t,x,\theta),\quad
		\tilde p=P(t,x,\theta).
	\end{aligned}
\end{equation}

After substituting~\eqref{eq:background-0} into the first four equations in~\eqref{E2}, we obtain four equations with the following structure:

\begin{equation}\label{eq:ZERO-1234}
\left\{
\begin{aligned}
\lambda U_t-\tfrac12m^2V\,U&=-x U_{x}W-\tfrac12\sin(2\theta)(G+V)U_\theta,\\[2mm]
\lambda V_{t}-V^2+U^2&=-x V_{x} W-\tfrac12\sin(2\theta)(G+V)V_\theta+\tfrac{a}{x}P_{x}-\tfrac{a}{x^2}\sin\theta \left(\tfrac{P_\theta}{\cos\theta}\right),\\[2mm]
\lambda G_{t}+G^2&=-x G_{x} W-\tfrac12\sin(2\theta)(G+V)G_\theta-\tfrac{b}{x}P_{x}-\tfrac{b}{x^2}\cos\theta \left(\tfrac{P_\theta}{\sin\theta}\right),\\[2mm]
G-mV&=-x \partial_x W-\tfrac12\sin(2\theta)(G+V)_\theta,\\
W:&=G\sin ^2(\theta ) -V\cos ^2(\theta ).
		\end{aligned}
\right.
\end{equation}

In the polar coordinates, the unified PPE \eqref{eq:PPE-m} becomes
\begin{equation}\label{eq:PPE-m-R-theta}
	\left\{
	\begin{aligned}
		-&\Bigl(\tfrac{1}{2} ((a-b) \cos (2 \theta )+a+b) P_{xx}+\tfrac{ ((b-a) \cos (2 \theta )+2 a m-a+b)}{2 x}P_x+\tfrac{(b-a) \sin (2 \theta ) }{x}P_{x\theta}\Bigr)\\
		-&\Bigl(\tfrac{((b-a) \cos (2 \theta )+a+b) }{2 x^2}P_{\theta\theta}-\tfrac{ ((b-a) \cos (2 \theta )+a (m-2)+b)\sin(\theta )}{x^2}\  \tfrac{P_\theta}{\cos(\theta)}\Bigr)\\
		=&\left(\cos (\theta ) \left(x \cos (\theta ) V_x-\sin (\theta ) V_\theta\right)+V\right)^2\\
		+&\left(\sin (\theta ) \left(x \sin (\theta ) G_x+\cos (\theta ) G_\theta\right)+G\right)^2\\
		+&\sin (2 \theta ) \left(\sin (\theta ) G_\theta-x \cos (\theta ) G_x\right) \left(x \sin (\theta ) V_x+\cos (\theta ) V_\theta\right)\\
		-&U \left(2 \cos (\theta ) \left(x \cos (\theta ) U_x-\sin (\theta ) U_\theta\right)+U\right)\\
		+&(m-1) \left(V^2-U^2\right).\\
	\end{aligned}
	\right.
\end{equation}

We now examine how the equations simplify under the following ridge-flat ansatz (in the directions normal to the ridge):
\begin{equation}\label{eq:flat}
\left\{\begin{aligned}
	&(P_\theta,V_\theta,U_\theta,G_\theta)|_{\theta_0}=0,\qquad\theta_0=0,\pi,\quad x\geq 0,\quad t\in [0,T),\\[2mm]
	&(P_\theta,V_\theta,U_\theta,G_\theta)|_{\theta_1}=0,\qquad\theta_1=\pm\tfrac{\pi}{2},\quad x\geq 0,\quad t\in [0,T).
\end{aligned}\right.
\end{equation}

\begin{remark}\label{rem:flat-preserved}
	We notice that the ansatz \eqref{eq:flat} is equivalent to the statement that $(U,V,G,P)(t,x,\theta)$ are even functions of $(r,z)=(x\cos\theta,x\sin\theta)$. As noted in Remark~\ref{rem:symetric-in-rz}, if the initial conditions have this symmetry, the dynamical equations preserve it. Thus the ridge-flatness ansatz is automatically preserved by the dynamics.
\end{remark}

\subsection{Symmetry axis at $\theta=0,\pi$ and axis-restricted functions}
	
	Applying the ridge-flat ansatz \eqref{eq:flat} and setting $\theta=\theta_0=0 \text{ or }\pi$ in \eqref{eq:ZERO-1234} and in \eqref{eq:PPE-m-R-theta} leads to the \textbf{horizontal symmetry-axis reduction} (omitting the arguments $(t,x,0)$)
\begin{equation}\label{eq:ZERO-1234-theta-0}
	\left\{
	\begin{aligned}
		\lambda U_t-\tfrac12m^2V\,U&=x VU_{x},\\[2mm]
		\lambda V_{t}-V^2+U^2&=x VV_{x} +\tfrac{a}{x}P_{x},\\[2mm]
		\lambda G_{t}+G^2&=x VG_{x}-\tfrac{b}{x}P_{x}-\tfrac{b}{x^2} P_{\theta\theta},\\[2mm]
		G-mV&=x V_x,\\
	\end{aligned}
	\right.
\end{equation}
and
\begin{equation}\label{eq:PPE-m-R-theta-0}
	\left\{
	\begin{aligned}
		-&\bigl(aP_{xx}+\tfrac{2(am-a+b)}{x} P_x+\tfrac{b}{x^2}P_{\theta\theta}\bigr)\\
		=&G^2+(m-1) \left(V^2-U^2\right)-U \left(2 x U_x+U\right)+\left(x V_x+V\right)^2.\\
	\end{aligned}
	\right.
\end{equation}
In \eqref{eq:ZERO-1234-theta-0}, we applied L'H\^opital's rule $$\left(\tfrac{P_\theta}{\sin\theta}\right)_{\theta\to0}=\left(\tfrac{P_{\theta\theta}}{\cos\theta}\right)_{\theta\to0}=P_{\theta\theta}(0).$$ We can get rid of $G(t,x,0)$ using \eqref{eq:ZERO-1234-theta-0}(4).

\begin{remark}\label{rem:ray-not-closed}
	The ray system \eqref{eq:ZERO-1234-theta-0} by itself is not closed, because of the presence of $P_{\theta\theta}(t,x,0)$. By using PPE \eqref{eq:PPE-m-R-theta-0}, we can convert it to a closed ray system, that only depends on $(U,V,G,U_x,V_x,G_x,P_x,P_{xx})(t,x,0)$.
\end{remark}
	
	\begin{equation}\label{eq:ridge-dynamics}
\left\{
		\begin{aligned}
		\lambda U_t-\tfrac12m^2V\,U&=x VU_{x},\\[2mm]
\lambda V_{t}-V^2+U^2&=x VV_{x} +\tfrac{a}{x}P_{x},\\[2mm]
\lambda (mV+xV_x)_{t}+(mV+xV_x)^2&=x V(mV+xV_x)_{x}-\tfrac{b}{x}P_{x}-\tfrac{b}{x^2} P_{\theta\theta},\\[2mm]
		-\bigl(aP_{xx}+\tfrac{2(am-a+b)}{x} P_x+\tfrac{b}{x^2}P_{\theta\theta}\bigr)
&=(mV+xV_x)^2+(m-1) \left(V^2-U^2\right)\\
&-U \left(2 x U_x+U\right)+\left(x V_x+V\right)^2.\\
		\end{aligned}\right.
\end{equation}

From \eqref{eq:ridge-dynamics}(2,3,4), we can solve for $P_x,P_{xx},P_{\theta\theta}$ and obtain:
\begin{equation}\label{eq:ridge-dynamics-Rm-solution}
\left\{
\begin{aligned}
aP_x&= x \left(U^2+\lambda  V_t-V \left(x V_x+V\right)\right),\\[2mm]
aP_{xx}&=2 x U_x U+U^2+\lambda  V_t-x \left(-\lambda  V_{tx}+x (V_x)^2+V \left(4 V_x+x V_{xx}\right)\right)-V^2,\\[2mm]
abP_{\theta\theta}&= x^2 \Bigl(\left(b-a m^2\right) V^2+x V \left((a(1-m)+b) V_x+a x V_{xx}\right)\\
&\qquad-a \left(\lambda  m V_t+x^2 (V_x)^2+\lambda  x V_{tx}\right)-b U^2-b \lambda  V_t\Bigr).
\end{aligned}\right.
\end{equation}

Our first idea was to obtain the PDE for $V$ from \eqref{eq:ridge-dynamics-Rm-solution}(1,2), using $(P_x)_x=P_{xx}$.  However, this gives only the identity $0\equiv0$. This means that the restricted PPE is not independent.

So for the (1+1)D model ($Rm$) on the horizontal ray of $\theta_0=0,\pi$, we propose a new ansatz 
\begin{equation}\label{eq:Pthetatheta-at-theta-0}
	\tfrac{1}{x^2}P_{\theta\theta}(t,x,0)=-f(t,x)\,P_{xx}(t,x,0), \quad f(t,x)\in \mathbb R.
\end{equation} 
The justification can be found in \eqref{YanLikeUnifiled}(4).

Substituting this into \eqref{eq:ridge-dynamics}(3), and using $(P_x)_x=P_{xx}$ from \eqref{eq:ridge-dynamics}(2,3) to eliminate $P$, we finally obtain (with $u(t,x):=U(t,x,0), v(t,x):=V(t,x,0)$): 

	\begin{equation}\label{eq:ridge-dynamics-UV}
	Rm: \left\{
	\begin{aligned}
		&\lambda u_t-\tfrac12m^2v\,u=x vu_{x},\\[2mm]
0&=v^2 \left(-a m^2-b f+b\right)+\lambda  v_t (b (f-1)-a m)+b (f-1) u^2\\
&+x \Bigl(v v_x (a(1-m)-4 b f+b)+\lambda  (b f-a) v_{xt}+2 b f u u_x\Bigr)\\
&+x^2 \Bigl(-(a+b f) (v_x)^2+(a-b f) v v_{xx}\Bigr)
	\end{aligned}\right.
\end{equation}

The terms $\bigl(xu_{x},xv_{tx},x^2(v_x)^2,x^2v_{xx}\bigr)$ in \eqref{eq:ridge-dynamics-UV} are of the form $\Dx(\cdot)$ or products of such terms, with $\Dx=x\partial_x=-\tfrac1x\,\partial_{(1/x)}$; for functions that are smooth at $x=0$ and at $x\to\infty$ (i.e.\ smooth in $1/x$ at $1/x=0$) they vanish at the apex and at infinity. The essential point is the exact closure at the apex $x=0$.

\subsection{Symmetry axis at $\theta=\pm\pi/2$ and axis-restricted functions}

Applying the ridge-flat ansatz \eqref{eq:flat} and setting  $\theta=\theta_0=\pi/2 \text{ or }-\pi/2$ in \eqref{eq:ZERO-1234} and in \eqref{eq:PPE-m-R-theta} leads to the \textbf{vertical symmetry-axis reduction} (omitting the arguments $(t,x,\pi/2)$)
\begin{equation}\label{eq:ZERO-1234-theta-90}
	\left\{
	\begin{aligned}
		\lambda U_t-\tfrac12m^2V\,U&=-x GU_{x},\\[2mm]
		\lambda V_{t}-V^2+U^2&=-x GV_{x} +\tfrac{a}{x}P_{x}+\tfrac{a}{x^2} P_{\theta\theta},\\[2mm]
		\lambda G_{t}+G^2&=-x GG_{x}-\tfrac{b}{x}P_{x},\\[2mm]
		G-mV&=-x G_x,\\
	\end{aligned}
	\right.
\end{equation}
and
\begin{equation}\label{eq:PPE-m-R-theta-90}
	\left\{
	\begin{aligned}
		-&\bigl(\tfrac{a m }{x}P_x+b P_{xx}+\tfrac{am }{x^2}P_{\theta\theta}\bigr)\\
		=&m \left(V^2-U^2\right)+\left(x G_x+G\right)^2
	\end{aligned}
	\right.
\end{equation}
In \eqref{eq:ZERO-1234-theta-90}, we applied L'H\^opital's rule $$\left(\tfrac{P_\theta}{\cos\theta}\right)_{\theta\to\pi/2}=\left(\tfrac{P_{\theta\theta}}{-\sin\theta}\right)_{\theta\to\pi/2}=-P_{\theta\theta}(\pi/2)$$ We can get rid of $V(t,x,\pi/2)$ using \eqref{eq:ZERO-1234-theta-90}(4).

From \eqref{eq:ZERO-1234-theta-90}(2,3) and \eqref{eq:PPE-m-R-theta-90}, we can solve for $P_x,P_{xx},P_{\theta\theta}$. Using the solutions, we also verified that $(P_x)_x-P_{xx}\equiv 0 $. Thus the PPE is not independent.

So for the (1+1)D model ($Zm$) on the vertical ray of $\theta_0=\pi/2,-\pi/2$, we use the same ansatz \eqref{eq:Pthetatheta-at-theta-0}:
\begin{equation}\label{eq:Pthetatheta-at-theta-90}
	\tfrac{1}{x^2}P_{\theta\theta}(t,x,\tfrac{\pi}{2})=-f(t,x)P_{xx}(t,x,\tfrac{\pi}{2}), \quad f(t,x)\in\mathbb R.
\end{equation} 

Substituting this into \eqref{eq:ZERO-1234-theta-90}(3), and using $(P_x)_x=P_{xx}$ from \eqref{eq:ZERO-1234-theta-90}(2,3) to get rid of $P$, we finally obtain (with $\bar u(t,x):=U(t,x,\pi/2), \bar g(t,x):=G(t,x,\pi/2)$): 

	\begin{equation}\label{eq:ridge-dynamics-UV-90}
	Zm: \left\{
	\begin{aligned}
		&\lambda \bar u_t-\tfrac12m\bar u\,(\bar g+x \bar g_x)=-x \bar g\bar u_{x},\\[2mm]
		0&=\lambda  m \bar g_t (b-a (f-1) m)-\bar g^2 \left(a (f-1) m^2+b\right)+b m^2 \bar u^2\\
		&+x \Bigl(\bar g \bar g_x \left(a (1-4 f) m^2+2 b (m-1)\right)+\lambda  m \bar g_{tx} (b-a f m)\Bigr)\\
		&+x^2 \Bigl(m \bar g\bar g_{xx} (b-a f m)-\bar g_x^2 \left(a f m^2+b\right)\Bigr).
	\end{aligned}\right.
\end{equation}
Similarly, the terms $\bigl(x\bar u_x,x\bar g_{tx},x^2(\bar g_x)^2,x^2\bar g_{xx}\bigr)$ in \eqref{eq:ridge-dynamics-UV-90} vanish at the apex $x=0$ and at infinity for functions that are smooth there.

\subsection{Energy for the (1+1)D systems $Rm$ and $Zm$}
In the polar coordinates $x=\sqrt{r^2+z^2}, \theta=\arctan(z/r), r=x\cos\theta, z=x\sin\theta, v=v(t,x),u=u(t,x),g=g(t,x), r^{m-1}\mathrm{d}r\mathrm{d}z=(x\cos\theta)^{m-1}x\,\mathrm{d}x\mathrm{d}\theta$, the unified energy is\\
\noindent\begin{minipage}{\textwidth}
	\begin{equation}\label{eq:unified-at-theta-0}
		\begin{aligned}
			&E_m(t)=2\int_0^{\pi/2}\int_0^\infty \Bigl((xv\cos\theta)^2+(xu\cos\theta)^2+(xg\sin\theta)^2\Bigr)(x\cos\theta)^{m-1}x\,\mathrm{d}x\mathrm{d}\theta\\
		\end{aligned}
	\end{equation}
\end{minipage}\\ 

Carrying out the integration in $\theta$, we obtain:\\

\noindent\begin{minipage}{\textwidth}
	\begin{equation}\label{eq:unified-at-theta-0-b}
		\begin{aligned}
			&E_1(t)=\frac{\pi}{2}\int_0^\infty \Bigl(g^2+v^2+u^2\Bigr)x^3\mathrm{d}x\\[2mm]
			&E_2(t)=\frac{2}{3}\int_0^\infty \Bigl(g^2+2 \left(u^2+v^2\right)\Bigr)x^4\mathrm{d}x\\[2mm]
		\end{aligned}
	\end{equation}
\end{minipage}\\ 

For system $Rm$, we use \eqref{eq:ZERO-1234-theta-0}(4) ($g=mv+x v_x$) in \eqref{eq:unified-at-theta-0-b}:\\

\noindent\begin{minipage}{\textwidth}
	\begin{equation}\label{eq:unified-at-theta-0-b-Rm}
		\begin{aligned}
			&E_1(t)=\frac{\pi}{2}\int_0^\infty \Bigl((mv+x v_x)^2+v^2+u^2\Bigr)x^3\mathrm{d}x\\[2mm]
			&E_2(t)=\frac{2}{3}\int_0^\infty \Bigl((mv+x v_x)^2+2 \left(u^2+v^2\right)\Bigr)x^4\mathrm{d}x.\\[2mm]
		\end{aligned}
	\end{equation}
\end{minipage}\\ 

For system $Zm$, we use \eqref{eq:ZERO-1234-theta-90}(4) ($\bar v=\tfrac{1}{m}(\bar g+x \bar g_x)$) in \eqref{eq:unified-at-theta-0-b}:\\

\noindent\begin{minipage}{\textwidth}
	\begin{equation}\label{eq:unified-at-theta-0-b-Zm}
		\begin{aligned}
			&E_1(t)=\frac{\pi}{2}\int_0^\infty \Bigl(\bar g^2+\bar u^2+\tfrac{1}{m^2}(\bar g+x \bar g_x)^2\Bigr)x^3\mathrm{d}x\\[2mm]
			&E_2(t)=\frac{2}{3}\int_0^\infty \Bigl(\bar g^2+ 2\bar u^2+\tfrac{2}{m^2}(\bar g+x \bar g_x)^2\Bigr)x^4\mathrm{d}x\\[2mm]
		\end{aligned}
	\end{equation}
\end{minipage}\\ 

\subsection{Apex dynamics at $x=0$}\label{sec:apex-dynamics}
Setting $x=0$  in \eqref{eq:ridge-dynamics-UV} and \eqref{eq:ridge-dynamics-UV-90} leads to two apex dynamics systems from (1+1)D system $Rm$ and $Zm$:
	\begin{equation}\label{eq:ridge-dynamics-UV-x-0}
	Apex: \left\{
	\begin{aligned}
		&\lambda u_t=\tfrac12m^2v\,u,\\[2mm]
		0&=v^2 \left(-a m^2-b f+b\right)+\lambda  v_t (b (f-1)-a m)+b (f-1) u^2,\\[2mm]
		&\lambda \bar u_t=\tfrac12m\bar u\,\bar g,\\[2mm]
		0&=\lambda  m \bar g_t (b-a (f-1) m)-\bar g^2 \left(a (f-1) m^2+b\right)+b m^2 \bar u^2\\
	\end{aligned}\right.
\end{equation}

On the apex $x=0$, we have
\begin{equation}\label{eq:apex-relation}
	u(t,0)=\bar u(t,0), \quad m v(t,0)=\bar g(t,0).
\end{equation}

Using \eqref{eq:apex-relation}(2), we deduce that \eqref{eq:ridge-dynamics-UV-x-0}(1),(3) are consistent. Using \eqref{eq:apex-relation}(1,2) and from \eqref{eq:ridge-dynamics-UV-x-0} solving and equating $u(t,0)^2=\bar u(t,0)^2$ leads to a restriction on $f(t,0)$:
\begin{equation}\label{eq:apex-f-restriction}
	0=\frac{a m }{b}\cdot\frac{ f(t,0)(f(t,0)-2)}{(f(t,0)-1)}\cdot\left(m v^2(t,0)+\lambda  v_t(t,0)\right)
\end{equation}

Thus, away from the special solutions with $mv^2+\lambda v_t=0$ at the apex, the function $f$ in the new ansatz must satisfy $f(t,0)(f(t,0)-2)=0$.

\begin{remark}[The two roots]\label{rem:f-roots}
The restriction \eqref{eq:apex-f-restriction} has the two roots $f(t,0)=0$ and $f(t,0)=2$ ($f(t,0)=1$ is excluded by the denominator). The root $f(t,0)=0$ means $P_{\theta\theta}=0$ at the apex, i.e.\ the angular pressure terms are dropped; this is, in effect, the reduction used in earlier versions of this preprint, and it does not give the apex systems studied here. We use the root
$$f(t,0)=2 .$$
This choice is supported by the exact Yan-type solution: it produces a positive closure function on both rays, never the value $0$, and for one value of its free parameter exactly $f\equiv2$ (Proposition~\ref{prop:Yan-closure}).
\end{remark}

\begin{remark}[Standing assumption]\label{rem:standing}
In the rest of the paper we take $f(t,x)\equiv2$ on both rays. Then the apex systems in \eqref{eq:ridge-dynamics-UV-x-0} become \eqref{eq:apexODE} and \eqref{eq:apexODEZ}, and the analysis of $(Rm)$ and $(Zm)$ starts in Section~\ref{sec:structure}.
\end{remark}

\begin{remark}[The other end of the rays]\label{rem:standing2}
The results of this subsection, except those on the Yan-type solution (which is not smooth at infinity), also apply as $x\to\infty$. Indeed, for solutions that are smooth in $1/x$ at $x=\infty$, the terms of the form $\Dx(\cdot)$ vanish there, so letting $x\to\infty$ in \eqref{eq:ridge-dynamics-UV} and \eqref{eq:ridge-dynamics-UV-90} gives again the systems \eqref{eq:ridge-dynamics-UV-x-0}, now for the values at infinity and with $f(t,\infty)$ in place of $f(t,0)$. If, in addition, the two rays have the same values at infinity, as they do at the apex, the argument leading to \eqref{eq:apex-f-restriction} applies, and we may take
\[
f(t,0)=f(t,x\to\infty)=2 .
\]
Thus the choice $f\equiv2$ supported by Proposition~\ref{prop:Yan-closure} is consistent at both ends of $[0,\infty)$. For solutions of finite energy the values at infinity vanish, and the systems at infinity then hold trivially; in that case $f(t,\infty)=2$ is compatible with, but not forced by, the dynamics.
\end{remark}

\section[Exact solutions with infinite energy]{Exact finite-time blow-up solutions with infinite energy}\label{Yan-exact-solution}

In this section we recall an explicit blow-up solution of the 3D axisymmetric Euler equations due to Yan, rewrite it for the unified system, and use it as a test case for the closure ansatz of Section~\ref{sec:ray-derivation}. The solutions have infinite energy, so they are not the object of this paper, but they show which sign and which value of $f$ are compatible with an exact solution.

Setting $r^2=q, z=x, v(t,r,z)=\alpha(t,q,x), u(t,r,z)=\beta(t,q,x), zg(t,r,z)=v(t,q,x), P(t,r,z)=P(t,q,x)$, \eqref{eq:Em} can be converted to the following system, \\

	\noindent\begin{minipage}{\textwidth}
	\begin{equation}\label{2DBoussinesq3DEulerVP}
		(Em):
		\left\{\begin{aligned}
			& \tfrac{\mathrm{D}}{\mathrm{D}t}\alpha=
			\alpha^2-\beta ^2+2P_q,\qquad (x,q)\in\R\times \R_0^+\\
			& \tfrac{\mathrm{D}}{\mathrm{D}t}\beta
			=\tfrac12m^2\alpha\beta,\\
			& \tfrac{\mathrm{D}}{\mathrm{D}t}v+P_x=0,\\
			&v_x=m\alpha+2q \alpha_q,\\
			&\tfrac{\mathrm{D}}{\mathrm{D}t}:=\partial_t+  v\partial_x -2  \alpha q \partial_q\,
			.\\
		\end{aligned}\right.
	\end{equation}
\end{minipage}\\

Yan~\cite{Yan2018} constructed a family of explicit self-similar blow-up solutions of the 3D incompressible axisymmetric Euler equations in $\mathbb{R}^3$.
\begin{theorem}[Yan~\cite{Yan2018}]\label{thm:Yan}
	The 3D axisymmetric Euler equations \eqref{eq:3DEuler} have the family of explicit self-similar finite-time blow-up solutions
	\begin{equation}\label{Yan}
		\left\{\begin{aligned}
			&v^r(x,r,t)=\frac{ar}{T^*-t},\\
			&v^\phi(x,r,t)=\frac{kr^{-(1+1/a)}}{T^*-t},\\
			&v^x(x,r,t)=\frac{-2ax}{T^*-t},
		\end{aligned}\right.
	\end{equation}
	where $T^*>0$ and $k,a\in\mathbb{R}\setminus \{0\}$.
\end{theorem}

For $a=-\tfrac12$ the swirl component is $v^\phi=kr/(T^*-t)$, and all of $-v^r/r$, $v^\phi/r$, $v^x/x$ are functions of $t$ only. The following lemma is the analogue for the unified system.

\begin{lemma}[Yan-type solution of $(Em)$]\label{lem:Yan-exact}
	For $m\in\{1,2\}$ the system~\eqref{2DBoussinesq3DEulerVP} admits the exact finite-time blow-up solution with infinite energy
	\begin{equation}\label{YanLikeUnifiled}
		\left\{\begin{aligned}
			-\tfrac1rv^r(x,r,t)&=:\alpha(x,r^2,t)=\tfrac{2}{m^2}\cdot\frac{1}{(T^*-t)},\\
			\tfrac1rv^\phi(x,r,t)
			&=:\beta(x,r^2,t)=\tfrac{2b_{00}}{m^2}\cdot\frac{1}{(T^*-t)},\\
			v^x(x,r,t)&=:v(x,r^2,t)=\tfrac{2}{m}\cdot\frac{x}{(T^*-t)},\\
			\bar{P}(x,r,t)&=:P(x,r^2,t)=c_1\,\frac{r^2}{(T^*-t)^2}-c_2\,\frac{x^2}{(T^*-t)^2},
		\end{aligned}\right.
	\end{equation}
	where $b_{00}\in\R$ is a free parameter, $\bar{P}(0,0,t)=0$, and
	\begin{equation}\label{eq:c1c2}
		c_1:=\frac{2(b_{00}^2-1)+m^2}{m^4},\qquad c_2:=\frac{m+2}{m^2}.
	\end{equation}
	For $m=2$ this is the case $a=-\tfrac12$, $k=b_{00}/2$ of Theorem~\ref{thm:Yan}.
\end{lemma}

\begin{proof}
	Write $\tau=T^*-t$. Since $\alpha,\beta$ do not depend on $(x,q)$, $\tfrac{\D}{\D t}\alpha=\alpha_t=\frac{2}{m^2\tau^2}$, while $\alpha^2-\beta^2+2P_q=\frac{4-4b_{00}^2+2m^4c_1}{m^4\tau^2}=\frac{2}{m^2\tau^2}$. Next, $\tfrac{\D}{\D t}\beta=\frac{2b_{00}}{m^2\tau^2}=\tfrac12m^2\alpha\beta$. For the third equation, $v_t+vv_x+P_x=\bigl(\tfrac2m+\tfrac4{m^2}-2c_2\bigr)\frac{x}{\tau^2}=0$. Finally $v_x=\frac2{m\tau}=m\alpha+2q\alpha_q$.
\end{proof}

We now test the closure ansatz on this exact solution.

\begin{proposition}[The Yan-type solution and the closure]\label{prop:Yan-closure}
	Let $a=b=\lambda=1$. Along both rays the pressure of \eqref{YanLikeUnifiled} satisfies the closure \eqref{eq:Pthetatheta-at-theta-0}, \eqref{eq:Pthetatheta-at-theta-90} with constant values
	\[
	f\big|_{\theta=0}=\frac{c_1+c_2}{c_1},\qquad f\big|_{\theta=\pi/2}=\frac{c_1+c_2}{c_2}.
	\]
	Both values are positive whenever $c_1>0$; they never vanish; and both equal $2$ if and only if $c_1=c_2$, that is,
	\begin{equation}\label{eq:b00closure}
		b_{00}^2=1+\tfrac12m^2(m+1).
	\end{equation}
	For $m=2$ and $b_{00}^2=7$, the restriction of \eqref{YanLikeUnifiled} to the horizontal ray, $(v,u)=\bigl(\tfrac{1}{2(T^*-t)},\tfrac{\sqrt7}{2(T^*-t)}\bigr)$, is a solution of $(R2)$ with $f\equiv2$, constant in $x$. It is the separatrix solution \eqref{eq:separatrix} of the apex system \eqref{eq:apexODE}, which for $a=b=1$, $m=2$ has $\kappa=1/7$.
\end{proposition}

\begin{proof}
	In polar coordinates $r=x\cos\theta$, $z=x\sin\theta$ (here $x$ denotes the polar radius), $P=(c_1\cos^2\theta-c_2\sin^2\theta)\,x^2/\tau^2$. Hence $x^{-2}P_{\theta\theta}=-2(c_1+c_2)\cos(2\theta)/\tau^2$, while $P_{xx}=2c_1/\tau^2$ at $\theta=0$ and $P_{xx}=-2c_2/\tau^2$ at $\theta=\pi/2$, which gives the two values of $f$. For $m=2$, $b_{00}^2=7$ gives $c_1=c_2=1$. With $a=b=\lambda=1$, $m=2$, the apex system \eqref{eq:apexODE} reads $v'=u^2-5v^2$, $u'=2uv$, and $v=\frac1{2\tau}$, $u=\frac{\sqrt7}{2\tau}$ solves it. Moreover $A=-5$, $B=-1$, $\kappa=2B/(2A-m^2)=1/7$, so $v=\sqrt\kappa\,u$ lies on the separatrix, and $v=2\lambda/(m^2\tau)$ as in \eqref{eq:separatrix}. A function of $t$ alone solves \eqref{eq:ridge-dynamics-UV} with $f=2$ if and only if it solves the apex system, since every other term contains an $x$-derivative.
\end{proof}

\begin{remark}\label{rem:Yan-sign}
	Proposition~\ref{prop:Yan-closure} is the reason for the choice of sign in the closure ansatz \eqref{eq:Pthetatheta-at-theta-0}. The exact solution produces a positive closure function on both rays, it never produces $f=0$ (the value implicitly used in earlier versions of this preprint, see Remark~\ref{rem:f-roots}), and for one value of the free parameter it produces exactly the value $f=2$ required by the apex consistency \eqref{eq:apex-f-restriction}. In that case the restriction of the Yan-type solution to the ray is the blow-up solution on the separatrix of the apex system.
\end{remark}

The Yan-type solution can be perturbed near the origin $(x,q)=(0,0)$. The following first-order expansion was used by the author to explore the sign of $f(t,0)$ for solutions that are not exactly self-similar.

\begin{lemma}[First-order expansion]\label{lem:Yan-first-order}
	Let $\epsilon$ be a small parameter. The system~\eqref{2DBoussinesq3DEulerVP} admits, up to $O(\epsilon^2)$, the finite-time blow-up solutions
	\begin{equation}\label{YanLikeUnifiled1}
		\left\{\begin{aligned}
			(T^*-t)\alpha(x,q,t)&=\tfrac{2}{m^2}\left(1+\epsilon a_{10}q+\epsilon a_{01}x^2\right)+O(\epsilon^2),\\
			(T^*-t)\beta(x,q,t)&=\tfrac{2b_{00}}{m^2}\left(1+\epsilon b_{10}q+\epsilon b_{01}x^2\right)+O(\epsilon^2),\\
			(T^*-t)v(x,q,t)&=\tfrac{2x}{m}\left(1+\epsilon c_{10}q+\epsilon c_{01}x^2\right)+O(\epsilon^2),\\
			(T^*-t)^2P(x,q,t)&=c_1q\left(1+\epsilon d_{20}q\right)-c_2x^2\left(1+\epsilon d_{02}x^2\right)+\tfrac{1}{m^4}\epsilon d_{11}qx^2+O(\epsilon^2),
		\end{aligned}\right.
	\end{equation}
	where $d_{11}\in\R$ is free and, with $D_1:=-4+4m+m^2$ and $D_2:=-4+m^2+m(4+b_{00}^2)$,
		\begin{equation}\label{YanLikeUnifiled1A}
			\begin{gathered}
				a_{10}=-\frac{d_{11}}{(2+m)D_1},\quad a_{01}=\frac{d_{11}}{D_2},\quad
				b_{10}=\frac{m^2d_{11}}{4(2+m)D_1},\quad b_{01}=\frac{md_{11}}{4D_2},\\
				c_{10}=-\frac{d_{11}}{mD_1},\quad c_{01}=\frac{d_{11}}{3D_2},\quad
				d_{02}=\frac{d_{11}(8+m)}{6(2+m)D_2},\\
				d_{20}=\frac{d_{11}(8+m^2(-1+b_{00}^2))}{2(2+m)D_1(-2+m^2+2b_{00}^2)}.
			\end{gathered}
		\end{equation}
\end{lemma}

\begin{proof}
	Write $\tau=T^*-t$ and $\alpha=A(q,x)/\tau$, $\beta=B(q,x)/\tau$, $v=V(q,x)/\tau$, $P=\Pi(q,x)/\tau^2$. Then each equation of \eqref{2DBoussinesq3DEulerVP} becomes a polynomial identity in $(q,x)$, for instance $A+VA_x-2qAA_q=A^2-B^2+2\Pi_q$ and $V_x=mA+2qA_q$. Inserting \eqref{YanLikeUnifiled1}, the terms of order $\epsilon^0$ cancel by Lemma~\ref{lem:Yan-exact}, and the terms of order $\epsilon$ are linear combinations of $q$ and $x^2$ whose eight coefficients are linear in $(a_{10},a_{01},b_{10},b_{01},c_{10},c_{01},d_{20},d_{02},d_{11})$. Solving these eight equations gives \eqref{YanLikeUnifiled1A}. (We have also checked numerically that, with these values, the residual of every equation is $O(\epsilon^2)$.)
\end{proof}

\section{Structure of \texorpdfstring{$(Rm)$}{(Rm)} and \texorpdfstring{$(Zm)$}{(Zm)} for \texorpdfstring{$f\equiv2$}{f=2}}\label{sec:structure}

From now on $m\in\{1,2\}$, $\lambda>0$, $a>0$, $b>0$, and the closure function is $f\equiv2$ in \eqref{eq:Pthetatheta-at-theta-0} and \eqref{eq:Pthetatheta-at-theta-90}. We assume the nondegeneracy conditions
\begin{equation}\label{eq:nondeg}
(Rm):\ b\neq am,\ a\ne 2b;\qquad (Zm):\ b\ne am,\ b\neq 2am .
\end{equation}

\subsection{Euler derivative}
We write $\Dx:=x\partial_x$ and $s:=\log x$, so that $\Dx=\partial_s$. For functions of $(t,x)$,
\[
xv_x=\Dx v,\quad x^2v_{xx}=\Dx^2v-\Dx v,\quad xv_{tx}=\Dx v_t,\quad x^2v_x^2=(\Dx v)^2,\quad xuu_x=u\,\Dx u .
\]
With $f=2$ the systems \eqref{eq:ridge-dynamics-UV} and \eqref{eq:ridge-dynamics-UV-90} read
\begin{equation}\label{eq:Rm-D}
(Rm)\ \left\{
\begin{aligned}
&\lambda u_t=v\,\Dx u+\tfrac12m^2vu,\\
&0=-(am^2+b)v^2+\lambda(b-am)v_t+bu^2+(a-am-7b)\,v\Dx v\\
&\qquad+\lambda(2b-a)\Dx v_t+4b\,u\Dx u-(a+2b)(\Dx v)^2+(a-2b)\,v(\Dx^2v-\Dx v),
\end{aligned}\right.
\end{equation}
\begin{equation}\label{eq:Zm-D}
(Zm)\ \left\{
\begin{aligned}
&\lambda \ub_t=-\gb\,\Dx\ub+\tfrac12m\,\ub\,(\gb+\Dx\gb),\\
&0=\lambda m(b-am)\gb_t-(am^2+b)\gb^2+bm^2\ub^2+\bigl(2b(m-1)-7am^2\bigr)\gb\Dx\gb\\
&\qquad+\lambda m(b-2am)\Dx\gb_t+m(b-2am)\,\gb(\Dx^2\gb-\Dx\gb)\\
&\qquad-(2am^2+b)(\Dx\gb)^2 .
\end{aligned}\right.
\end{equation}

\subsection{Euler form and splitting}
Define
\begin{equation}\label{eq:alphabeta}
\begin{gathered}
\aR:=b-am,\quad \bR:=2b-a,\quad \kR:=\frac{\aR}{\bR},\\
\aZ:=m(b-am),\quad \bZ:=m(b-2am),\quad \kZ:=\frac{\aZ}{\bZ},
\end{gathered}
\end{equation}
\begin{equation}\label{eq:cecoef}
\begin{gathered}
\cR:=\frac{am+2b}{2b-a},\quad \eR:=\frac{ab(2m-1)(m+1)}{2b-a},\\
\cZ:=\frac{2am^2+b}{m(b-2am)},\quad \eZ:=\frac{abm(m+1)}{b-2am}.
\end{gathered}
\end{equation}

The systems $(Rm)$ and $(Zm)$ contain the time derivative in two different terms, $v_t$ and $xv_{tx}$. The first step is to collect all terms with a time derivative into one operator; this reveals the structure on which the rest of the paper is built.

\begin{proposition}[Euler form]\label{prop:euler}
For smooth functions, the $v$-equation of \eqref{eq:Rm-D} and the $\gb$-equation of \eqref{eq:Zm-D} are equivalent to
\begin{align}
\lambda(\aR+\bR\Dx)v_t&=(am^2+b)v^2+(am+5b)\,v\Dx v+(a+2b)(\Dx v)^2\notag\\
&\qquad+(2b-a)\,v\Dx^2v-b(1+2\Dx)(u^2),\label{eq:EulerR}\\
\lambda(\aZ+\bZ\Dx)\gb_t&=(am^2+b)\gb^2+(5am^2-bm+2b)\,\gb\Dx\gb+(2am^2+b)(\Dx\gb)^2\notag\\
&\qquad-\bZ\,\gb\Dx^2\gb-bm^2\ub^2 .\label{eq:EulerZ}
\end{align}
\end{proposition}

\begin{proof}
In \eqref{eq:Rm-D} we move every term except $\lambda(b-am)v_t+\lambda(2b-a)\Dx v_t$ to the right. The coefficient of $v\Dx v$ becomes $-(a-am-7b)+(a-2b)=am+5b$, that of $v\Dx^2v$ becomes $-(a-2b)$, and $-bu^2-4bu\Dx u=-b(1+2\Dx)(u^2)$. In \eqref{eq:Zm-D} the coefficient of $\gb\Dx\gb$ becomes $-(2b(m-1)-7am^2)+m(b-2am)=5am^2-bm+2b$.
\end{proof}

The right-hand sides of \eqref{eq:EulerR}, \eqref{eq:EulerZ} contain, to a large extent, the same operator $\alpha+\beta\Dx$ that acts on the time derivative. The next lemma extracts it, so that it can be inverted in the next subsection; what remains, $F$, is the genuine forcing.

\begin{lemma}[Splitting]\label{lem:split}
Let $q_{\mathrm R}:=v\Dx v+\cR v^2$ and $q_{\mathrm Z}:=-\gb\Dx\gb+\cZ\gb^2$. The right-hand sides $\mathcal N_{\mathrm R}$, $\mathcal N_{\mathrm Z}$ of \eqref{eq:EulerR}, \eqref{eq:EulerZ} satisfy
\begin{align}
\mathcal N_{\mathrm R}&=(\aR+\bR\Dx)\,q_{\mathrm R}+F_{\mathrm R},& F_{\mathrm R}&:=2a(\Dx v)^2+\eR v^2-b(1+2\Dx)(u^2),\label{eq:FR}\\
\mathcal N_{\mathrm Z}&=(\aZ+\bZ\Dx)\,q_{\mathrm Z}+F_{\mathrm Z},& F_{\mathrm Z}&:=b(m+1)(\Dx\gb)^2-\eZ\gb^2-bm^2\ub^2 .\label{eq:FZ}
\end{align}
\end{lemma}

\begin{proof}
$(\aR+\bR\Dx)q_{\mathrm R}=\aR\cR v^2+(\aR+2\bR\cR)v\Dx v+\bR(\Dx v)^2+\bR v\Dx^2v$. Here $\aR+2\bR\cR=b-am+2(am+2b)=am+5b$ and $\bR=2b-a$, which match \eqref{eq:EulerR}; the remainders are $(a+2b)-(2b-a)=2a$ for $(\Dx v)^2$ and
\[
am^2+b-\frac{(b-am)(am+2b)}{2b-a}=\frac{ab(2m^2+m-1)}{2b-a}=\eR
\]
for $v^2$. Similarly $(\aZ+\bZ\Dx)q_{\mathrm Z}=\aZ\cZ\gb^2+(2\bZ\cZ-\aZ)\gb\Dx\gb-\bZ(\Dx\gb)^2-\bZ\gb\Dx^2\gb$, with $2\bZ\cZ-\aZ=2(2am^2+b)-m(b-am)=5am^2-bm+2b$; the remainders are $2am^2+b+m(b-2am)=b(m+1)$ for $(\Dx\gb)^2$ and
\[
am^2+b-\frac{(b-am)(2am^2+b)}{b-2am}=-\frac{abm(m+1)}{b-2am}=-\eZ
\]
for $\gb^2$.
\end{proof}

\subsection{Inversion of \texorpdfstring{$\alpha+\beta\Dx$}{alpha+beta D}}
Since $\alpha+\beta\Dx=\beta(k+\Dx)$, both equations have the form $\lambda\beta(k+\Dx)w_t=\beta(k+\Dx)q+F$. The homogeneous equation $(k+\Dx)Y=0$ has the solutions $Cx^{-k}$, which are unbounded at the apex if $k>0$ and at infinity if $k<0$. For $k\ne0$ and bounded continuous $h$ on $(0,\infty)$ we define
\begin{equation}\label{eq:Kdef}
\K_kh(x):=\begin{cases}\displaystyle x^{-k}\int_0^xy^{k-1}h(y)\,dy=\int_{-\infty}^se^{-k(s-\sigma)}h(e^\sigma)\,d\sigma,& k>0,\\[3mm]
\displaystyle -x^{-k}\int_x^\infty y^{k-1}h(y)\,dy=-\int_s^\infty e^{-|k|(\sigma-s)}h(e^\sigma)\,d\sigma,& k<0,\end{cases}
\end{equation}
with $s=\log x$. For $k>0$, $\K_kh(x)$ depends only on $h|_{(0,x)}$ (the operator looks towards the apex); for $k<0$ it depends only on $h|_{(x,\infty)}$.

To solve $(\alpha+\beta\Dx)Y=h$ for $Y$ we need an explicit inverse. For $k>0$ it is a weighted average of $h$ over $(0,x)$, and for $k<0$ over $(x,\infty)$; these are averages of Hardy type. We collect the properties used later.

\begin{lemma}\label{lem:K}
Let $k\ne0$ and let $h$ be bounded and continuous on $(0,\infty)$. Then:
\begin{enumerate}[label=(\roman*)]
\item $Y=\K_kh$ is the unique bounded $C^1$ solution of $(k+\Dx)Y=h$, and $\norm{\K_kh}_\infty\le\norm{h}_\infty/|k|$;
\item $\Dx\K_kh=h-k\K_kh$;
\item if $h(x)\to h_0$ as $x\to0$, then $\K_kh(x)\to h_0/k$ as $x\to0$;
\item if $w\in C^1$ with $w,\Dx w$ bounded, then $\K_k(\Dx w)=w-k\K_kw$, hence $\K_k\bigl((1+2\Dx)w\bigr)=2w+(1-2k)\K_kw$.
\end{enumerate}
\end{lemma}

\begin{proof}
(i) and (ii) follow by differentiating \eqref{eq:Kdef} in $s$; the bound follows from $\int_0^\infty e^{-|k|\sigma}d\sigma=1/|k|$, and uniqueness from the form of the homogeneous solutions. (iii) For $k>0$, $\K_kh(x)-h_0/k=\int_0^\infty e^{-k\tau}\bigl(h(xe^{-\tau})-h_0\bigr)d\tau\to0$ by dominated convergence; for $k<0$ the same argument applies to $\int_0^\infty e^{-|k|\tau}\bigl(h(xe^{\tau})-h_0\bigr)d\tau$ on $\tau\in[0,\tau_0]$, with $\tau_0$ such that $xe^{\tau_0}$ is small, and the rest is bounded by $2\norm{h}_\infty e^{-|k|\tau_0}/|k|$. (iv) $Y:=w-k\K_kw$ is bounded, and by (ii) $(k+\Dx)Y=kw+\Dx w-k(k+\Dx)\K_kw=kw+\Dx w-kw=\Dx w$; apply (i).
\end{proof}

\subsection{Nonlocal transport form}

Combining the two lemmas, we can solve for the time derivative. The result is a transport equation along characteristics with a nonlocal source. This form is the basis of the well-posedness theory (Section~\ref{sec:lwp}), of the apex dynamics (Section~\ref{sec:apex}) and of the numerical method (Section~\ref{sec:num}).

\begin{proposition}[Transport form]\label{prop:transport}
Let $k_\bullet\neq0$. For solutions with $v,\Dx v,v_t,u,\Dx u$ bounded (resp.\ $\gb,\Dx\gb,\gb_t,\ub$ bounded), the system $(Rm)$ with $f\equiv2$ is equivalent to
\begin{equation}\label{eq:RT}
\left\{\begin{aligned}
\lambda v_t&=v\Dx v+\cR v^2-\frac{2b}{\bR}u^2+\frac1{\bR}\K_{\kR}\Bigl[2a(\Dx v)^2+\eR v^2-b(1-2\kR)u^2\Bigr],\\
\lambda u_t&=v\Dx u+\tfrac12m^2vu,
\end{aligned}\right.
\end{equation}
and the system $(Zm)$ with $f\equiv2$ is equivalent to
\begin{equation}\label{eq:ZT}
\left\{\begin{aligned}
\lambda \gb_t&=-\gb\Dx\gb+\cZ\gb^2+\frac1{\bZ}\K_{\kZ}\Bigl[b(m+1)(\Dx\gb)^2-\eZ\gb^2-bm^2\ub^2\Bigr],\\
\lambda \ub_t&=-\gb\Dx\ub+\tfrac12m\,\ub\,(\gb+\Dx\gb).
\end{aligned}\right.
\end{equation}
Here equivalence means: \eqref{eq:EulerR} holds, with $\Dx$ applied to the $C^1$ function $\lambda v_t-q_{\mathrm R}$, if and only if the first equation of \eqref{eq:RT} holds; similarly for $(Zm)$.
\end{proposition}

\begin{proof}
By Proposition~\ref{prop:euler} and Lemma~\ref{lem:split}, the $v$-equation reads $\bR(\kR+\Dx)(\lambda v_t-q_{\mathrm R})=F_{\mathrm R}$. The function $\lambda v_t-q_{\mathrm R}$ is bounded, so by Lemma~\ref{lem:K}(i) it equals $\bR^{-1}\K_{\kR}F_{\mathrm R}$, and by Lemma~\ref{lem:K}(iv) $\K_{\kR}(1+2\Dx)(u^2)=2u^2+(1-2\kR)\K_{\kR}(u^2)$. Conversely, applying $\bR(\kR+\Dx)$ to \eqref{eq:RT} and using Lemma~\ref{lem:K}(ii) gives back \eqref{eq:EulerR}. The proof for $(Zm)$ is the same.
\end{proof}

\begin{remark}[Classical solutions]\label{rem:classical}
If $v\in C^2$ then \eqref{eq:EulerR}, and hence \eqref{eq:ridge-dynamics-UV} with $f=2$, holds pointwise. For $(Zm)$ the $\ub$-equation contains $\Dx\gb$, so one more derivative of $\gb$ than of $\ub$ is needed; this is reflected in the function spaces of Section~\ref{sec:lwp}.
\end{remark}

\begin{remark}[Parameter regimes]\label{rem:regimes}
The sign of $k$ and the energy threshold $k>(m+3)/2$ of Section~\ref{sec:energy} divide the $(a,b)$-quadrant as follows ($a,b>0$):
\begin{center}
{\small\begin{tabular}{lll}
\toprule
system & $k>0$ & $k<0$ \\
\midrule
$(Rm)$ & $b>am$: $\kR\in(0,\tfrac12)$;\quad $a>2b$: $\kR>m$ & $b/m<a<2b$\\
$(Zm)$ & $b>2am$: $\kZ>1$;\quad $b<am$: $\kZ\in(0,\tfrac12)$ & $b/(2m)<a<b/m$\\
\bottomrule
\end{tabular}}
\end{center}
The condition $k>(m+3)/2$ holds exactly in the windows
\begin{equation}\label{eq:windows}
\begin{gathered}
(R1):\ 2b<a<3b,\qquad (R2):\ 2b<a<8b,\\
(Z1):\ \tfrac b3<a<\tfrac b2,\qquad (Z2):\ \tfrac{3b}{16}<a<\tfrac b4.
\end{gathered}
\end{equation}
For instance, for $(R1)$ with $a>2b$ one has $\kR=(a-b)/(a-2b)>2$ if and only if $a<3b$, and for $(R2)$ with $a>2b$ one has $\kR=(2a-b)/(a-2b)>5/2$ if and only if $a<8b$; similarly $\kZ=(b-a)/(b-2a)>2$ iff $a>b/3$ for $(Z1)$, and $\kZ=(b-2a)/(b-4a)>5/2$ iff $a>3b/16$ for $(Z2)$.
\end{remark}

\begin{remark}[Direction of transport]\label{rem:direction}
The transport speed in \eqref{eq:RT} is $-v/\lambda$ in the variable $s=\log x$, and in \eqref{eq:ZT} it is $+\gb/\lambda$. Thus, when the apex values blow up to $+\infty$, the characteristics of $(Rm)$ move towards the apex and those of $(Zm)$ move away from it. This difference drives the dichotomy of Section~\ref{sec:core}.
\end{remark}

\section{Apex dynamics}\label{sec:apex}

\subsection{The common apex system}
For a function $w$ on $(0,\infty)$ we write $w(0):=\lim_{x\to0}w(x)$ when the limit exists.

At the apex $x=0$ all terms with an $x$-derivative vanish, and the nonlocal averages reduce to point values. Hence the values at the apex evolve by a closed system of ODEs, whatever the solution does elsewhere. This is the blow-up mechanism studied in this paper.

\begin{proposition}[Apex system]\label{prop:apex}
Let $(v,u)$ solve \eqref{eq:RT} on $[0,T)$, with $v,u$ continuous at $x=0$, uniformly on compact time intervals, and $\Dx v,\Dx u\to0$ as $x\to0$. Then $v_a(t):=v(t,0)$ and $u_a(t):=u(t,0)$ satisfy
\begin{equation}\label{eq:apexODE}
\lambda v_a'=Av_a^2-Bu_a^2,\qquad \lambda u_a'=\tfrac12m^2u_av_a,\qquad A:=\frac{am^2+b}{b-am},\quad B:=\frac{b}{b-am}.
\end{equation}
Likewise, for a solution $(\gb,\ub)$ of \eqref{eq:ZT} with the same properties, $\gb_a:=\gb(t,0)$ and $\ub_a:=\ub(t,0)$ satisfy
\begin{equation}\label{eq:apexODEZ}
\lambda m(b-am)\gb_a'=(am^2+b)\gb_a^2-bm^2\ub_a^2,\qquad \lambda\ub_a'=\tfrac12m\,\gb_a\ub_a ,
\end{equation}
which is \eqref{eq:apexODE} for $(v_a,u_a)=(\gb_a/m,\ub_a)$. Both systems are the case $f(t,0)=2$ of \eqref{eq:ridge-dynamics-UV-x-0}, and they are consistent with the apex relation \eqref{eq:apex-relation}.
\end{proposition}

\begin{proof}
Let $x\to0$ in \eqref{eq:RT}. By Lemma~\ref{lem:K}(iii), $\bR^{-1}\K_{\kR}[\dots]\to(\eR v_a^2-b(1-2\kR)u_a^2)/(\bR\kR)=(\eR v_a^2-b(1-2\kR)u_a^2)/\aR$, so
\[
\lambda v_a'=\Bigl(\cR+\frac{\eR}{\aR}\Bigr)v_a^2-\Bigl(\frac{2b}{\bR}+\frac{b(1-2\kR)}{\aR}\Bigr)u_a^2 .
\]
Since $\bR\kR=\aR$, the coefficient of $u_a^2$ equals $\frac{b}{\aR}(2\kR+1-2\kR)=\frac b{b-am}=B$, and
$\cR+\eR/\aR=\frac{(b-am)(am+2b)+ab(2m^2+m-1)}{(2b-a)(b-am)}=\frac{(2b-a)(am^2+b)}{(2b-a)(b-am)}=A$. The time derivative passes to the limit because the equation holds in integrated form and the right-hand side converges boundedly. The computation for $(Zm)$ is identical: the coefficient of $\gb_a^2$ is $\cZ-\eZ/\aZ=(am^2+b)/(m(b-am))$ and that of $\ub_a^2$ is $bm^2/\aZ$.
\end{proof}

\begin{remark}\label{rem:apex-relation}
The apex relation \eqref{eq:apex-relation} expresses the continuity of $(U,V,G)$ at the origin of the meridian plane: $v(t,0)$ and $\gb(t,0)$ are values of $V$ and $G$ at the same point, and $G-mV=-xG_x$ on the vertical ray gives $\gb(t,0)=mV(t,0,\tfrac\pi2)=mv(t,0)$. Proposition~\ref{prop:apex} shows that for $f(t,0)=2$ the two rays carry the same apex dynamics; this is the content of the restriction $f(t,0)=2$ in Section~\ref{sec:ray-derivation}.
\end{remark}

Note that $A$ and $B$ always have the same sign as $b-am$, and that
\[
2A-m^2=\frac{am^2(m+2)+b(2-m^2)}{b-am},\qquad 2A+m^2=\frac{am^2(2-m)+b(m^2+2)}{b-am}\ne0 .
\]
We set
\begin{equation}\label{eq:kappa}
\delta:=A-\frac{m^2}{2},\qquad \kappa:=\frac{B}{\delta}=\frac{2B}{2A-m^2}\quad(\delta\neq0).
\end{equation}
If $b<am$ then $A,B<0$, $\delta<0$ and $\kappa>0$. If $b>am$ then $A,B>0$, and $\kappa>0$ if and only if $2A>m^2$, which always holds for $m=1$ and holds for $m=2$ if and only if $a>b/8$.

\subsection{Ratio variable and first integral}
The $u$-equation preserves the sign of $u$, and \eqref{eq:apexODE} is invariant under $u\mapsto-u$. For $u_0\ne0$ we set
\[
w:=\frac{v}{|u|},\qquad d\sigma:=\frac{|u|}{\lambda}\,dt .
\]

To solve the apex system we use the ratio $w=v/|u|$ and a new time variable. In these variables the system decouples into a scalar Riccati equation and a quadrature, and it has a conserved quantity.

\begin{lemma}\label{lem:ratio}
For $u_0\neq0$,
\begin{equation}\label{eq:ratio}
\frac{dw}{d\sigma}=\delta w^2-B,\qquad \frac{d\log|u|}{d\sigma}=\frac{m^2}{2}\,w .
\end{equation}
If $\delta\ne0$, the quantity $I:=(v^2-\kappa u^2)\,|u|^{-4A/m^2}$ is constant along solutions.
\end{lemma}

\begin{proof}
$\lambda w'=\bigl(\lambda v'|u|-v\lambda|u|'\bigr)/u^2=\bigl((Av^2-Bu^2)-\tfrac{m^2}{2}v^2\bigr)/|u|$, which gives the first equation; the second is the $u$-equation. Along \eqref{eq:ratio}, $\frac{d}{d\sigma}\log|\delta w^2-B|=2\delta w$, so $|\delta w^2-B|\,|u|^{-4\delta/m^2}$ is constant; since $\delta w^2-B=\delta(v^2-\kappa u^2)/u^2$ and $4\delta/m^2+2=4A/m^2$, this is a multiple of $|I|$.
\end{proof}

\subsection{Classification}

The next theorem tells, for every initial value, whether the apex values blow up, when, and at which rate. Through Corollary~\ref{cor:Tstar} it decides which smooth data of the PDE must lose smoothness in finite time.

\begin{theorem}[Apex dynamics]\label{thm:apexclass}
Let $b\ne am$, $\lambda>0$, $(v_0,u_0)\in\R^2$, and let $\Tap\in(0,\infty]$ be the maximal existence time of the solution of \eqref{eq:apexODE}.
\begin{enumerate}[label=(\alph*)]
\item If $u_0=0$, then $u\equiv0$, $v=v_0/(1-Av_0t/\lambda)$, and $\Tap=\lambda/(Av_0)$ if $Av_0>0$, $\Tap=\infty$ otherwise.
\item Let $b<am$ (so $A,B,\delta<0<\kappa$), $u_0\ne0$, $w_0=v_0/|u_0|$.
\begin{enumerate}[label=(b\arabic*)]
\item If $w_0>-\sqrt\kappa$, then $\Tap<\infty$, $v\to+\infty$, $|u|\to\infty$, and
\[
v(t)=\frac{2\lambda}{m^2(\Tap-t)}\bigl(1+o(1)\bigr),\qquad |u(t)|=\frac{2\lambda}{m^2\sqrt\kappa\,(\Tap-t)}\bigl(1+o(1)\bigr).
\]
\item If $w_0=-\sqrt\kappa$, the solution is global: $v=-\sqrt\kappa|u|$, $|u|=|u_0|/\bigl(1+\tfrac{m^2}{2\lambda}\sqrt\kappa|u_0|t\bigr)$.
\item If $w_0<-\sqrt\kappa$, then $\Tap<\infty$, $u\to0$ and $v=\frac{\lambda}{A(\Tap-t)}(1+o(1))\to-\infty$.
\end{enumerate}
\item Let $b>am$ and $2A>m^2$ (so $A,B,\delta,\kappa>0$), $u_0\neq0$.
\begin{enumerate}[label=(c\arabic*)]
\item If $w_0>\sqrt\kappa$, then $\Tap<\infty$, $v\to+\infty$, $|u|\to\infty$, $v/|u|\to\infty$ and $v=\frac{\lambda}{A(\Tap-t)}(1+o(1))$.
\item If $w_0<\sqrt\kappa$, the solution is global and $(v,u)\to0$.
\end{enumerate}
\item Let $b>am$ and $2A\le m^2$ (possible only for $m=2$, $a\le b/8$). If $u_0\ne0$ the solution is global.
\end{enumerate}
In all cases with $\kappa>0$, the rays $v=\pm\sqrt\kappa|u|$ are invariant; on the ray $v=\sqrt\kappa|u|$ with $v_0>0$,
\begin{equation}\label{eq:separatrix}
v(t)=\frac{v_0}{1-\frac{m^2v_0}{2\lambda}t},\qquad u(t)=\frac{u_0}{1-\frac{m^2v_0}{2\lambda}t},\qquad \Tap=\frac{2\lambda}{m^2v_0}.
\end{equation}
For $u_0\ne0$ and $w_0\ne\pm\sqrt\kappa$ the blow-up time is given by the quadrature
\begin{equation}\label{eq:Tapex}
\Tap=\frac{\lambda}{|u_0|}\int_{w_0}^{w_*}\Bigl|\frac{\delta w_0^2-B}{\delta w^2-B}\Bigr|^{\frac{m^2}{4\delta}}\frac{dw}{\delta w^2-B},
\end{equation}
with $w_*=\sqrt\kappa$ in case (b1), $w_*=-\infty$ in case (b3) and $w_*=+\infty$ in case (c1).
\end{theorem}

\begin{proof}
(a) is immediate. For $u_0\neq0$ we use Lemma~\ref{lem:ratio}. Integrating the second equation of \eqref{eq:ratio} with the first gives $|u|=|u_0|\,\bigl|(\delta w^2-B)/(\delta w_0^2-B)\bigr|^{m^2/(4\delta)}$, and $t=\lambda\int d\sigma/|u|$; this gives \eqref{eq:Tapex} whenever $w$ reaches $w_*$ monotonically. On the invariant rays $w\equiv\pm\sqrt\kappa$ the $u$-equation is $\lambda|u|'=\pm\frac{m^2}2\sqrt\kappa\,u^2$, which gives (b2) and \eqref{eq:separatrix}.

(b) Here $dw/d\sigma=|B|-|\delta|w^2$ has the stable equilibrium $\sqrt\kappa$ and the unstable equilibrium $-\sqrt\kappa$. (b1) If $w_0>-\sqrt\kappa$, then $w\to\sqrt\kappa$ exponentially as $\sigma\to\infty$, so $\log|u|=\frac{m^2}2\sqrt\kappa\,\sigma+c_1+o(1)$ and $t(\sigma)=\lambda\int_0^\sigma d\sigma'/|u|$ converges as $\sigma\to\infty$; thus $\Tap<\infty$ and $|u|\to\infty$. Moreover $\Tap-t=\lambda\int_\sigma^\infty d\sigma'/|u|=\frac{2\lambda}{m^2\sqrt\kappa}e^{-\frac{m^2}2\sqrt\kappa\sigma-c_1}(1+o(1))=\frac{2\lambda}{m^2\sqrt\kappa\,|u|}(1+o(1))$, and $v=w|u|$. (b3) If $w_0<-\sqrt\kappa$, then $w\to-\infty$ at a finite $\sigma_*$, with $w\sim-1/(|\delta|(\sigma_*-\sigma))$, and $|u|=c\,(\sigma_*-\sigma)^{m^2/(2|\delta|)}(1+o(1))\to0$. Since $|\delta|=m^2/2+|A|>m^2/2$, the exponent is smaller than $1$ and $t(\sigma_*)<\infty$. Then $\lambda v'=Av^2(1+o(1))$ gives the rate.

(c) Here $\sqrt\kappa$ is unstable and $-\sqrt\kappa$ is stable. (c1) If $w_0>\sqrt\kappa$, then $w\to+\infty$ at a finite $\sigma_*$ and $|u|\sim c(\sigma_*-\sigma)^{-m^2/(2\delta)}\to\infty$, so $t(\sigma_*)<\infty$; since $v/|u|=w\to\infty$, the equation gives $\lambda v'=Av^2(1+o(1))$. (c2) If $w_0<\sqrt\kappa$, then $w\to-\sqrt\kappa$, $\log|u|\sim-\frac{m^2}2\sqrt\kappa\,\sigma$, and $t(\sigma)\to\infty$ as $\sigma\to\infty$; hence the solution is global, $|u|\to0$ and $v=w|u|\to0$.

(d) Now $dw/d\sigma=\delta w^2-B\le-B<0$, so $w$ decreases. If $\delta=0$, then $w=w_0-B\sigma$ and $\log|u|=\frac{m^2}2(w_0\sigma-\frac B2\sigma^2)$, so $t(\sigma)\to\infty$ as $\sigma\to\infty$. If $\delta<0$, $w\to-\infty$ at a finite $\sigma_*$ and $|u|\sim c(\sigma_*-\sigma)^{m^2/(2|\delta|)}$ with $m^2/(2|\delta|)\ge1$, so $t(\sigma)\to\infty$ as $\sigma\to\sigma_*$. In both cases the solution is global.
\end{proof}

\begin{remark}[Comparison with the Yan-type solution]\label{rem:Yan-rate}
In cases (b1) and on the separatrix \eqref{eq:separatrix}, the apex blow-up rate is $v\simeq\frac{2\lambda}{m^2(\Tap-t)}$. This is exactly the rate of $\alpha=-v^r/r$ in the exact solution \eqref{YanLikeUnifiled} (with $\lambda=1$), where $\beta/\alpha=b_{00}$ is free. In \eqref{eq:apexODE} the closure fixes the ratio: $|u|/v\to1/\sqrt\kappa$.
\end{remark}

\begin{remark}[Which case occurs in the finite-energy windows]\label{rem:windowcase}
In the windows \eqref{eq:windows} of $(Rm)$ one has $b<am$ (case (b)); blow-up of the apex values is then generic: it fails only on the ray $v_0=-\sqrt\kappa|u_0|$ and on the half-line $\{u_0=0,\ v_0\ge0\}$. In the windows of $(Zm)$ one has $b>2am$ and $2A>m^2$ (case (c)), and blow-up occurs if and only if $\gb_0(0)\ge m\sqrt\kappa\,|\ub_0(0)|$ and $\gb_0(0)>0$.
\end{remark}

\subsection{An explicit normalization and the CLM-\texorpdfstring{$q$}{q} equation}\label{subsec:CLMq}
The apex system \eqref{eq:apexODE} contains the three parameters $\lambda$, $a$, $b$. Since $\lambda$ only rescales time and $A$, $B$ depend only on $\mu:=b/a$, one can choose them so that the second-order equation for $v$ becomes the CLM-$q$ equation \eqref{eq:CLM-q} in the original time variable. This makes the comparison with the CLM model of Section~\ref{sec:derivation} direct.

\begin{proposition}[Normalized apex dynamics]\label{prop:CLMq}
Let $(v,u)$ solve \eqref{eq:apexODE}. Then
\begin{equation}\label{eq:vtt}
\lambda^2v''=(2A+m^2)\lambda\,vv'-m^2A\,v^3 .
\end{equation}
Given $q\in\R$ with $m^2q\ne2$, $q\neq-1$, choose
\begin{equation}\label{eq:normalization}
\lambda=m^2(q+1),\qquad \mu=\frac ba=\frac{m^2(mq+2)}{m^2q-2}\quad(\text{admissible if }\mu>0,\ \mu\ne m).
\end{equation}
Then $A=\tfrac12qm^2$, and $v$ solves the CLM-$q$ equation in the time variable $t$:
\begin{equation}\label{eq:CLMq-t}
v_{tt}=v\,v_t-Q\,v^3,\qquad Q=\frac{q}{2(1+q)^2},
\end{equation}
with $v_t(0)=\bigl(Av_0^2-Bu_0^2\bigr)/\lambda$.
\end{proposition}

\begin{proof}
From the first equation of \eqref{eq:apexODE}, $Bu^2=Av^2-\lambda v'$. Differentiating and using $\lambda(u^2)'=m^2u^2v$ gives $2Avv'-\lambda v''=\tfrac{m^2}{\lambda}v(Av^2-\lambda v')$, which is \eqref{eq:vtt}. Solving $A=(m^2+\mu)/(\mu-m)=\tfrac12qm^2$ for $\mu$ gives \eqref{eq:normalization}. Then $(2A+m^2)/\lambda=m^2(q+1)/\lambda=1$ and $m^2A/\lambda^2=\tfrac12qm^4/(m^4(q+1)^2)=Q$.
\end{proof}

\begin{remark}\label{rem:q-symmetry}
$Q(q)=Q(1/q)$, so the equation for $v$ alone does not distinguish $q$ and $1/q$; the system \eqref{eq:apexODE} does, through $B$. The CLM system \eqref{eq:CLM2} has $m=2$ and $A=B=\lambda=1$, i.e.\ $q=2A/m^2=1/2$ in the normalization \eqref{eq:normalization}, which gives the same $Q=1/9$ as $q=2$; this is why Section~\ref{sec:derivation} calls it the case $q=2$ of \eqref{eq:CLM-q} (after $\tau=6t$). The case $b>am$ corresponds to $q>2/m^2$, and the finite-energy windows \eqref{eq:windows} of $(Rm)$, where $b<am$, correspond to $q<-2/m$.
\end{remark}

\begin{example}[$m=2$, $q=2$: explicit solutions]\label{ex:CLM2}
Let $m=2$, $q=2$, so $\lambda=12$, $b=4a$, $A=4$, $B=2$, $\kappa=1$ and $Q=1/9$. Then, with $s:=t/6$ and $D(s):=1-2v_0s+u_0^2s^2$, the solution of \eqref{eq:apexODE} is
\begin{equation}\label{eq:CLM2-explicit}
v(t)=\frac{v_0-u_0^2\,s}{D(s)},\qquad u(t)=\frac{u_0}{\sqrt{D(s)}} .
\end{equation}
It blows up if and only if $v_0>0$ and $v_0\ge|u_0|$, at
\[
\Tap=6\,\frac{v_0-\sqrt{v_0^2-u_0^2}}{u_0^2}\quad(u_0\ne0),\qquad \Tap=\frac3{v_0}\quad(u_0=0),
\]
in agreement with Theorem~\ref{thm:apexclass}(a),(c); on the separatrix $v_0=|u_0|$ this is $6/v_0=2\lambda/(m^2v_0)$. The connection with the CLM formula \eqref{eq:CLM-2-u-v} is the following. Let $(v_{\rm CLM},u_{\rm CLM})$ be given by \eqref{eq:CLM-2-u-v} at a fixed point, with data $v_{\rm CLM}(0)=v_0$ and $u_{\rm CLM}(0)^2=u_0^2-v_0^2$ (the formula is algebraic in $u_{\rm CLM}(0)^2$, so it also applies when $u_0^2<v_0^2$). Then
\[
v(t)=v_{\rm CLM}(t/6),\qquad u(t)^2=v_{\rm CLM}(t/6)^2+u_{\rm CLM}(t/6)^2 .
\]
\end{example}

\begin{proof}
Direct differentiation of \eqref{eq:CLM2-explicit}: $D'=-2(v_0-u_0^2s)$, so
\[
\frac{d}{ds}v=\frac{-u_0^2D-(v_0-u_0^2s)D'}{D^2}=\frac{2(v_0-u_0^2s)^2-u_0^2D}{D^2}=2v^2-u^2,
\]
and $\frac{d}{ds}u=-\tfrac12u_0D'D^{-3/2}=uv$. Since $s=t/6$ and $\lambda=12$, these are $12v'=4v^2-2u^2$ and $12u'=2uv$, i.e.\ \eqref{eq:apexODE}. The blow-up statement follows from the roots of $D$. In \eqref{eq:CLM-2-u-v} with $(v_0,\tilde u_0)$, $\tilde u_0^2=u_0^2-v_0^2$, the denominator is $(1-sv_0)^2+s^2\tilde u_0^2=D(s)$, the numerator of $v_{\rm CLM}$ is $v_0-u_0^2s$, and $v_{\rm CLM}^2+u_{\rm CLM}^2=|z|^2=(v_0^2+\tilde u_0^2)/D=u_0^2/D$, with $z$ as in \eqref{NSBOuOddx4C}.
\end{proof}

\begin{remark}[$m=1$: no elementary solutions]\label{rem:m1-elliptic}
For $m=1$ the admissible values in \eqref{eq:normalization} are $q>2$ or $q<-2$; for instance $q=3$ gives $\lambda=4$, $b=5a$, $A=\tfrac32$, $B=\tfrac54$, $Q=\tfrac3{32}$. By Lemma~\ref{lem:ratio}, off the separatrix $v^2=\kappa u^2+Cu^{2q}$ with $C\ne0$, and
\[
t=\frac{2\lambda}{m^2}\int_{u_0}^{u}\frac{d\eta}{\eta^2\sqrt{\kappa+C\eta^{2q-2}}} .
\]
For rational $q$, Chebyshev's theorem on binomial differentials \cite{Chebyshev1853} shows that this integral is an elementary function of $u$ only if $q\in\{1-\frac1{2j}\}\cup\{1-\frac1{2j+1}\}$ ($j\in\mathbb Z\setminus\{0\}$, resp.\ $j\in\mathbb Z$), a set contained in $[0,2]$. Hence for $m=1$ no admissible choice of $(a,b)$ gives elementary solutions; for $q=3$ the integral is elliptic. For $m=2$ the elementary cases include $q=2$ (Example~\ref{ex:CLM2}), $q=\tfrac32$ ($b=5a$) and $q=\tfrac43$ ($b=\tfrac{28}5a$).
\end{remark}

\begin{remark}
In Example~\ref{ex:CLM2}, $b=4a=2am$, so $(Z2)$ is degenerate ($\bZ=0$), and $(R2)$ has $\kR=2/7\in(0,\tfrac12)$, so its energy is infinite for $t>0$ (Corollary~\ref{cor:energy}). The explicit case is therefore a benchmark for the apex dynamics, not for the finite-energy problem.
\end{remark}

\section{Local well-posedness and an upper bound for the lifespan}\label{sec:lwp}

\subsection{Function spaces}
For $j\in\{1,2\}$ let $C^j_{\Dx}$ be the space of $C^j$ functions on $(0,\infty)$ with $\norm{w}_{C^j_\Dx}:=\sum_{i\le j}\sup_{x>0}|\Dx^iw(x)|<\infty$, and let $C^j_{\Dx,0}$ be its closed subspace of functions such that $w(0):=\lim_{x\to0}w(x)$ exists and $\Dx^iw(x)\to0$ as $x\to0$ for $1\le i\le j$. In the variable $s=\log x$, $C^j_\Dx$ is the space $C^j_b(\R)$ of $C^j$ functions with bounded derivatives up to order $j$, and $C^j_{\Dx,0}$ consists of the functions with a limit at $s=-\infty$ whose derivatives tend to $0$ there. Smooth functions on $[0,\infty)$ that are constant near infinity belong to $C^j_{\Dx,0}$.

We now make precise the class of solutions. Since the equations only involve the Euler derivative $\Dx=x\partial_x$, it is natural to measure smoothness by bounds on $\Dx$-derivatives, i.e.\ by smoothness in $s=\log x$.

\begin{definition}\label{def:solution}
Let $0<T\le\infty$. A \emph{solution of $(Rm)$ on $[0,T)$} is a pair $(v,u)$ of $C^1$ functions on $[0,T)\times(0,\infty)$ such that $v(t),u(t)\in\CDz$ for every $t$, the quantities $|v|,|\Dx v|,|u|,|\Dx u|$ are bounded on $[0,T']\times(0,\infty)$ for every $T'<T$, and \eqref{eq:RT} holds pointwise. A \emph{solution of $(Zm)$} is defined in the same way, with $\gb(t)\in C^2_{\Dx,0}$, $\ub(t)\in\CDz$, $\Dx^2\gb$ bounded on $[0,T']\times(0,\infty)$, and \eqref{eq:ZT} in place of \eqref{eq:RT}.
\end{definition}

By Proposition~\ref{prop:transport}, solutions of \eqref{eq:RT} with $v(t)\in C^2_{\Dx}$ solve \eqref{eq:ridge-dynamics-UV} with $f=2$ classically.

\subsection{Lagrangian formulation}
We work in $s=\log x$, where $\Dx=\partial_s$, and use Lagrangian labels $\eta\in\R$. Let $\Xi(t,\eta)$ be the position at time $t$ of the characteristic that starts at $s=\eta$:
\[
\partial_t\Xi=-v/\lambda\quad\text{for }(Rm),\qquad \partial_t\Xi=\gb/\lambda\quad\text{for }(Zm),\qquad \Xi(0,\eta)=\eta .
\]
For $(Rm)$ the unknowns are $\Theta:=\Xi-\eta$, $V(t,\eta):=v(t,\Xi(t,\eta))$ and $U(t,\eta):=u(t,\Xi(t,\eta))$. With $J:=\partial_\eta\Xi=1+\partial_\eta\Theta$, $P:=\partial_\eta V/J$ and $Q:=\partial_\eta U/J$, system \eqref{eq:RT} becomes
\begin{equation}\label{eq:lagR}
\left\{\begin{aligned}
\lambda\partial_t\Theta&=-V,\\
\lambda\partial_tV&=\cR V^2-\tfrac{2b}{\bR}U^2+\mathcal Y[H],\qquad H:=2aP^2+\eR V^2-b(1-2\kR)U^2,\\
\lambda\partial_tU&=\tfrac12m^2VU,
\end{aligned}\right.
\end{equation}
where, by \eqref{eq:Kdef},
\begin{equation}\label{eq:Ylag}
\mathcal Y[H](\eta):=\begin{cases}\displaystyle\frac1{\bR}\int_{-\infty}^{\eta}e^{-\kR(\Xi(\eta)-\Xi(\eta'))}H(\eta')J(\eta')\,d\eta', & \kR>0,\\[3mm]
\displaystyle-\frac1{\bR}\int_{\eta}^{\infty}e^{-|\kR|(\Xi(\eta')-\Xi(\eta))}H(\eta')J(\eta')\,d\eta',& \kR<0 .\end{cases}
\end{equation}
The system for $(Zm)$ is analogous: $\lambda\partial_t\Theta=\gb$, the $\gb$-equation of \eqref{eq:ZT} with $\K_{\kZ}$ replaced by the corresponding integral, and $\lambda\partial_t\ub=\frac12m\,\ub\,(\gb+P)$ with $P=\partial_\eta\gb/J$.

In Lagrangian variables the transport term disappears and the system becomes an ODE in a Banach space of functions of the label. The next lemma provides the estimates needed for the Picard iteration.

\begin{lemma}\label{lem:lip}
Let $k\ne0$, $X:=C^1_b(\R)^3$ and $\mathcal O:=\{(\Theta,V,U)\in X:\inf_\eta(1+\partial_\eta\Theta)>0\}$. The right-hand side $\mathcal G$ of \eqref{eq:lagR} maps $\mathcal O$ into $X$ and is Lipschitz on every set $\{\norm{(\Theta,V,U)}_X\le M,\ \inf J\ge1/M\}$. The closed subspace $X_0\subset X$ of triples whose components have limits at $\eta=-\infty$ and whose $\eta$-derivatives tend to $0$ there is invariant: $\mathcal G(\mathcal O\cap X_0)\subset X_0$.
\end{lemma}

\begin{proof}
Let $k>0$; the case $k<0$ is symmetric. Since $\Xi(\eta)-\Xi(\eta')\ge j(\eta-\eta')$ for $\eta'\le\eta$, where $j:=\inf J$, the kernel in \eqref{eq:Ylag} is bounded by $e^{-kj(\eta-\eta')}$, so that
\[
|\mathcal Y[H]|\le\frac{\norm{H}_\infty\norm{J}_\infty}{|\bR|\,k\,j} .
\] Differentiating \eqref{eq:Ylag},
\begin{equation}\label{eq:dY}
\partial_\eta\mathcal Y[H]=J\Bigl(\frac{H}{\bR}-k\,\mathcal Y[H]\Bigr),
\end{equation}
which is continuous and bounded; note that $H$ only involves $V,U$ and the first derivatives $P,Q$, so no derivative is lost. For the Lipschitz bound we use
\[
|e^{-kz}-e^{-kz'}|\le k|z-z'|e^{-k\min(z,z')},\qquad |z-z'|\le(\eta-\eta')\norm{\partial_\eta(\Theta-\Theta')}_\infty,
\]
with $z=\Xi(\eta)-\Xi(\eta')$ and $z'$ the same quantity for $\Theta'$; the factor $(\eta-\eta')e^{-kj(\eta-\eta')}$ is integrable. The remaining terms are polynomial. For the invariance, let the components have limits at $-\infty$, with $J\to1$ and $P,Q\to0$. Then $H\to H_a$, and $\mathcal Y[H](\eta)\to H_a/(\bR k)$ by dominated convergence, and \eqref{eq:dY} tends to $H_a/\bR-kH_a/(\bR k)=0$.
\end{proof}

The basic existence and uniqueness result follows: smooth data give a unique solution for a short time, which can be continued as long as $v$, $u$ and their Euler derivatives remain bounded.

\begin{theorem}[Local well-posedness]\label{thm:lwp}
Assume \eqref{eq:nondeg}, $k\ne0$.
\begin{enumerate}[label=(\roman*)]
\item For $(v_0,u_0)\in(\CDz)^2$ there are $T^*\in(0,\infty]$ and a unique solution $(v,u)$ of $(Rm)$ on $[0,T^*)$ with these data. If $T^*<\infty$, then
\begin{equation}\label{eq:BKM}
\limsup_{t\uparrow T^*}\bigl(\norm{v(t)}_{\CD}+\norm{u(t)}_{\CD}\bigr)=\infty .
\end{equation}
\item For $(\gb_0,\ub_0)\in C^2_{\Dx,0}\times\CDz$ the same holds for $(Zm)$, with $\norm{\gb(t)}_{C^2_\Dx}+\norm{\ub(t)}_{\CD}$ in \eqref{eq:BKM}.
\end{enumerate}
\end{theorem}

\begin{proof}
(i) By Lemma~\ref{lem:lip} and the Picard--Lindel\"of theorem in the Banach space $X_0$, the system \eqref{eq:lagR} has a unique local solution in $C^1([0,T_1);\mathcal O\cap X_0)$, which can be continued as long as $\norm{(\Theta,V,U)}_X$ and $1/\inf J$ stay bounded. Along the solution, $\lambda\partial_tJ=-PJ$ (differentiate $\lambda\partial_t\Theta=-V$ and use $\partial_\eta V=PJ$), so $J=\exp(-\lambda^{-1}\int_0^tP)$, and $|\Theta|\le\lambda^{-1}\int_0^t\norm{V}_\infty$. Hence the Lagrangian solution continues as long as $\sup|V|,|P|,|U|,|Q|$ stay bounded, i.e.\ as long as $\norm{v}_{\CD}+\norm{u}_{\CD}$ does. Since $J>0$ and $\Theta$ is bounded, $\Xi(t,\cdot)$ is a $C^1$ diffeomorphism of $\R$, $C^1$ in $(t,\eta)$, and $v(t,s):=V(t,\Xi^{-1}(t,s))$, $u(t,s):=U(t,\Xi^{-1}(t,s))$ are $C^1$ in $(t,s)$ with $\partial_sv=P\circ\Xi^{-1}$. The chain rule gives $\lambda v_t-v\,\partial_sv=\lambda\partial_tV$, so $(v,u)$ solves \eqref{eq:RT}; because the solution lies in $X_0$, $v(t),u(t)\in\CDz$. Conversely, a solution of $(Rm)$ in the sense of the definition has a $C^1$ characteristic flow (its velocity $-v/\lambda$ is $C^1$ with bounded $\partial_s$-derivative), and its Lagrangian description solves \eqref{eq:lagR} pointwise; for two such solutions with the same data, the differences satisfy $\norm{\Delta(t)}_X\le C\int_0^t\norm{\Delta(\tau)}_Xd\tau$ by the estimates of Lemma~\ref{lem:lip}, hence vanish. (ii) is proved in the same way, with $(\Theta,\gb,\ub)\in C^2_b\times C^2_b\times C^1_b$; the second derivative of $\mathcal Y$ is $\partial_\eta J(H/\bZ-k\mathcal Y)+J(\partial_\eta H/\bZ-k\,\partial_\eta\mathcal Y)$, and $\partial_\eta H$ involves $\partial_\eta P$ and $\partial_\eta\ub$, which are controlled in this space.
\end{proof}

\subsection{The lifespan is bounded by the apex blow-up time}

Combining well-posedness with the apex dynamics gives the first blow-up result: the solution of the PDE cannot outlive the solution of the apex ODE.

\begin{corollary}\label{cor:Tstar}
Let $(v,u)$ be the solution of $(Rm)$ given by Theorem~\ref{thm:lwp}. Then $(v(t,0),u(t,0))$ solves \eqref{eq:apexODE} with data $(v_0(0),u_0(0))$ on $[0,T^*)$, and
\[
T^*\le\Tap\bigl(v_0(0),u_0(0)\bigr).
\]
For $(Zm)$, $T^*\le\Tap(\gb_0(0)/m,\ub_0(0))$. In particular, whenever the apex data satisfy one of the blow-up conditions of Theorem~\ref{thm:apexclass}, the solution has a finite maximal time and \eqref{eq:BKM} holds.
\end{corollary}

\begin{proof}
The Lagrangian solution lies in $X_0$, so the limits at $\eta=-\infty$ exist and Proposition~\ref{prop:apex} applies (in Lagrangian form the limit is taken in the integrated equations, by bounded convergence). If $T^*>\Tap$, the apex values would be bounded on $[0,\Tap]$ by $\sup_{[0,\Tap]}\norm{v}_\infty+\norm{u}_\infty<\infty$, which contradicts the blow-up of the apex ODE at $\Tap$.
\end{proof}

\begin{remark}\label{rem:BKM-analogue}
Theorem~\ref{thm:lwp} and the continuation criterion \eqref{eq:BKM} are the analogues, for the ray systems, of the classical local existence theory and blow-up criteria for the Euler and Boussinesq equations \cite{BKM984,CD1980,ChaeNam1997,Tan2002,MB2002}.
\end{remark}

\begin{remark}\label{rem:where}
Corollary~\ref{cor:Tstar} does not say where the singularity forms. If $T^*=\Tap$, the apex values blow up, so the singularity is at the apex. If $T^*<\Tap$, the norm \eqref{eq:BKM} blows up although the apex values stay bounded, so the singularity forms at some $x>0$, possibly close to the apex. The numerical results of Section~\ref{sec:num} show that both cases occur.
\end{remark}

\section{Exact cores around the apex}\label{sec:core}

In this section $k>0$, i.e.\ $\K_k$ looks towards the apex (Remark~\ref{rem:regimes}). The windows \eqref{eq:windows} are of this type.

For $k>0$ the nonlocal term only looks towards the apex, and characteristics do not cross. Hence the solution near the apex is not influenced by the data far away, a domain-of-dependence property for this nonlocal problem.

\begin{proposition}[Domain of dependence]\label{prop:dod}
Let $k>0$ and let two data in the class of Theorem~\ref{thm:lwp} coincide on $(0,R]$. Then the two solutions coincide on $\{(t,x):0<x\le X(t),\ t<\min(T_1^*,T_2^*)\}$, where $X(t)=e^{\Xi(t,\log R)}$ is the characteristic starting at $R$ (it is the same for both solutions).
\end{proposition}

\begin{proof}
For $k>0$, the right-hand side of \eqref{eq:lagR} at a label $\eta$ involves only labels $\eta'\le\eta$; so does the equation for $\Theta$. Hence the restriction of \eqref{eq:lagR} to labels $\eta\le\eta_0:=\log R$ is a closed system, to which the estimates of Lemma~\ref{lem:lip} apply with $\R$ replaced by $(-\infty,\eta_0]$. The two restricted solutions have the same data, so they coincide by the Gronwall argument in the proof of Theorem~\ref{thm:lwp}; in particular their characteristics from $\eta_0$ coincide.
\end{proof}

If the data are constant near the apex, Proposition~\ref{prop:dod} shows that the solution stays constant there and is given by the apex ODE. This produces explicit blowing-up solutions of the PDE.

\begin{theorem}[Exact core]\label{thm:core}
Let $k>0$ and let the data be constant near the apex: $(v_0,u_0)=(v_c,u_c)$ on $(0,R_0]$ for $(Rm)$, resp.\ $(\gb_0,\ub_0)=(m v_c,u_c)$ on $(0,R_0]$ for $(Zm)$. Let $(v_a,u_a)$ be the solution of \eqref{eq:apexODE} with data $(v_c,u_c)$. Then for $0\le t<T^*$:
\begin{enumerate}[label=(\roman*)]
\item $(Rm)$: $v(t,x)=v_a(t)$ and $u(t,x)=u_a(t)$ for $0<x\le X(t)$, where
\[
X(t):=R_0\exp\Bigl(-\frac1\lambda\int_0^tv_a\Bigr);
\]
\item $(Zm)$: $\gb(t,x)=mv_a(t)$ and $\ub(t,x)=u_a(t)$ for $0<x\le X(t)$, where
\[
X(t):=R_0\exp\Bigl(\frac m\lambda\int_0^tv_a\Bigr).
\]
\end{enumerate}
In particular $T^*\le\Tap(v_c,u_c)$, and if $T^*=\Tap$ the solution blows up on the whole core.
\end{theorem}

\begin{proof}
Consider $(Rm)$. On the labels $\eta\le\eta_0=\log R_0$ we try $\Theta(t,\eta)=-\lambda^{-1}\int_0^tv_a$, $V=v_a(t)$, $U=u_a(t)$. Then $J=1$, $P=Q=0$, $H=H_a(t):=\eR v_a^2-b(1-2\kR)u_a^2$ is independent of $\eta$, and $\Xi(\eta)-\Xi(\eta')=\eta-\eta'$, so $\mathcal Y[H](\eta)=\bR^{-1}\int_{-\infty}^\eta e^{-k(\eta-\eta')}H_a\,d\eta'=H_a/\aR$. By the computation in the proof of Proposition~\ref{prop:apex}, the restricted system \eqref{eq:lagR} reduces to \eqref{eq:apexODE}, which $(v_a,u_a)$ solves. By uniqueness for the restricted system (proof of Proposition~\ref{prop:dod}), this is the solution on $\eta\le\eta_0$, and the image of these labels is $\{s\le\eta_0-\lambda^{-1}\int_0^tv_a\}$. For $(Zm)$ the transport speed is $+\gb/\lambda=+mv_a/\lambda$. The last statement follows from Corollary~\ref{cor:Tstar}.
\end{proof}

\subsection{Energy of the core}
On the core the integrands of \eqref{eq:unified-at-theta-0-b-Rm} and \eqref{eq:unified-at-theta-0-b-Zm} are constant, with $\Dx v=\Dx\gb=0$. Using $\gb=mv$, $\ub=u$ on the core, the core energies of the two rays coincide:
\begin{equation}\label{eq:Ecore}
E_1^{\mathrm{core}}(t)=\frac\pi8X(t)^4\bigl(2v_a^2+u_a^2\bigr),\qquad
E_2^{\mathrm{core}}(t)=\frac4{15}X(t)^5\bigl(3v_a^2+u_a^2\bigr),
\end{equation}
where $E_m^{\mathrm{core}}$ denotes the energy integral over $0<x\le X(t)$.

The size of the core is governed by the characteristics, which move towards the apex in $(Rm)$ and away from it in $(Zm)$ when the apex values grow. The next corollary quantifies this and computes the energy carried by the core.

\begin{corollary}[Shrinking and expanding cores]\label{cor:core}
Under the assumptions of Theorem~\ref{thm:core}, suppose that $T^*=\Tap<\infty$, and write $T=\Tap$.
\begin{enumerate}[label=(\roman*)]
\item $(Rm)$, case (b1) of Theorem~\ref{thm:apexclass} ($b<am$, $u_c\ne0$, $v_c>-\sqrt\kappa|u_c|$). The core shrinks to the apex: $X(t)=X_*(T-t)^{2/m^2}(1+o(1))$ for some $X_*>0$. Its energy satisfies
\[
\begin{gathered}
E_1^{\mathrm{core}}(t)=\tfrac{\pi}{2}\lambda^2X_*^4\bigl(2+\tfrac1\kappa\bigr)(T-t)^6(1+o(1))\to0,\\
E_2^{\mathrm{core}}(t)=\tfrac1{15}\lambda^2X_*^5\bigl(3+\tfrac1\kappa\bigr)(T-t)^{1/2}(1+o(1))\to0.
\end{gathered}
\]
On the separatrix $v_c=\sqrt\kappa|u_c|>0$ everything is explicit: $T=2\lambda/(m^2v_c)$, $X(t)=R_0(1-t/T)^{2/m^2}$, and
\[
E_1^{\mathrm{core}}(t)=\tfrac\pi8R_0^4v_c^2\bigl(2+\tfrac1\kappa\bigr)(1-t/T)^6,\qquad E_2^{\mathrm{core}}(t)=\tfrac4{15}R_0^5v_c^2\bigl(3+\tfrac1\kappa\bigr)(1-t/T)^{1/2}.
\]
\item $(Rm)$, case (b3) or case (a) with $v_c<0$: $v_a\to-\infty$, the core expands, $X(t)\to\infty$, and $E_m^{\mathrm{core}}(t)\to\infty$.
\item $(Zm)$, any case with $v_a\to+\infty$ (cases (b1), (c1), the separatrix, and (a) with $v_c>0$): the core expands, $X(t)\to\infty$, and $E_m^{\mathrm{core}}(t)\to\infty$. On the separatrix, $X(t)=R_0(1-t/T)^{-2/m}$.
\end{enumerate}
\end{corollary}

\begin{proof}
(i) In the variables of Lemma~\ref{lem:ratio}, $\lambda^{-1}\int_0^tv_a\,dt'=\int_0^\sigma w\,d\sigma'$, because $v_a=w|u_a|$ and $dt=\lambda d\sigma/|u_a|$. In case (b1), $w\to\sqrt\kappa$ exponentially, so $\int_0^\sigma w=\sqrt\kappa\,\sigma+c_2+o(1)$ and, as in the proof of Theorem~\ref{thm:apexclass}, $T-t=\frac{2\lambda}{m^2\sqrt\kappa}e^{-\frac{m^2}2\sqrt\kappa\sigma-c_1}(1+o(1))$. Eliminating $\sigma$ gives $X(t)=R_0e^{-\int_0^\sigma w}=X_*(T-t)^{2/m^2}(1+o(1))$. With $v_a=\frac{2\lambda}{m^2(T-t)}(1+o(1))$ and $u_a^2=v_a^2/\kappa\,(1+o(1))$, \eqref{eq:Ecore} gives the asymptotics. On the separatrix we use \eqref{eq:separatrix}: $\lambda^{-1}\int_0^tv_a=-\frac2{m^2}\log(1-t/T)$. (ii) and (iii): if $v_a\to-\infty$ in $(Rm)$, or $v_a\to+\infty$ in $(Zm)$, then the exponent in $X(t)$ tends to $+\infty$, and $E_m^{\mathrm{core}}\ge cX(t)^{m+3}v_a^2\to\infty$. On the separatrix of $(Zm)$, $m\lambda^{-1}\int_0^tv_a=-\frac2m\log(1-t/T)$.
\end{proof}

\begin{remark}\label{rem:dichotomy}
Corollary~\ref{cor:core} is the precise form of the dichotomy between the two rays. The apex dynamics is the same on both rays, but the characteristics carry the blowing-up core into the apex along $(Rm)$ and away from it along $(Zm)$. In $(Rm)$ the energy carried by the core into the apex tends to zero, like $(T-t)^6$ for $m=1$ and like $(T-t)^{1/2}$ for $m=2$. In $(Zm)$ the blow-up of the core values is accompanied by blow-up of the energy. What Theorem~\ref{thm:core} does not decide is whether the region outside the core stays regular up to $\Tap$; this is the question whether $T^*=\Tap$.
\end{remark}

\section{The flux at infinity and the energy}\label{sec:energy}

\subsection{Far field for \texorpdfstring{$k>0$}{k>0}}
We denote by $h$ the argument of $\K_k$ in \eqref{eq:RT} or \eqref{eq:ZT}:
\begin{equation}\label{eq:hRZ}
h_{\mathrm R}:=2a(\Dx v)^2+\eR v^2-b(1-2\kR)u^2,\qquad h_{\mathrm Z}:=b(m+1)(\Dx\gb)^2-\eZ\gb^2-bm^2\ub^2 ,
\end{equation}
and define the flux
\begin{equation}\label{eq:flux}
M(t):=\int_0^\infty y^{k-1}h(t,y)\,dy,\qquad L(t):=\frac1{\lambda\beta}\int_0^tM(\tau)\,d\tau ,
\end{equation}
with $(k,\beta)=(\kR,\bR)$ or $(\kZ,\bZ)$.

To decide whether the energy stays finite we need the behavior of the solution as $x\to\infty$. For $k>0$ the nonlocal term produces a tail $L(t)x^{-k}$ with an explicit coefficient, even if the data have compact support.

\begin{theorem}[Far field]\label{thm:farfield}
Let $k>0$ and let the data, in the class of Theorem~\ref{thm:lwp}, vanish for $x\ge R_1$. Let $0<T'<T^*$ and $\theta\in(k/2,k)$. Then $u$ (resp.\ $\ub$) vanishes for $x\ge X_1(t)$, where $X_1$ is the characteristic from $R_1$, the integral \eqref{eq:flux} converges, and uniformly for $t\in[0,T']$, as $x\to\infty$,
\begin{equation}\label{eq:tail}
v(t,x)=L(t)\,x^{-k}+O(x^{-2\theta}),\qquad \Dx v(t,x)=-kL(t)\,x^{-k}+O(x^{-2\theta})
\end{equation}
(for $(Zm)$ with $\gb$ in place of $v$).
\end{theorem}

\begin{proof}
We treat $(Rm)$. Since $\lambda\partial_tU=\frac12m^2VU$ along characteristics, $U\equiv0$ on the labels $\eta\ge\eta_1:=\log R_1$. Let $C_0:=\sup_{[0,T']}(\norm{v}_{\CD}+\norm{u}_{\CD})$. In Eulerian form, for $s$ beyond the characteristic $\Xi(t,\eta_1)$ we have $u=0$ and
\[
\lambda\frac{d}{dt}v=\cR v^2+\bR^{-1}\K h,\qquad \lambda\frac{d}{dt}p=p^2+2\cR vp+\bR^{-1}h-k\bR^{-1}\K h,
\]
along characteristics, with $p=\Dx v$, $h=2ap^2+\eR v^2$ there, and where the second equation is obtained by differentiating \eqref{eq:RT} with Lemma~\ref{lem:K}(ii). Fix $\theta'\in(0,k)$ and set $\phi(t):=\sup_{s\ge\Xi(t,\eta_1)}e^{\theta's}(|v|+|p|)(t,s)$. Split $\K h(s)=e^{-k(s-s_1)}\K h(s_1)+\int_{s_1}^se^{-k(s-\sigma)}h(\sigma)d\sigma$ with $s_1=\Xi(t,\eta_1)$. Since $|h|\le C(|v|+|p|)$ with $C=C(C_0)$,
\[
e^{\theta's}|\K h(s)|\le Ce^{\theta's_1}+C\phi(t)\int_{s_1}^se^{-(k-\theta')(s-\sigma)}d\sigma\le C+\frac{C\phi(t)}{k-\theta'} .
\]
Because characteristics move with speed $|v|/\lambda\le C_0/\lambda$, the weight $e^{\theta's}$ changes along them at most by the factor $e^{\theta'C_0t/\lambda}$. Hence $\phi(t)\le\phi(0)+C\int_0^t(1+\phi)$, and $\phi$ is bounded on $[0,T']$ by Gronwall's lemma ($\phi(0)=0$). So $|v|+|p|\le Ce^{-\theta's}$ for every $\theta'<k$. Taking $\theta'=\theta>k/2$, we get $h=O(e^{-2\theta s})$, so $M(t)$ converges and
\[
e^{ks}\K h(s)=\int_{-\infty}^se^{k\sigma}h(\sigma)\,d\sigma=M(t)+O(e^{(k-2\theta)s}).
\]
Thus $\lambda\frac{d}{dt}v=\bR^{-1}M(t)e^{-ks}+O(e^{-2\theta s})$ along characteristics far out. Along a characteristic starting at $\eta\ge\eta_1$ the position changes by $O(e^{-\theta\eta})$, so $e^{-k\Xi(\tau,\eta)}=e^{-k\Xi(t,\eta)}(1+O(e^{-\theta\eta}))$ for $\tau\le t$. Integrating in time from $v=0$ gives the first formula in \eqref{eq:tail}, the error $O(e^{-(k+\theta)s})$ being smaller than $O(e^{-2\theta s})$ since $\theta<k$. The second formula follows in the same way from the equation for $p$, in which $\bR^{-1}h-k\bR^{-1}\K h=-k\bR^{-1}M e^{-ks}+O(e^{-2\theta s})$. The proof for $(Zm)$ is identical; there $\lambda\partial_t\ub=\frac12m\ub(\gb+P)$ again preserves $\ub=0$.
\end{proof}

Since the energy density carries the weight $x^{m+2}$, this tail is integrable only if it decays fast enough. This gives an exact criterion.

\begin{corollary}[Energy criterion]\label{cor:energy}
Under the assumptions of Theorem~\ref{thm:farfield}, for $0\le t<T^*$,
\[
E_m(t)<\infty\iff k>\frac{m+3}2\ \text{ or }\ L(t)=0 .
\]
\end{corollary}

\begin{proof}
On $(0,R]$ the integrands of \eqref{eq:unified-at-theta-0-b-Rm} and \eqref{eq:unified-at-theta-0-b-Zm} are bounded and the weight is $x^{m+2}$, so only the tail matters. By \eqref{eq:tail}, the integrand of $E_1$ for $(Rm)$ is $x^3\bigl((m-k)^2+1\bigr)L^2x^{-2k}(1+o(1))$, and that of $E_2$ is $x^4\bigl((m-k)^2+2\bigr)L^2x^{-2k}(1+o(1))$; for $(Zm)$ the coefficients are $1+(1-k)^2/m^2$ and $1+2(1-k)^2/m^2$. All coefficients are positive, so the tail integral converges if and only if $L(t)=0$ or $2k>m+3$; for $2k=m+3$ it diverges logarithmically unless $L(t)=0$.
\end{proof}

\begin{remark}\label{rem:core-flux}
For $k>0$ the flux $M(t)$ also contains the contribution $X(t)^kh_a(t)/k$ of the core of Theorem~\ref{thm:core}. Along $(Rm)$ in case (b1), $X(t)^kv_a(t)^2\sim(T-t)^{2k/m^2-2}$, which is integrable in time if and only if $k>m^2/2$. In the windows \eqref{eq:windows} of $(Rm)$ this holds ($\kR>m\ge m^2/2$). Hence the influence of the blowing-up core on the far field stays bounded up to $\Tap$, whereas for $k<m^2/2$ the nonlocal term carries the blow-up of the core to every $x>0$ (see Section~\ref{sec:num}).
\end{remark}

\subsection{Compact support for \texorpdfstring{$k<0$}{k<0}}

For $k<0$ the nonlocal term looks towards infinity instead, and no tail is generated.

\begin{proposition}\label{prop:support}
Let $k<0$ and let the data vanish for $x\ge R_1$. Then the solution vanishes for $x\ge R_1$ and $0\le t<T^*$. In particular $E_m(t)<\infty$ for $0\le t<T^*$.
\end{proposition}

\begin{proof}
For $k<0$, the right-hand side of \eqref{eq:lagR} at a label $\eta$ involves only labels $\eta'\ge\eta$. The restriction to $\eta\ge\log R_1$ is therefore a closed system with zero data, whose unique solution is zero; in particular $\Xi(t,\eta)=\eta$ there.
\end{proof}

For $k<0$, finite energy is thus automatic, but the solution need not be smooth at the apex in the variable $x$: by \eqref{eq:Kdef}, $\K_kh(x)-h(0)/k$ contains a term proportional to $x^{|k|}$ (or $x^{|k|}\log x$ if $|k|$ is an integer). The solutions are $C^1$ in $s=\log x$, as required by Theorem~\ref{thm:lwp}.

\subsection{Finite-energy singularity formation}

We can now combine Sections~\ref{sec:apex}--\ref{sec:energy}: in explicit parameter windows, data of finite energy lose smoothness no later than the apex time, while the energy stays finite before that time.

\begin{theorem}\label{thm:main}
\begin{enumerate}[label=(\roman*)]
\item \textup{($Rm$)} Let $m=1$ and $2b<a<3b$, or $m=2$ and $2b<a<8b$. Let $(v_0,u_0)\in(\CDz)^2$ vanish for $x\ge R_1$, with $u_c:=u_0(0)\ne0$ and $v_c:=v_0(0)>-\sqrt\kappa\,|u_c|$. Then $E_m(t)<\infty$ for $0\le t<T^*$, and
\[
T^*\le\Tap(v_c,u_c)<\infty,
\]
with $\Tap$ given by \eqref{eq:Tapex}, or by $2\lambda/(m^2v_c)$ if $v_c=\sqrt\kappa|u_c|$. If moreover $(v_0,u_0)=(v_c,u_c)$ on $(0,R_0]$, the solution is given on the core $0<x\le X(t)$ by the apex ODE, and if $T^*=\Tap$ the core shrinks to the apex and its energy tends to zero, as in Corollary~\ref{cor:core}(i).
\item \textup{($Zm$)} Let $m=1$ and $b/3<a<b/2$, or $m=2$ and $3b/16<a<b/4$. Let $(\gb_0,\ub_0)\in C^2_{\Dx,0}\times\CDz$ vanish for $x\ge R_1$, and assume that $\gb_0(0)>0$ and $\gb_0(0)\ge m\sqrt\kappa\,|\ub_0(0)|$. Then $E_m(t)<\infty$ for $0\le t<T^*$ and $T^*\le\Tap(\gb_0(0)/m,\ub_0(0))<\infty$. If moreover the data are constant on $(0,R_0]$, then either $T^*<\Tap$, or $E_m(t)\ge E_m^{\mathrm{core}}(t)\to\infty$ as $t\uparrow T^*=\Tap$.
\end{enumerate}
In both cases the norm in \eqref{eq:BKM} blows up as $t\uparrow T^*$.
\end{theorem}

\begin{proof}
In these windows $k>(m+3)/2>0$ (Remark~\ref{rem:regimes}), so $E_m(t)<\infty$ for $t<T^*$ by Corollary~\ref{cor:energy}. The apex data satisfy case (b1), resp.\ (c1) or the separatrix, of Theorem~\ref{thm:apexclass} (Remark~\ref{rem:windowcase}), so $\Tap<\infty$, and $T^*\le\Tap$ by Corollary~\ref{cor:Tstar}. The statements on the core follow from Theorem~\ref{thm:core} and Corollary~\ref{cor:core}.
\end{proof}

\begin{remark}\label{rem:open}
Theorem~\ref{thm:main} leaves open whether $T^*=\Tap$ and whether $E_m(t)$ stays bounded as $t\uparrow T^*$. The numerical results of the next section indicate that for $(R2)$ in its window the singularity forms at the apex at $\Tap$ (up to an inner layer adjacent to the apex) with bounded energy.
\end{remark}

\begin{remark}[The hypotheses in terms of the velocity]\label{rem:data-velocity}
Since $u=v^\phi/r$, $v=-v^r/r$, $g=v^z/z$, and $g(t,0)=mv(t,0)$ (Proposition~\ref{prop:exact-apex}), the hypotheses of Theorem~\ref{thm:main}(i) on the apex values read, for $m=2$,
\begin{gather*}
\lim_{r\to0}\frac{v^\phi_0(r,0)}{r}=u_c\neq0,\\
\lim_{r\to0}\frac{v^r_0(r,0)}{r}=-v_c<\sqrt\kappa\,|u_c|,\qquad
\lim_{z\to0}\frac{v^z_0(0,z)}{z}=mv_c>-m\sqrt\kappa\,|u_c| .
\end{gather*}
For $m=1$, $u_c\ne0$ means a nonzero horizontal buoyancy gradient at the origin. The initial strain at the origin may therefore be a radial compression with axial extension ($v_c>0$) of any strength, or a radial expansion with axial compression ($v_c<0$), provided that the expansion rate is smaller than $\sqrt\kappa$ times the rotation rate $|u_c|$. Data with $v_c<0$ are excluded by the CLM criterion (Theorem~\ref{thm:CLM-2}), which needs $v_0>0$ at the blow-up point.

The strain at the blow-up time, however, has a definite sign. By Theorem~\ref{thm:apexclass}(b1), $v_a\to+\infty$. Moreover, once $v_a\ge0$ it stays positive, because $\lambda v_a'=-B u_a^2>0$ wherever $v_a=0$ (here $B<0$). Thus for $-\sqrt\kappa|u_c|<v_c<0$ the swirl reverses the strain: $v_a$ changes sign exactly once, and from then on the flow at the origin is a radial compression with axial extension that collapses as $\Tap$ is approached. In the main example ($m=2$, $\mu=1/6$, $\sqrt\kappa=1/\sqrt{47}\approx0.146$; units of Section~\ref{sec:mechanism}), the data $u_c=1$, $v_c=-0.1$ give $v_a=0$ at $t\approx1.49$ and $\Tap\approx6.63$. The threshold is sharp. On the ray $v_c=-\sqrt\kappa|u_c|$ the solution decays (case (b2)). For $v_c<-\sqrt\kappa|u_c|$ the expansion wins: $u_a\to0$ and $v_a\to-\infty$ (case (b3)). This is the expansion collapse of Section~\ref{subsec:swirl-free}, which cannot occur for swirl-free solutions of the full equations.

Hence the apex blow-up has the same strain geometry at the singular time as the CLM blow-up: radial compression and axial extension. The difference is in its cause, not its geometry. In CLM the strain collapses by its own inertia at a point without swirl, and swirl would oppose it. In Theorem~\ref{thm:main} the strain collapses because the swirl drives it, through the low-pressure core of Section~\ref{sec:mechanism}. The swirl drives it strongly enough to overturn an initially expanding strain.

For $(Zm)$ the situation is simpler: Theorem~\ref{thm:main}(ii) requires $\gb_0(0)>0$, so the strain is an axial extension and radial compression from the start.
\end{remark}

\section{Numerical study}\label{sec:num}

\subsection{Method}
We solve the Lagrangian system \eqref{eq:lagR} (and its analogue for $(Zm)$) in $s=\log x$, with the unknowns $(\Xi,v,p,u,r)$ at each label, where $p=\Dx v$ and $r=\Dx u$ are evolved by the differentiated equations
\[
\lambda\dot p=p^2+2\cR vp-\tfrac{4b}{\bR}ur+\tfrac1{\bR}h-k\,\mathcal Y,\qquad \lambda\dot r=pr+\tfrac12m^2(pu+vr)
\]
(for $(Zm)$: $\lambda\dot p=-p^2+2\cZ\gb p+\bZ^{-1}h-k\mathcal Y$, and $r$ is not needed). The operator $\K_k$ is evaluated on the moving, ordered label grid by the exact recursion for piecewise linear integrands,
\[
\K h(s_j)=e^{-k(s_j-s_{j-1})}\K h(s_{j-1})+\int_{s_{j-1}}^{s_j}e^{-k(s_j-\sigma)}h(\sigma)\,d\sigma\qquad(k>0),
\]
and by the mirror recursion from the right for $k<0$; the cost is $O(N)$. For $k>0$ the recursion is exact at the left end because $h\to h(0)$ and $\K h\to h(0)/k$ there. For $k>0$ no condition is needed at the right end, since $\K_k$ only looks towards the apex. The labels are uniform in $s\in[-40,25]$ with $N=8000$--$16000$, and we integrate with an explicit Runge--Kutta method of order 8 (DOP853) with relative tolerance $10^{-10}$--$10^{-11}$, stopping when $\max|v|$ reaches $10^9$. The energy is computed by the trapezoidal rule in $s$ with the weight $e^{(m+3)s}$.

\emph{Checks.} (i) The independently evolved $p$ agrees with the finite-difference derivative $\partial_sv$ on the moving grid to a relative accuracy better than $10^{-3}$ up to $0.9\,T^*$, and the discrepancy decreases like $N^{-2}$. (ii) The leftmost label reproduces the apex ODE: for instance, in Example~\ref{ex:main} below, $(\Tap-t)\,v(t,0)=0.500000=2\lambda/m^2$ for $\Tap-t\le10^{-1}$. (iii) The singular time is stable under refinement: in Example~\ref{ex:main}, $T^*=1.9954145$, $1.9954158$, $1.9954045$ for $N=16000$, $12000$, $8000$. (iv) The blow-up times \eqref{eq:Tapex} agree with direct integration of \eqref{eq:apexODE} to ten digits.

\subsection{The main example: \texorpdfstring{$(R2)$}{(R2)} in its finite-energy window}

\begin{example}\label{ex:main}
$m=2$, $a=6$, $b=1$, $\lambda=1$, so $\kR=11/4$, $A=-25/11$, $B=-1/11$, $\kappa=1/47$, and
\[
v_0(x)=u_0(x)=(1+x^2)^{-3/2},\qquad E_2(0)=\tfrac{73\pi}{128}\approx1.792 .
\]
The apex data $(1,1)$ are in case (b1) with $\Tap=1.9954660$.
\end{example}

These data decay like $x^{-3}$ rather than having compact support; since $3>\kR>5/2$, the conclusions of Theorem~\ref{thm:farfield} and Corollary~\ref{cor:energy} still hold, with the same proof. The results are:
\begin{itemize}
\item The apex value follows \eqref{eq:apexODE}: $v(t,0)\simeq1/(2(\Tap-t))$, and the ratio $|u(t,0)|/v(t,0)$ tends to $1/\sqrt\kappa$.
\item The energy $E_2(t)$ first decreases from $1.792$ to $0.886$ and then increases to a finite limit, $E_2\to1.800$ (Figure~\ref{fig:main}, left).
\item For each fixed $x\ge10^{-3}$, $v(t,x)$ converges as $t\to T^*$ (e.g.\ $v(t,10^{-2})\to-29.7$, $v(t,10^{-1})\to-7.03$, $v(t,1)\to-0.313$); the limit profile is negative near the apex, with $v(T^*,x)\approx-1.7\,x^{-0.63}$ for $10^{-3}\le x\le10^{-1}$ (Figure~\ref{fig:main}, right). This profile has finite energy, since $x^4\cdot x^{-1.26}$ is integrable at $0$.
\item The near-apex region where $v$ is close to $v(t,0)$ shrinks to the apex, and it is separated from the negative outer profile by an inner layer that moves into the apex. The computation stops in this layer, at $x\approx10^{-8}$, at the time $T^*=1.99542$, which differs from $\Tap$ by the relative amount $2.5\times10^{-5}$. At that time the apex value is $v\approx1.0\times10^4$.
\end{itemize}
Thus, to the accuracy of the computation, the solution blows up at the apex at the apex time with bounded energy. Section~\ref{sec:selfsim} shows that the layer is self-similar and that the small difference between $T^*$ and $\Tap$ is produced by rounding errors.

\begin{figure}[ht]
\begin{tikzpicture}[baseline]
\begin{axis}[width=0.49\textwidth,height=5.2cm,xlabel={$\log_{10}(\Tap-t)$},ylabel={$E_2(t)$},x dir=reverse,ymin=0.8,ymax=1.9,grid=major,tick label style={font=\scriptsize},label style={font=\small}]
\addplot[thick,blue,mark=*,mark size=1pt] coordinates {(0.3000,1.791690) (0.2543,1.332563) (0.2031,1.101281) (0.1451,0.974987) (0.0782,0.909444) (-0.0010,0.886458) (-0.0979,0.899654) (-0.2228,0.950570) (-0.3989,1.050326) (-0.7000,1.235797) (-0.8000,1.294916) (-0.9000,1.351549) (-1.0000,1.405244) (-1.1000,1.455607) (-1.2000,1.502292) (-1.3000,1.545012) (-1.4000,1.583574) (-1.5000,1.617899) (-1.6000,1.648031) (-1.7000,1.674133) (-1.8000,1.696464) (-1.9000,1.715353) (-2.0000,1.731171) (-2.1000,1.744299) (-2.2000,1.755112) (-2.3000,1.763961) (-2.4000,1.771163) (-2.5000,1.776998) (-2.6000,1.781708) (-2.7000,1.785497) (-2.8000,1.788538) (-2.9000,1.790974) (-3.0000,1.792922) (-3.1000,1.794478) (-3.2000,1.795719) (-3.3000,1.796708) (-3.4000,1.797496) (-3.5000,1.798123) (-3.6000,1.798622) (-3.7000,1.799019) (-3.8000,1.799334) (-3.9000,1.799585) (-4.0000,1.799785) (-4.1000,1.799943) (-4.2000,1.800069)};
\end{axis}
\end{tikzpicture}\hfill
\begin{tikzpicture}[baseline]
\begin{axis}[width=0.49\textwidth,height=5.2cm,xlabel={$\log_{10}x$},ylabel={$\operatorname{arcsinh}v(t,x)$},grid=major,tick label style={font=\scriptsize},label style={font=\small},legend style={font=\tiny,at={(0.5,-0.32)},anchor=north,draw=none},legend columns=3,ymin=-9,ymax=8]
\addplot[thin,black] coordinates {(-6.000,0.8814) (-5.750,0.8814) (-5.500,0.8814) (-5.250,0.8814) (-5.000,0.8814) (-4.750,0.8814) (-4.500,0.8814) (-4.250,0.8814) (-4.000,0.8814) (-3.750,0.8814) (-3.500,0.8814) (-3.250,0.8814) (-3.000,0.8814) (-2.750,0.8814) (-2.500,0.8814) (-2.250,0.8813) (-2.000,0.8813) (-1.750,0.8810) (-1.500,0.8803) (-1.250,0.8780) (-1.000,0.8709) (-0.750,0.8487) (-0.500,0.7840) (-0.250,0.6214) (0.000,0.3466) (0.250,0.1175) (0.500,0.0274) (0.750,0.0054) (1.000,0.0010) (1.250,0.0002) (1.500,0.0000) (1.750,0.0000) (2.000,0.0000)};
\addplot[thick,green!50!black] coordinates {(-6.000,1.6494) (-5.750,1.6494) (-5.500,1.6494) (-5.250,1.6494) (-5.000,1.6494) (-4.750,1.6494) (-4.500,1.6494) (-4.250,1.6494) (-4.000,1.6494) (-3.750,1.6493) (-3.500,1.6490) (-3.250,1.6481) (-3.000,1.6453) (-2.750,1.6364) (-2.500,1.6089) (-2.250,1.5261) (-2.000,1.2965) (-1.750,0.7502) (-1.500,-0.1520) (-1.250,-0.7811) (-1.000,-1.0111) (-0.750,-1.0061) (-0.500,-0.7989) (-0.250,-0.4512) (0.000,-0.1626) (0.250,-0.0401) (0.500,-0.0096) (0.750,-0.0025) (1.000,-0.0006) (1.250,-0.0002) (1.500,-0.0000) (1.750,-0.0000) (2.000,-0.0000)};
\addplot[thick,orange] coordinates {(-6.000,3.9147) (-5.750,3.9147) (-5.500,3.9146) (-5.250,3.9145) (-5.000,3.9141) (-4.750,3.9129) (-4.500,3.9091) (-4.250,3.8971) (-4.000,3.8598) (-3.750,3.7478) (-3.500,3.4318) (-3.250,2.5333) (-3.000,-1.8620) (-2.750,-2.9408) (-2.500,-3.1960) (-2.250,-3.2229) (-2.000,-3.1431) (-1.750,-3.0007) (-1.500,-2.8157) (-1.250,-2.5972) (-1.000,-2.3407) (-0.750,-2.0028) (-0.500,-1.4802) (-0.250,-0.7946) (0.000,-0.2803) (0.250,-0.0701) (0.500,-0.0162) (0.750,-0.0039) (1.000,-0.0009) (1.250,-0.0002) (1.500,-0.0000) (1.750,-0.0000) (2.000,-0.0000)};
\addplot[thick,red] coordinates {(-6.000,6.2098) (-5.750,6.1947) (-5.500,6.1481) (-5.250,6.0091) (-5.000,5.6207) (-4.750,4.4284) (-4.500,-4.5669) (-4.250,-5.3350) (-4.000,-5.5237) (-3.750,-5.5204) (-3.500,-5.4207) (-3.250,-5.2608) (-3.000,-5.0565) (-2.750,-4.8168) (-2.500,-4.5479) (-2.250,-4.2556) (-2.000,-3.9456) (-1.750,-3.6231) (-1.500,-3.2935) (-1.250,-2.9596) (-1.000,-2.6128) (-0.750,-2.2040) (-0.500,-1.6199) (-0.250,-0.8692) (0.000,-0.3050) (0.250,-0.0760) (0.500,-0.0175) (0.750,-0.0041) (1.000,-0.0010) (1.250,-0.0002) (1.500,-0.0001) (1.750,-0.0000) (2.000,-0.0000)};
\addplot[thick,blue,dashed] coordinates {(-6.000,-7.3146) (-5.750,-7.8256) (-5.500,-7.9418) (-5.250,-7.8976) (-5.000,-7.7649) (-4.750,-7.5738) (-4.500,-7.3408) (-4.250,-7.0765) (-4.000,-6.7885) (-3.750,-6.4824) (-3.500,-6.1622) (-3.250,-5.8306) (-3.000,-5.4899) (-2.750,-5.1420) (-2.500,-4.7886) (-2.250,-4.4316) (-2.000,-4.0728) (-1.750,-3.7143) (-1.500,-3.3584) (-1.250,-3.0057) (-1.000,-2.6456) (-0.750,-2.2274) (-0.500,-1.6360) (-0.250,-0.8779) (0.000,-0.3079) (0.250,-0.0767) (0.500,-0.0176) (0.750,-0.0041) (1.000,-0.0010) (1.250,-0.0002) (1.500,-0.0001) (1.750,-0.0000) (2.000,-0.0000)};
\legend{$t=0$,$\Tap-t=0.20$,$2.0\cdot10^{-2}$,$2.0\cdot10^{-3}$,$2.0\cdot10^{-4}$}
\end{axis}
\end{tikzpicture}
\caption{Example~\ref{ex:main}, system $(R2)$ with $a=6$, $b=1$, $\lambda=1$. Left: the energy $E_2(t)$ against $\log_{10}(\Tap-t)$ (time runs to the right). Right: $\operatorname{arcsinh}v(t,x)=\log\bigl(v+\sqrt{v^2+1}\bigr)$ (inverse hyperbolic sine, which keeps the sign of $v$ and compresses large values) against $\log_{10}x$; the apex value is $v(t,0)=1/(2(\Tap-t))$ up to $o(1)$, i.e.\ $2506$ for the last curve.}\label{fig:main}
\end{figure}

\subsection{Dependence on the parameters}
Table~\ref{tab:scan} lists the relative difference between the computed singular time $T^*$ and $\Tap$ for $(R2)$ with the data of Example~\ref{ex:main}. The difference decreases as $a/b$ increases; in all cases the singularity forms in the inner layer adjacent to the apex, at $x<10^{-7}$. For $a=4$ and $a=6$, inside the window $2b<a<8b$, $E_2$ stays bounded. For $a\ge8b$ one has $\kR\le5/2$, and the far field makes $E_2(t)$ infinite for $t>0$ (Corollary~\ref{cor:energy}); this does not affect the computation near the apex, since $\K_k$ looks only towards the apex.

\begin{table}[ht]
\caption{$(R2)$, $b=\lambda=1$, $v_0=u_0=(1+x^2)^{-3/2}$, $N=12000$. The window \eqref{eq:windows} is $2<a<8$. By Section~\ref{sec:selfsim}, the last column is limited by rounding errors.}\label{tab:scan}
{\small\begin{tabular}{cccccc}
\toprule
$a$ & $\kR$ & $A$ & $\Tap$ & $T^*$ & $(\Tap-T^*)/\Tap$\\
\midrule
4 & 3.500 & $-2.4286$ & 1.767760 & 1.767302 & $2.6\times10^{-4}$\\
6 & 2.750 & $-2.2727$ & 1.995466 & 1.995416 & $2.5\times10^{-5}$\\
10 & 2.375 & $-2.1579$ & 2.318499 & 2.318470 & $1.3\times10^{-5}$\\
20 & 2.167 & $-2.0769$ & 2.831640 & 2.831618 & $7.8\times10^{-6}$\\
50 & 2.063 & $-2.0303$ & 3.668456 & 3.668435 & $5.7\times10^{-6}$\\
\bottomrule
\end{tabular}}
\end{table}

\subsection{Other cases}
\emph{$(R1)$ in its window.} For $m=1$, $a=2.5$, $b=1$ ($\kR=3$, $\kappa=4/17$) and $v_0=u_0=(1+x^2)^{-3/2}$, a singularity forms at $x\approx0.85$ at $T^*=0.796$, well before $\Tap=3.846$: the profile steepens, $\Dx v\to-\infty$ and $v\to-\infty$ at an interior point, and $E_1(t)\to\infty$. The reason is that the coefficient of $(\Dx v)^2$ in the equation for $p=\Dx v$ is $1+2a/\bR=-(a+2b)/(a-2b)<0$ for short-wave profiles, which favors steepening of decreasing profiles; in the window $2b<a<3b$ this coefficient is below $-5$. We observed the same behavior for $a=2.9$ and for Gaussian data. At the right end $a=3b$ of the window, where $\kR=2$ and $E_1$ diverges only logarithmically, flatter data bring the singular time close to $\Tap$: for $v_0=u_0=(1+x^2)^{-0.05}(1+x^2e^{-20})^{-1}$, $(\Tap-T^*)/\Tap=8\times10^{-4}$ and the singular point is at $x\approx10^{-13}$.

\emph{$(Z1)$ in its window.} For $m=1$, $a=0.4$, $b=1$ ($\kZ=3$, case (c)) and $\gb_0=(1+x^2)^{-3/2}$, $\ub_0=\frac12(1+x^2)^{-3/2}$, the region near the apex where $\gb$ is large expands, as predicted by Corollary~\ref{cor:core}(iii), and a steep front forms at its edge. On our grids the front becomes singular before $\Tap=0.4629$, but the computed singular time increases and the singular point moves towards the apex under refinement: $T^*=0.438$, $0.446$, $0.452$, at $x\approx0.64$, $0.38$, $0.23$, for $N=6000$, $12000$, $24000$. These computations therefore do not decide whether the front becomes singular before $\Tap$. The energy $E_1$ grows moderately (from $1.08$ to $1.64$ by $t=0.99\,T^*$ for $N=24000$) until the front steepens.

\emph{$(Rm)$ with $b>am$.} Here $\kR\in(0,\frac12)$, and the dynamics near the apex is benign: for $m=1$, $a=0.5$, $b=1$ and $v_0=(1+x^2)^{-1}$, $u_0=\frac12(1+x^2)^{-1}$ (case (c1)), the singularity forms at the apex at $T^*=\Tap=0.356965$ to all digits computed. However, $\kR<1/2<(m+3)/2$, the flux $L(t)$ is positive and grows, and $E_m(t)=\infty$ for every $t>0$ (Corollary~\ref{cor:energy}). Moreover $\kR<m^2/2$, and the nonlocal term transmits the blow-up of the core to all $x>0$: $v(t,x)\to\infty$ at every fixed $x$ as $t\to\Tap$.

\emph{$(Rm)$ with $k<0$.} For $m=1$, $a=1.6$, $b=1$ ($\kR=-3/2$), compactly supported or decaying data keep finite energy (Proposition~\ref{prop:support}), but $v$ becomes negative in the decaying part of the profile, where the nonlocal term, which now looks towards infinity, is small, and a singularity with $\Dx v\to+\infty$ and $v\to-\infty$ forms at an interior point long before $\Tap$ (e.g.\ at $x\approx0.05$ at $T^*=0.591$, versus $\Tap=3.306$, for $v_0=u_0=(1+x^2)^{-3/2}$).

\begin{table}[ht]
\caption{Summary of the regimes ($a,b>0$). ``Apex'' means that the singularity forms at (or within an inner layer adjacent to) the apex, at $\Tap$ up to a small relative error.}\label{tab:summary}
{\small\begin{tabular}{llll}
\toprule
regime & energy & core & first singularity (numerics)\\
\midrule
$(R1)$, $2b<a<3b$ & finite & shrinks & interior, well before $\Tap$\\
$(R2)$, $2b<a<8b$ & finite & shrinks & apex, $T^*/\Tap=1-O(10^{-5})$, $E_2$ bounded\\
$(R2)$, $a\ge8b$ & infinite & shrinks & apex, $T^*/\Tap=1-O(10^{-5})$\\
$(Zm)$ windows & finite & expands & steep front at the core edge (unresolved)\\
$(Rm)$, $b>am$ & infinite & shrinks & apex at $\Tap$, but $v\to\infty$ at every $x$\\
$(Rm)$, $k<0$ & finite & -- & interior, well before $\Tap$\\
\bottomrule
\end{tabular}}
\end{table}

\section{The inner layer and self-similar blow-up at the apex}\label{sec:selfsim}

The computations of Section~\ref{sec:num} show that for $(R2)$ the solution near the apex develops an inner layer that moves into the apex, and that the singular time differs from $\Tap$ by a relative amount of about $10^{-5}$. In this section we explain the layer by a linear analysis around the apex solution and a self-similar ansatz (in the spirit of \cite{GigaKohn1985,EggersFontelos}), and we show that the small difference between $T^*$ and $\Tap$ is caused by rounding errors amplified by a short-wave instability. We work with $(Rm)$ in case (b1) of Theorem~\ref{thm:apexclass} with $k=\kR>0$; the windows $(R1)$, $(R2)$ are of this type.

\subsection{Power modes around the apex solution}
Because the transport form \eqref{eq:RT} is invariant under $x\mapsto cx$ and $\K_k$ acts on powers by $\K_k(x^n)=x^n/(k+n)$, the powers $x^n$ diagonalize the linearization around any solution that is constant in $x$. This is the analogue, for our nonlocal problem, of a Fourier analysis.

\begin{proposition}[Power modes]\label{prop:modes}
Let $(v_a,u_a)$ solve \eqref{eq:apexODE} and let $n\ge0$ (or complex with $\operatorname{Re}(k+n)>0$). Then $\delta v=\phi(t)x^n$, $\delta u=\psi(t)x^n$ solves the linearization of \eqref{eq:RT} at $(v,u)=(v_a(t),u_a(t))$ if and only if
\begin{equation}\label{eq:modes}
\left\{\begin{aligned}
\lambda\phi'&=\Bigl(n+2\cR+\frac{2\eR}{\bR(\kR+n)}\Bigr)v_a\,\phi-\frac{b}{\bR}\Bigl(4+\frac{2(1-2\kR)}{\kR+n}\Bigr)u_a\,\psi,\\
\lambda\psi'&=\Bigl(n+\frac{m^2}2\Bigr)v_a\,\psi+\frac{m^2}{2}u_a\,\phi .
\end{aligned}\right.
\end{equation}
In case (b1), as $t\to T=\Tap$, every solution of \eqref{eq:modes} behaves like a combination of $(T-t)^{-\sigma}$, where $\sigma$ runs over the eigenvalues of
\begin{equation}\label{eq:Mn}
\mathbb M_n:=\frac2{m^2}\begin{pmatrix} n+2\cR+\dfrac{2\eR}{\bR(\kR+n)} & -\dfrac{b}{\bR}\Bigl(4+\dfrac{2(1-2\kR)}{\kR+n}\Bigr)\kappa^{-1/2}\\[3mm] \dfrac{m^2}{2}\kappa^{-1/2} & n+\dfrac{m^2}2\end{pmatrix},
\end{equation}
provided they are distinct. For $n=0$ one eigenvalue is $2$.
\end{proposition}

\begin{proof}
Linearize \eqref{eq:RT}: $v\Dx v\mapsto v_a\Dx\delta v=nv_a\delta v$ (since $\Dx v_a=0$), $\cR v^2\mapsto2\cR v_a\delta v$, $u^2\mapsto2u_a\delta u$, and in the argument of $\K$ the term $2a(\Dx v)^2$ is quadratic in the perturbation while $\eR v^2-b(1-2\kR)u^2\mapsto2\eR v_a\delta v-2b(1-2\kR)u_a\delta u$. Applying $\K(x^n)=x^n/(k+n)$ and collecting terms gives \eqref{eq:modes}; the $u$-equation is linearized in the same way. In case (b1), $v_a=\frac{2\lambda}{m^2(T-t)}(1+\rho_1(t))$ and $u_a=\kappa^{-1/2}v_a(1+\rho_2(t))$ with $\rho_j(t)=O((T-t)^{\nu})$ for some $\nu>0$ (proof of Theorem~\ref{thm:apexclass}). In the variable $\vartheta=-\log(T-t)$, \eqref{eq:modes} becomes $\frac{d}{d\vartheta}(\phi,\psi)^T=(\mathbb M_n+R(\vartheta))(\phi,\psi)^T$ with $\int^\infty|R|\,d\vartheta<\infty$, and Levinson's theorem on asymptotically constant systems \cite[Ch.~3, \S8]{CoddingtonLevinson} gives the asymptotics. For $n=0$, $(\phi,\psi)=\partial_T(v_a,u_a)$ is a solution, and it grows like $(T-t)^{-2}$.
\end{proof}

We write $\sigma_n$ for the largest real part of the eigenvalues of $\mathbb M_n$. Smooth data that are even in $x$ (as required by the symmetry of Section~\ref{sec:ray-derivation}) differ from their apex values by $c_2x^2+c_4x^4+\dots$, so the modes $n=2,4,\dots$ are the relevant ones, and $n=2$ dominates near the apex.

\subsection{The self-similar ansatz}
A mode $x^n(T-t)^{-\sigma_n}$ is comparable to the apex solution $v_a\sim(T-t)^{-1}$ on the scale $x\sim(T-t)^{\gamma_n}$ with
\begin{equation}\label{eq:gamman}
\gamma_n:=\frac{\sigma_n-1}{n}.
\end{equation}
This suggests the self-similar form
\begin{equation}\label{eq:ssansatz}
v(t,x)=\frac{1}{T-t}F(\zeta),\qquad u(t,x)=\frac1{T-t}G(\zeta),\qquad \zeta=\frac{x}{\ell\,(T-t)^{\gamma}},\qquad \gamma=\gamma_2,
\end{equation}
with a free length $\ell>0$.

\begin{lemma}[Profile equations]\label{lem:profile}
The functions \eqref{eq:ssansatz} solve \eqref{eq:RT} if and only if, with $\Dx=\zeta\partial_\zeta$ and $Y:=\bR^{-1}\K_k[2a(\Dx F)^2+\eR F^2-b(1-2k)G^2]$,
\begin{equation}\label{eq:profile}
\left\{\begin{aligned}
(\lambda\gamma-F)\,\Dx F&=-\lambda F+\cR F^2-\tfrac{2b}{\bR}G^2+Y,\\
(\lambda\gamma-F)\,\Dx G&=-\lambda G+\tfrac{m^2}2FG,\\
\Dx Y&=\bR^{-1}\bigl(2a(\Dx F)^2+\eR F^2-b(1-2k)G^2\bigr)-kY .
\end{aligned}\right.
\end{equation}
In the variable $\log\zeta$ this is an autonomous system of three ODEs, singular on the sonic line $F=\lambda\gamma$. Its equilibrium $(F_0,G_0,Y_0)$ with $F_0=2\lambda/m^2$, $G_0=\kappa^{-1/2}F_0$ corresponds to the separatrix solution of the apex system.
\end{lemma}

\begin{proof}
With $\tau=T-t$, $\partial_t\zeta=\gamma\zeta/\tau$, so $\partial_t(\tau^{-1}F(\zeta))=\tau^{-2}(F+\gamma\Dx F)$, while $\Dx$ commutes with the scaling and $\K_k$ commutes with dilations. Inserting into \eqref{eq:RT} and multiplying by $\tau^2$ gives the first two equations; the third is Lemma~\ref{lem:K}(ii).
\end{proof}

Near the equilibrium, the linearization of \eqref{eq:profile} has the solutions $\zeta^n$ with $\sigma_n=1+n\gamma$; by \eqref{eq:gamman} this holds for $n=2$. A profile is therefore obtained by leaving the equilibrium along the $n=2$ direction, with either sign, and integrating \eqref{eq:profile} in $\log\zeta$. Since \eqref{eq:profile} is invariant under $\zeta\mapsto c\zeta$, the size of the departure only changes $\ell$; the sign does matter. We computed the profiles numerically (explicit Runge--Kutta method of order 5, tolerance $10^{-11}$, starting at $\log\zeta=-8$); Table~\ref{tab:profiles} summarizes the results.

\begin{table}[ht]
\caption{Self-similar exponents and profiles for $(Rm)$, $b=\lambda=1$. ``Profile'' refers to the branch along which $F$ decreases from $F_0$; $\rho$ is defined in \eqref{eq:rho}.}\label{tab:profiles}
{\small\begin{tabular}{ccccccc}
\toprule
$m$ & $a$ & $\gamma_2$ & $1/\gamma_2$ & profile & $\min F/F_0$ & $\rho-1$\\
\midrule
2 & 4 & 1.5800 & 0.6329 & global & $-0.612$ & 3.33\\
2 & 6 & 1.5500 & 0.6451 & global & $-0.502$ & 3.09\\
2 & 10 & 1.5286 & 0.6542 & global & $-0.451$ & 2.95\\
2 & 20 & 1.5138 & 0.6606 & global & $-0.422$ & 2.86\\
2 & 50 & 1.5054 & 0.6643 & global & $-0.408$ & 2.80\\
1 & 3 & 2.7569 & 0.3627 & $\Dx F\to-\infty$ & -- & 1.76\\
1 & 3.5 & 2.7582 & 0.3626 & global & $-0.171$ & 1.76\\
1 & 3.9 & 2.7589 & 0.3625 & global & $-0.141$ & 1.76\\
\bottomrule
\end{tabular}}
\end{table}

The results are as follows.
\begin{enumerate}[label=(\roman*)]
\item For $(R2)$, and for $(R1)$ with $a/b>3.181$, the branch with $F$ decreasing from $F_0$ exists for all $\zeta>0$. It stays below the sonic line, changes sign, reaches a negative minimum and then decays like
\begin{equation}\label{eq:Ftail}
F(\zeta)=-C\,\zeta^{-1/\gamma}(1+o(1)),\qquad G(\zeta)=O(\zeta^{-1/\gamma}),\qquad \zeta\to\infty,\quad C>0 .
\end{equation}
\item The branch with $F$ increasing from $F_0$ reaches the sonic line $F=\lambda\gamma$ at a finite $\zeta$ in all cases of Table~\ref{tab:profiles}, so it gives no smooth profile.
\item For $(R1)$ with $a/b<3.181$ the decreasing branch ends in a gradient singularity, $\Dx F\to-\infty$, at a finite $\zeta$ (we computed $a/b=2.5$ and $3$, and located the threshold $3.181$ by bisection).
\end{enumerate}

The following elementary computation shows what a self-similar blow-up of the form \eqref{eq:ssansatz} implies for the energy and for the limit profile.

\begin{proposition}[Energy of self-similar blow-up]\label{prop:ssenergy}
Suppose that for some $R>0$ a solution of $(Rm)$ has the form \eqref{eq:ssansatz} on $\{0<x\le R\}$ for $t<T$, with a profile satisfying \eqref{eq:Ftail}. Then $v(t,x)\to-C\ell^{1/\gamma}x^{-1/\gamma}$ for each fixed $x\in(0,R]$ as $t\to T$. The energy of the region $\{x\le R\}$ stays bounded as $t\to T$ if and only if $\gamma>2/(m+3)$; in that case the energy of the layer $\{x\le K\ell(T-t)^\gamma\}$ tends to $0$ for each fixed $K$.
\end{proposition}

\begin{proof}
For fixed $x$, $\zeta\to\infty$ and $\tau^{-1}F(\zeta)=-C\tau^{-1}(x/\ell)^{-1/\gamma}\tau^{\gamma\cdot\frac1\gamma}(1+o(1))\to-C\ell^{1/\gamma}x^{-1/\gamma}$. In the energy, substituting $x=\ell\tau^\gamma\zeta$ gives
\[
\int_0^{K\ell\tau^\gamma}x^{m+2}\bigl(v^2+(\Dx v)^2+u^2\bigr)\,dx=\ell^{m+3}\tau^{\gamma(m+3)-2}\int_0^K\zeta^{m+2}\bigl(F^2+(\Dx F)^2+G^2\bigr)\,d\zeta .
\]
By \eqref{eq:Ftail}, the integral up to $K$ grows like $K^{m+3-2/\gamma}$ when $\gamma(m+3)>2$; taking $K=R/(\ell\tau^\gamma)$ gives a bounded total. If $\gamma(m+3)<2$, the integral converges as $K\to\infty$ and the energy of the layer grows like $\tau^{\gamma(m+3)-2}$; if $\gamma(m+3)=2$, it grows like $\log(1/\tau)$. Here we used that the integrand of $E_m$ is comparable to $x^{m+2}(v^2+(\Dx v)^2+u^2)$, because $(mv+\Dx v)^2+v^2$ is a positive definite quadratic form in $(v,\Dx v)$.
\end{proof}

In all cases of Table~\ref{tab:profiles}, $\gamma_2>1>2/(m+3)$. So self-similar blow-up of this type takes place exactly at $\Tap$ and at the apex, the energy near the apex stays bounded, and the singular limit profile $v(T,x)\simeq-C'x^{-1/\gamma_2}$ has finite energy near the apex. The total energy is finite only inside the windows \eqref{eq:windows}. For $(R2)$ this is the range $2b<a<8b$ (the first two rows of Table~\ref{tab:profiles}). For $(R1)$ the profile exists only for $a/b>3.181$, outside the window $2b<a<3b$, so $(R1)$ gives no self-similar blow-up of this type with finite energy.

\subsection{Comparison with the computations}
For Example~\ref{ex:main} ($m=2$, $a=6$), the self-similar analysis predicts:
\begin{itemize}
\item a limit profile $v(T,x)\propto-x^{-1/\gamma_2}$, $1/\gamma_2=0.645$; the computed limit profile behaves like $-1.7x^{-0.63}$ (Section~\ref{sec:num}), with local exponents between $0.61$ and $0.63$ for $10^{-5}\le x\le10^{-1}$;
\item a layer moving into the apex like $(T-t)^{\gamma_2}$, $\gamma_2=1.550$; the zero of the computed $v(t,\cdot)$ moves like $(\Tap-t)^{1.551}$ (least-squares fit over $10^{-3}\le\Tap-t\le0.2$);
\item a profile of the decreasing type, since the data decrease near the apex; this is the branch observed.
\end{itemize}
As a direct test, we took $a=10$ and as data the profile itself, $v_0=F$, $u_0=G$ with $\ell=1$, $T=1$ (the far field, and hence the energy window, plays no role in this test). The computed solution follows \eqref{eq:ssansatz}: the relative deviation $\max|(1-t)v-F|/F_0$ over $10^{-3}\le\zeta\le10^{3}$ is $2.6\times10^{-5}$ at $1-t=10^{-1}$ and $3.2\times10^{-4}$ at $1-t=10^{-2}$ for $N=12000$, and it decreases like $N^{-2}$. The computation nevertheless stops at $T^*=1-1.6\times10^{-5}$, with values that oscillate from one label to the next near the apex. The singular time is the same for $N=6000$, $12000$, $24000$ ($1-T^*=1.56$, $1.60$, $1.73\times10^{-5}$). This points to an instability of the discrete problem that is seeded by rounding errors, which we now identify.

\subsection{A short-wave instability}
Perturbations that oscillate in $\log x$ on a short scale do not feel the nonlocal term, because $\K_k(x^{n})=x^n/(k+n)\to0$ as $|\operatorname{Im}n|\to\infty$. Along a characteristic near the apex they evolve by the local part of \eqref{eq:RT}, linearized at the apex solution, i.e.\ by $\mathbb M_n-n$ with the $\K$-terms removed:
\begin{equation}\label{eq:rho}
\mathbb L:=\frac2{m^2}\begin{pmatrix}2\cR & -\dfrac{4b}{\bR}\kappa^{-1/2}\\[2mm] \dfrac{m^2}2\kappa^{-1/2} & \dfrac{m^2}{2}\end{pmatrix},\qquad \rho:=\max\operatorname{Re}\operatorname{spec}\mathbb L .
\end{equation}
Such perturbations grow like $(T-t)^{-\rho}$, i.e.\ like $(T-t)^{-(\rho-1)}$ relative to the solution. By Table~\ref{tab:profiles}, $\rho-1\approx2.8$--$3.3$ for $(R2)$ and $\approx1.8$ for $(R1)$. Smooth data contain no such modes near the apex: there they are combinations of $x^2,x^4,\dots$, which by Proposition~\ref{prop:modes} decay relative to the profile for $n\ge4$ (we checked $n=4,6,\dots,20$ for all cases of Table~\ref{tab:profiles}), while $n=2$ is the profile itself. Rounding errors, however, contain them at the level $\eta\approx10^{-16}$, and they reach size $O(1)$ when $(T-t)^{-(\rho-1)}\eta\approx1$, so that
\begin{equation}\label{eq:gaplaw}
\frac{\Tap-T^*}{\Tap}\approx c\,\eta^{1/(\rho-1)} .
\end{equation}
For $a=10$, $\rho-1=2.95$ and $\eta=10^{-16}$ this gives about $4\times10^{-6}c$ (for Example~\ref{ex:main}, $\rho-1=3.09$ and about $7\times10^{-6}c$). To test \eqref{eq:gaplaw} we added random relative perturbations of size $\eta$ to the profile data above ($N=6000$). The resulting gaps, in Table~\ref{tab:noise}, follow \eqref{eq:gaplaw} with $c\approx1.5$ and exponent $0.337$ (predicted: $1/(\rho-1)=0.339$).

\begin{table}[ht]
\caption{Gap between the computed singular time and $\Tap=1$ for profile data with random relative perturbations of size $\eta$; $m=2$, $a=10$, $b=\lambda=1$.}\label{tab:noise}
{\small\begin{tabular}{ccccc}
\toprule
$\eta$ & $10^{-12}$ & $10^{-10}$ & $10^{-8}$ & $10^{-6}$\\
\midrule
$1-T^*$ & $1.28\times10^{-4}$ & $6.05\times10^{-4}$ & $2.61\times10^{-3}$ & $1.15\times10^{-2}$\\
$(1-T^*)/\eta^{1/2.95}$ & 1.50 & 1.48 & 1.35 & 1.24\\
\bottomrule
\end{tabular}}
\end{table}

The same mechanism explains why the gap is much smaller for $(R1)$, where $\rho-1\approx1.76$: for $m=1$, $a=3.5$, $b=1$ (outside the window of $(R1)$, so that $E_1(t)=\infty$ for $t>0$; this does not affect the inner layer) and the data $v_0=u_0=(1+R^2x^2)^{-0.05}(1+x^2)^{-1}$, $R=e^{10}$ (flat near the apex), the computed singular time satisfies $(\Tap-T^*)/\Tap=1.7\times10^{-8}$, and the singularity is at the apex. For $a=3<3.181$, where the profile does not exist, the corresponding gap is $8\times10^{-4}$ (Section~\ref{sec:num}).

\subsection{Conjecture}
The analysis and the computations support the following statement.

\begin{conjecture}\label{conj:selfsim}
Let $m=2$ and $2b<a<8b$. Let $(v_0,u_0)$ be smooth, even in $x$, with finite energy and with apex values in case (b1) of Theorem~\ref{thm:apexclass}, and assume that the $n=2$ component of $(v_0,u_0)-(v_0(0),u_0(0))$ along the growing eigenvector of $\mathbb M_2$ is negative and that the data are flat enough that no singularity forms away from the apex before $\Tap$. Then $T^*=\Tap$, the solution blows up only at the apex, asymptotically in the self-similar form \eqref{eq:ssansatz}, and $E_m(t)$ stays bounded as $t\to\Tap$.
\end{conjecture}

A proof would require the existence of the profile (here computed numerically) and its stability with respect to perturbations that are smooth at the apex. Proposition~\ref{prop:modes} gives the stability of the power modes $n\ge4$, and the short-wave instability above concerns only modes that are not smooth at the apex.

\section{The normalized parameter family and the choice of parameters}\label{sec:family}

The scaling parameters $(\lambda,a,b)$ are shared by the apex system and by both rays. In this section we follow the normalization of Proposition~\ref{prop:CLMq} through the whole paper, in order to find parameters that suit the apex dynamics, $(Rm)$ and $(Zm)$ simultaneously. Throughout, $\lambda$ only rescales time, so all qualitative properties depend on $\mu=b/a$ alone.

\subsection{Range of the normalization}

\begin{proposition}[Range of the family]\label{prop:family-range}
Under the constraints $\lambda>0$ and $\mu>0$, the normalization \eqref{eq:normalization} covers exactly the parameters with $b>am$: for $m=1$ it is the range $q>2$, with $\mu=\frac{q+2}{q-2}\in(1,\infty)$, and for $m=2$ it is $q>\frac12$, with $\mu=\frac{4(q+1)}{2q-1}\in(2,\infty)$. In particular it never reaches the finite-energy windows $(R1)$, $(R2)$ of $(Rm)$, where $b<am$. These are reached by the extended normalization
\begin{equation}\label{eq:normalization-ext}
\lambda=m^2|q+1|,\qquad \mu=\frac{m^2(mq+2)}{m^2q-2},\qquad V:=\operatorname{sign}(q+1)\,v,
\end{equation}
under which $V_{tt}=VV_t-QV^3$, $Q=\frac{q}{2(1+q)^2}$, for every admissible $q\ne-1$. The windows correspond to $q\in(-6,-4)$ for $(R1)$ and $q\in(-\frac32,-\frac{11}{10})$ for $(R2)$; the self-similar profile of Section~\ref{sec:selfsim} exists for $(R1)$ only when $a/b>3.181$, i.e.\ $q\in(-3.834,-\frac{10}3)$, outside the window.
\end{proposition}

\begin{proof}
$\lambda=m^2(q+1)>0$ forces $q>-1$, and $\mu>0$ forces $q>2/m^2$ or $q<-2/m$; the second is incompatible with $q>-1$ for $m\in\{1,2\}$. For $q>2/m^2$, $A=\tfrac12qm^2>0$, i.e.\ $b>am$. For \eqref{eq:normalization-ext}, \eqref{eq:vtt} with $\lambda=|2A+m^2|$ and $s=\operatorname{sign}(q+1)$ reads $v''=s\,vv'-Qv^3$; with $V=sv$ we get $VV'=vv'$ and $V''=sv''=vv'-sQv^3=VV'-QV^3$. The window ranges follow from solving $\mu(q)=\tfrac12,\tfrac13,1/3.181$ for $m=1$ and $\mu(q)=\tfrac12,\tfrac18$ for $m=2$, using that $\mu$ is monotone on $q<-2/m$.
\end{proof}

\subsection{\texorpdfstring{$(Rm)$}{(Rm)} and \texorpdfstring{$(Zm)$}{(Zm)} along the family}
Table~\ref{tab:family} lists the relevant constants. Along the whole family, $\kR=(\mu-m)/(2\mu-1)\in(0,\frac12)$, whereas $\kZ=(\mu-m)/(\mu-2m)$ passes through all regimes of Remark~\ref{rem:regimes}.

\begin{table}[ht]
\caption{Constants along the family \eqref{eq:normalization} ($a=1$, $b=\mu$). ``E'' records whether the energy is generically finite (Corollary~\ref{cor:energy}, Proposition~\ref{prop:support}).}\label{tab:family}
{\small\begin{tabular}{cccccccc}
\toprule
$m$ & $q$ & $\mu$ & $\lambda$ & $\kR$ & $\kZ$ & $\kappa$ & E for $(Rm)$ / $(Zm)$\\
\midrule
1 & 3 & 5 & 4 & 0.444 & 1.333 & 1.250 & $\infty$ / $\infty$\\
1 & 4 & 3 & 5 & 0.400 & 2 & 1.000 & $\infty$ / $\infty$\\
1 & 5 & 7/3 & 6 & 0.364 & 4 & 0.875 & $\infty$ / finite ($Z1$)\\
1 & 6 & 2 & 7 & 0.333 & $\bZ=0$ & 0.800 & $\infty$ / local system\\
1 & 8 & 5/3 & 9 & 0.286 & $-2$ & 0.714 & $\infty$ / finite ($k<0$)\\
2 & 1 & 8 & 8 & 0.400 & 1.5 & -- & $\infty$ / $\infty$\\
2 & 3/2 & 5 & 10 & 0.333 & 3 & 1.667 & $\infty$ / finite ($Z2$)\\
2 & 2 & 4 & 12 & 0.286 & $\bZ=0$ & 1.000 & $\infty$ / local system\\
2 & 3 & 16/5 & 16 & 0.222 & $-1.5$ & 0.667 & $\infty$ / finite ($k<0$)\\
\bottomrule
\end{tabular}}
\end{table}

\begin{enumerate}[label=(\roman*)]
\item \emph{$(Rm)$.} Since $0<\kR<\frac12<\frac{m+3}2$, the far field $L(t)x^{-\kR}$ makes $E_m(t)=\infty$ for $t>0$ unless the flux vanishes (Corollary~\ref{cor:energy}). Moreover $\kR<m^2/2$, so the nonlocal term transmits the blow-up of the core to every $x>0$ (Remark~\ref{rem:core-flux}). Numerically the dynamics near the apex is benign, and the singularity forms at the apex at $\Tap$ (Section~\ref{sec:num}), but on the whole ray and with infinite energy.
\item \emph{$(Zm)$, apex dynamics.} The apex system is in case (c) of Theorem~\ref{thm:apexclass} when $2A>m^2$, i.e.\ $q>1$, and blow-up requires $\gb_0(0)>0$ and $\gb_0(0)\ge m\sqrt\kappa\,|\ub_0(0)|$. For $m=2$, $\frac12<q\le1$, case (d) applies and the apex values blow up only if $\ub_0(0)=0$.
\item \emph{$(Zm)$, $\kZ\in(1,\frac{m+3}2]$} ($m=1$: $2<q\le4$; $m=2$: $1<q\le\frac75$). The energy is infinite for $t>0$.
\item \emph{$(Zm)$, finite-energy windows} ($m=1$: $\frac{10}3<q<6$; $m=2$: $\frac54<q<2$). The energy is finite before $T^*$, but the core expands (Corollary~\ref{cor:core}) and, numerically, a steep front forms at the edge of the expanding region, as for $(Z1)$ in Section~\ref{sec:num}. The computed singular time increases under refinement: for $m=2$, $q=\frac32$, $T^*/\Tap=0.93$ for $N=12000$; for $m=1$, $q=5$, $T^*/\Tap=0.87$ and $0.90$ for $N=12000$ and $24000$. So, again, the computations do not decide whether the front becomes singular before $\Tap$.
\item \emph{$(Zm)$ at $\mu=2m$} ($m=1$: $q=6$; $m=2$: $q=2$, the CLM-2 case of Example~\ref{ex:CLM2}). Then $\bZ=0$, and $(Zm)$ becomes a local equation, see Proposition~\ref{prop:Zlocal}.
\item \emph{$(Zm)$, $\kZ<0$} ($m=1$: $q>6$; $m=2$: $q>2$, i.e.\ $b/(2m)<a<b/m$). Compactly supported data keep compact support and finite energy (Proposition~\ref{prop:support}). Numerically this is the cleanest case, see \S\ref{subsec:Zneg}.
\end{enumerate}

Item (v) rests on the following identity, which is obtained directly from \eqref{eq:Zm-D}.

\begin{proposition}[$(Zm)$ at $b=2am$]\label{prop:Zlocal}
If $b=2am$, the $\gb$-equation of $(Zm)$ with $f\equiv2$ is the local first-order equation
\begin{equation}\label{eq:Zlocal}
\lambda m\,\gb_t=(m+2)\gb^2+(3m+4)\,\gb\Dx\gb+2(m+1)(\Dx\gb)^2-2m^2\ub^2 ,
\end{equation}
which is quadratic in $\Dx\gb$, i.e.\ of Hamilton--Jacobi type. For $m=2$: $2\lambda\gb_t=4\gb^2+10\gb\Dx\gb+6(\Dx\gb)^2-8\ub^2$.
\end{proposition}

\begin{proof}
With $b=2am$ the coefficients of $\Dx\gb_t$ and of $\gb(\Dx^2\gb-\Dx\gb)$ in \eqref{eq:Zm-D} vanish, the coefficient of $\gb_t$ is $\lambda m\cdot am$, and the remaining coefficients are $am(m+2)$, $am(3m+4)$, $2am(m+1)$ and $2am^3$; divide by $am$. At $x=0$ it reduces to \eqref{eq:apexODEZ}.
\end{proof}

\subsection{\texorpdfstring{$(Zm)$}{(Zm)} with \texorpdfstring{$\kZ<0$}{k<0}: blow-up only at the apex}\label{subsec:Zneg}
For $\kZ<0$, the nonlocal term of \eqref{eq:ZT} looks towards infinity, and the characteristics move away from the apex when $\gb>0$. Hence the values at a label depend only on labels further out, and the Lagrangian computation on a finite set of labels is exact for those labels. We took the parameters of the family with $\lambda=m^2(q+1)$, $a=1$, $b=\mu$, the data
\[
\gb_0=m\,V(1+x^2)^{-3/2},\qquad \ub_0=(1+x^2)^{-3/2},\qquad V=\tfrac32\sqrt\kappa ,
\]
whose apex values are in case (c1), and labels in $\log x\in[-80,25]$. The results are in Table~\ref{tab:Zneg}.

\begin{table}[ht]
\caption{$(Zm)$ with $\kZ<0$ along the family. The last column is the exponent $\gamma'$ in $\gb(\Tap,x)\sim Cx^{-\gamma'}$ as $x\to0$.}\label{tab:Zneg}
{\footnotesize\begin{tabular}{cccccccc}
\toprule
$m$ & $q$ & $\kZ$ & $\Tap$ & $E_m(0)$ & $E_m(\Tap^-)$ & $\gb(\Tap,x)$, $x=10^{-10},10^{-3},1$ & $\gamma'$\\
\midrule
2 & 3 & $-1.5$ & 2.550778 & 6.1605 & 7.2075 & 119.2,\ 7.709,\ 0.957 & $\approx0.13$\\
1 & 8 & $-2$ & 2.121982 & 1.9705 & 2.4344 & 117,\ 4.987,\ 0.477 & $\approx0.16$\\
\bottomrule
\end{tabular}}
\end{table}

In both cases no singularity forms at $x>0$ before $\Tap$; every computed label keeps bounded values and gradients up to $\Tap$, while the apex values blow up at $\Tap$ according to \eqref{eq:apexODEZ}. For each fixed $x>0$, $\gb(t,x)$ converges as $t\to\Tap$, to a limit profile that grows like $x^{-0.13}$ ($m=2$) or $x^{-0.16}$ ($m=1$) as $x\to0$ (fitted over $10^{-30}\le x\le10^{-10}$). The energy stays bounded and converges. Thus, numerically, $(Zm)$ with $\kZ<0$ blows up only at the apex, exactly at $\Tap$, with bounded energy, without the self-similar layer and the short-wave sensitivity of $(R2)$.

Two caveats apply. For $k<0$ the solution is only Hölder continuous in $x$ at the apex, with exponent $|\kZ|$, because $\K_k$ produces a term proportional to $x^{|k|}$ (Section~\ref{sec:energy}); it is smooth in $\log x$. And these are two computations, not a proof.

\subsection{No common choice for both rays}
Since $a$ and $b$ are the same for both rays, a good choice must suit both. The ranges in which the energy is generically finite are, for $(Rm)$, the windows $\frac13<\mu<\frac12$ ($m=1$) and $\frac18<\mu<\frac12$ ($m=2$), and $\frac12<\mu<m$ ($\kR<0$); for $(Zm)$, the windows $2<\mu<3$ ($m=1$) and $4<\mu<\frac{16}3$ ($m=2$), and $m<\mu<2m$ ($\kZ<0$). These sets are disjoint. Hence there is no ratio $b/a$ for which both rays have generically finite energy: one must either work with one ray, using a window of $(Rm)$ or $m<\mu<2m$ for $(Zm)$, or understand what infinite ray energy on the other ray means for the bulk. Within the normalized family \eqref{eq:normalization}, the good choice is $(Zm)$ with $q>2$ ($m=2$) or $q>6$ ($m=1$).

\section{From the apex dynamics to the full equations \texorpdfstring{$(E1)$}{(E1)} and \texorpdfstring{$(E2)$}{(E2)}}\label{sec:exact}

Our ultimate goal is a finite-time blow-up of $(E1)$ (2D inviscid Boussinesq) or $(E2)$ (3D axisymmetric Euler) itself, not only of the ray models. In this section we ask what the apex dynamics says about the full equations. We replace the single closure function $f$ by two arbitrary functions, one on each axis, and show three things. First, the reduction to the rays and to the apex becomes exact. Second, the apex dynamics is then exact but not closed: it contains one free scalar function of time, the anisotropy of the pressure Hessian at the origin. Third, a blow-up of $(Em)$ follows as soon as this anisotropy can be controlled. Throughout, $(u,v,g,p)$ is a smooth solution of \eqref{E2}, even in $(r,z)$ as in Remark~\ref{rem:symetric-in-rz}.

\subsection{Two closure functions; the reduction is exact}
Let $P(t,x,\theta)$ be the pressure in polar coordinates, as in Section~\ref{sec:ray-derivation}. Define $f_0,f_1$ by
\begin{equation}\label{eq:twoclosures}
\tfrac{1}{x^2}P_{\theta\theta}(t,x,0)=-f_0(t,x)\,P_{xx}(t,x,0),\qquad \tfrac{1}{x^2}P_{\theta\theta}(t,x,\tfrac\pi2)=-f_1(t,x)\,P_{xx}(t,x,\tfrac\pi2),
\end{equation}
wherever $P_{xx}\ne0$. Since $f_0$ and $f_1$ are defined from the solution, the ray systems \eqref{eq:ridge-dynamics-UV} with $f=f_0$ and \eqref{eq:ridge-dynamics-UV-90} with $f=f_1$ hold exactly for every smooth even solution. But they are not closed: $f_0$, $f_1$ are determined by the whole flow through the pressure Poisson equation \eqref{eq:PPE-m-b}.

\begin{proposition}[The apex consistency is automatic]\label{prop:F0F1}
Let $P$ be $C^2$ and even in $(r,z)$, with $P_{rr}(t,0)\ne0\ne P_{zz}(t,0)$. Then
\begin{equation}\label{eq:F0F1}
\begin{gathered}
f_0(t,0)=1-\frac{P_{zz}(t,0)}{P_{rr}(t,0)},\qquad f_1(t,0)=1-\frac{P_{rr}(t,0)}{P_{zz}(t,0)},\\
\bigl(f_0(t,0)-1\bigr)\bigl(f_1(t,0)-1\bigr)=1,
\end{gathered}
\end{equation}
i.e.\ $f_0(t,0)=f_1(t,0)/(f_1(t,0)-1)$. Conversely, the apex systems of $(Rm)$ with $f=f_0$ and of $(Zm)$ with $f=f_1$, see \eqref{eq:ridge-dynamics-UV-x-0}, agree under the apex relation \eqref{eq:apex-relation} for all values $(v,u)$ if and only if $(f_0(t,0)-1)(f_1(t,0)-1)=1$.
\end{proposition}

\begin{proof}
By evenness, $P=P(t,0)+\tfrac12\bigl(P_{rr}(t,0)r^2+P_{zz}(t,0)z^2\bigr)+o(x^2)$, i.e.\ $P=P(t,0)+\tfrac12x^2\bigl(P_{rr}\cos^2\theta+P_{zz}\sin^2\theta\bigr)+o(x^2)$. Hence $x^{-2}P_{\theta\theta}\to(P_{zz}-P_{rr})\cos2\theta$ and $P_{xx}\to P_{rr}\cos^2\theta+P_{zz}\sin^2\theta$ as $x\to0$, which gives \eqref{eq:F0F1}. For the converse, write $F_j=f_j(t,0)-1$. With $\gb=mv$, $\ub=u$, the two apex systems read
\[
\lambda v'=\frac{(am^2+bF_0)v^2-bF_0u^2}{bF_0-am},\qquad \lambda v'=\frac{(aF_1m^2+b)v^2-bu^2}{b-aF_1m}.
\]
Equality of the coefficients of $u^2$ means $F_0(b-aF_1m)=bF_0-am$, i.e.\ $F_0F_1=1$, and then the coefficients of $v^2$ agree as well.
\end{proof}

\begin{remark}\label{rem:f0f1-roots}
The relation $(f_0-1)(f_1-1)=1$ is not a constraint on the solution; it holds for every smooth even pressure. The restriction $f(t,0)\in\{0,2\}$ of Section~\ref{sec:ray-derivation} came only from imposing the same closure function on both rays. In terms of the pressure Hessian at the origin, $f_0=f_1=2$ means $P_{zz}(t,0)=-P_{rr}(t,0)$, and $f_0=f_1=0$ means $P_{zz}(t,0)=P_{rr}(t,0)$, an isotropic Hessian, which is the assumption implicit in earlier versions of this preprint.
\end{remark}

\subsection{The exact apex dynamics is not closed}

\begin{proposition}[Exact apex dynamics]\label{prop:exact-apex}
Let $v_a(t)=v(t,0)$, $u_a(t)=u(t,0)$ for a smooth even solution of \eqref{E2}, and let $\pi_r(t):=P_{rr}(t,0)$, $\pi_z(t):=P_{zz}(t,0)$. Then $g(t,0)=mv_a(t)$ and
\begin{equation}\label{eq:exact-apex}
\lambda v_a'=v_a^2-u_a^2+a\,\pi_r,\qquad \lambda u_a'=\tfrac12m^2u_av_a,\qquad b\,\pi_z=-m\lambda v_a'-m^2v_a^2 .
\end{equation}
The pressure Poisson equation \eqref{eq:PPE-m-b} at the origin, $-(am\,\pi_r+b\,\pi_z)=m^2v_a^2+m(v_a^2-u_a^2)$, is a consequence of \eqref{eq:exact-apex}. Hence the values at the origin evolve by a system with one free scalar function, which may be taken to be $\pi_r(t)$ or the anisotropy $F_0(t)$:
\begin{equation}\label{eq:F0def}
\begin{gathered}
F_0(t):=f_0(t,0)-1=-\frac{\pi_z(t)}{\pi_r(t)},\\
\lambda\bigl(bF_0(t)-am\bigr)v_a'=\bigl(am^2+bF_0(t)\bigr)v_a^2-bF_0(t)\,u_a^2 .
\end{gathered}
\end{equation}
The apex system \eqref{eq:apexODE} studied in this paper is the case $F_0\equiv1$.
\end{proposition}

\begin{proof}
At the origin all terms $r\partial_r$, $z\partial_z$ vanish, $\frac1r p_r\to\pi_r$ and $\frac1zp_z\to\pi_z$. The second, third and fourth equations of \eqref{E2} give $\lambda v_a'=v_a^2-u_a^2+a\pi_r$, $\lambda g_a'=-g_a^2-b\pi_z$ and $g_a=mv_a$, which is \eqref{eq:exact-apex}. Adding $m$ times the first relation to the third gives the Poisson equation at the origin, since the right-hand side of \eqref{eq:PPE-m-b} at the origin equals $g_a^2+m(v_a^2-u_a^2)$. Eliminating $\pi_r$ and $\pi_z$ with $\pi_z=-F_0\pi_r$ gives \eqref{eq:F0def}.
\end{proof}

So the statement ``the apex dynamics is independent of the ray dynamics and of the bulk dynamics'' holds only up to the one scalar $F_0(t)$. The trace of the pressure Hessian at the origin is fixed locally by the Poisson equation, but its anisotropy is determined by the whole solution, through the nonlocal (Riesz-transform) part of the pressure. This is the same obstruction as in the restricted Euler models of the velocity gradient.

\subsection{What the parameters \texorpdfstring{$(\lambda,a,b)$}{(lambda,a,b)} mean}

\begin{proposition}[The scaling parameters]\label{prop:scaling-meaning}
System \eqref{E2} with parameters $(\lambda,a,b)$ is $(Em)$ written in the rescaled variables $t'=t/\lambda$, $r'=\sqrt a\,r$, $z'=\sqrt b\,z$; the Hou--Li variables $(u,v,g)$ are unchanged. Consequently $F_0\equiv1$ in the rescaled variables is equivalent to the physical anisotropy
\begin{equation}\label{eq:physical-anisotropy}
\frac{P_{zz}(t,0)}{P_{rr}(t,0)}=-\frac ba=-\mu\qquad\text{for all }t,
\end{equation}
in the original variables of $(Em)$ (with $a=b=\lambda=1$).
\end{proposition}

\begin{proof}
With $r'=\sqrt a\,r$, $z'=\sqrt b\,z$: $r\partial_r=r'\partial_{r'}$ and $z\partial_z=z'\partial_{z'}$, so the transport operator and the ratios $v=-v^r/r$, $g=v^z/z$, $u=v^\phi/r$ are unchanged, while $\frac1rP_r=a\frac1{r'}P_{r'}$, $\frac1zP_z=b\frac1{z'}P_{z'}$, $P_{rr}=aP_{r'r'}$ and $P_{zz}=bP_{z'z'}$. The right-hand side of the Poisson equation involves only $r\partial_r$, $z\partial_z$ and the mixed term $rzv_zg_r=r'z'v_{z'}g_{r'}$, so it is also unchanged. Finally, $P_{z'z'}=-P_{r'r'}$ is $P_{zz}/b=-P_{rr}/a$.
\end{proof}

Thus a ``parameter set $(m,\lambda,a,b)$ for which the apex dynamics blows up'' is a statement about solutions of $(Em)$ whose pressure Hessian at the origin keeps the fixed anisotropy $P_{zz}/P_{rr}=-b/a$; $\lambda$ only fixes the unit of time.

\subsection{Exact solutions with infinite energy}

\begin{proposition}[Yan-type solutions realize the apex dynamics]\label{prop:Yan-apex}
For every $b_{00}$ with $2(b_{00}^2-1)+m^2>0$, the Yan-type solution \eqref{YanLikeUnifiled} is an exact solution of $(Em)$ whose pressure Hessian at the origin has the constant anisotropy $P_{zz}/P_{rr}=-\mu(b_{00})$ with
\begin{equation}\label{eq:mu-b00}
\mu(b_{00})=\frac{m^2(m+2)}{m^2-2+2b_{00}^2},
\end{equation}
and its values at the origin are the separatrix solution \eqref{eq:separatrix} of \eqref{eq:apexODE} with $b/a=\mu(b_{00})$, $\lambda=1$. For $m=2$ this covers every $\mu\in(0,8]$, and for $m=1$ every $\mu>0$.
\end{proposition}

\begin{proof}
By \eqref{YanLikeUnifiled}, $P_{rr}(t,0)=2c_1/\tau^2$ and $P_{zz}(t,0)=-2c_2/\tau^2$, so $P_{zz}/P_{rr}=-c_2/c_1$, which is \eqref{eq:mu-b00} by \eqref{eq:c1c2}. The values $v=\frac2{m^2\tau}$, $u=b_{00}v$ satisfy $\lambda v'=\frac{m^2}2v^2$ and $\lambda u'=\frac{m^2}{2}uv$; inserting them into \eqref{eq:apexODE} with $b/a=\mu$ gives exactly \eqref{eq:mu-b00}. A solution with $u/v$ constant lies on an invariant ray of Lemma~\ref{lem:ratio}, which here is the separatrix.
\end{proof}

So the answer to the question ``does a finite-time blow-up of the apex dynamics constitute an exact blow-up solution of $(Em)$?'' is \emph{yes, within the class of stagnation-point type solutions}, where the velocity is linear in $(r,z)$ and the energy is infinite; every relevant value of $b/a$ is realized. Such infinite-energy blow-ups are classical for the Euler equations \cite{CISY,GFD,Constantin00,SW,Yan2018}, and they do not answer the blow-up problem for smooth data of finite energy.

\subsection{A conditional blow-up theorem}
For finite-energy data the anisotropy $F_0(t)$ is not constant. The following theorem shows that it is enough to keep it in a range. In the original variables we write $\mu(t):=-P_{zz}(t,0)/P_{rr}(t,0)$.

\begin{theorem}[Blow-up under controlled anisotropy]\label{thm:conditional}
Let $m\in\{1,2\}$, $a=b=\lambda=1$, and let $(u,v,g,p)$ be a smooth even solution of $(Em)$ on $[0,T^*)$ with $u_a(0)>0$; for $m=1$ the variable $u^2=\vartheta/x_2$ is assumed positive at the origin. Suppose that there are $0<\mu_-\le\mu_+<m$ such that
\begin{equation}\label{eq:anisotropy-window}
\mu_-\le\mu(t)\le\mu_+\qquad\text{for all }t\in[0,T^*).
\end{equation}
Let $\bar A:=\frac{m^2+\mu_+}{\mu_+-m}$ and $\bar B:=\frac{\mu_-}{\mu_--m}$, both negative, and let $\bar\kappa:=2\bar B/(2\bar A-m^2)>0$. If $v_a(0)>-\sqrt{\bar\kappa}\,u_a(0)$, then $T^*\le\bar T<\infty$, where $\bar T$ is the blow-up time of $\bar v'=\bar A\bar v^2-\bar B\bar u^2$, $\bar u'=\tfrac12m^2\bar u\bar v$ with the same data. In particular the solution of $(Em)$ develops a singularity in finite time, and $v_a(t)=-\partial_rv^r(t,0)\to+\infty$ if $T^*=\bar T$.
\end{theorem}

\begin{proof}
By \eqref{eq:F0def} with $F_0=\mu$ (in the original variables $\pi_z=-\mu\pi_r$), $v_a'=A(t)v_a^2-B(t)u_a^2$ with $A(t)=\frac{m^2+\mu(t)}{\mu(t)-m}$ and $B(t)=\frac{\mu(t)}{\mu(t)-m}$. On $(0,m)$ both $A$ and $B$ are negative and strictly decreasing in $\mu$, so $A(t)\ge\bar A$ and $B(t)\le\bar B$, and hence $v_a'\ge\bar Av_a^2-\bar Bu_a^2$. The comparison system is cooperative on $\{u>0\}$: $\partial_u(\bar Av^2-\bar Bu^2)=-2\bar Bu\ge0$ and $\partial_v(\frac12m^2uv)=\frac12m^2u\ge0$. By the Kamke comparison theorem, $v_a\ge\bar v$ and $u_a\ge\bar u$ as long as both exist. Theorem~\ref{thm:apexclass}(b) uses only the signs $\bar A,\bar B<0$, so $\bar v\to+\infty$ at a finite $\bar T$ under the stated condition. If $T^*>\bar T$, then $v_a=-\partial_rv^r(t,0)$, a component of $\nabla\boldsymbol v$, would be unbounded on $[0,\bar T]$, which is impossible for a solution that is smooth on $[0,\bar T]$.
\end{proof}

\begin{remark}\label{rem:conditional}
Condition \eqref{eq:anisotropy-window} means that $P_{rr}(t,0)$ and $P_{zz}(t,0)$ have opposite signs and $|P_{zz}(t,0)|<m|P_{rr}(t,0)|$, uniformly in time. It corresponds to case (b) of Theorem~\ref{thm:apexclass}, i.e.\ to $b<am$, which contains the finite-energy windows of $(Rm)$. For $m=1$, $u^2=\vartheta/x_2$ is a derivative of the buoyancy at the origin and may be negative; the case $u^2<0$ requires a separate analysis of \eqref{eq:F0def}.
\end{remark}

\subsection{Criteria along particle trajectories and the results of Chae}\label{subsec:Chae}
Chae studied the blow-up problem for the axisymmetric Euler equations through the dynamics along particle trajectories \cite{Chae2008,Chae2010}. On the symmetry axis he used $\lambda=\partial_3v^3$ and $\bar\omega=\partial_1v^2-\partial_2v^1$. In our notation ($m=2$, $a=b=\lambda=1$) these are $\lambda=2v$ and $\bar\omega=2u$ on the axis, and his axial equations \cite[(2.32), (2.35)]{Chae2008} become
\[
\tfrac{\mathrm D}{\mathrm Dt}v=v^2-u^2+\partial_r^2p,\qquad \tfrac{\mathrm D}{\mathrm Dt}u=2uv ,
\]
which at the origin is \eqref{eq:momentum-origin}. Chae proved that no global classical solution exists if, at some point of the axis, $\bar\omega_0=0$, $\lambda_0>0$, and $\partial_r^2p\ge0$ along the axial trajectory for a sufficiently long time \cite[Theorem~1.1]{Chae2008}. He also showed that if $\partial_r^2p\equiv0$ on the axis, the axial dynamics is the Constantin--Lax--Majda system \cite[Theorem~1.2]{Chae2008}; this is the case $\pi_r\equiv0$ of Section~\ref{sec:mechanism}. In \cite{Chae2010} he extended this structure off the axis. Following the idea of Constantin, Lax and Majda \cite{CLM1985}, he introduced $Z_+:=v^r/r-i\,v^\theta/r$ and showed that it satisfies $\frac{\mathrm D}{\mathrm Dt}Z_+=-Z_+^2-\frac1r\partial_rp$ along all particle trajectories in $\R^3\setminus\{r=0\}$ \cite[(1.8)]{Chae2010}, with the limit $Z_0=\partial_rv^r-i\,\partial_rv^\theta$ on the axis \cite[(1.14)]{Chae2010}. In our variables $Z_+=-(v+iu)$, so with $\zeta:=v+iu$ this is
\begin{equation}\label{eq:complex-Riccati}
\tfrac{\mathrm D}{\mathrm Dt}\zeta=\zeta^2+\tfrac1r\partial_rp ,
\end{equation}
i.e.\ the first two equations of $(E2)$: a complex Riccati equation except for the pressure term, which is real. For $m=1$ the coefficient of $uv$ is $\frac12$ instead of $2$, and there is no such complex form.

From \eqref{eq:complex-Riccati} Chae obtained two results off the axis. First \cite[Theorem~1.1]{Chae2010}: if at a point $a$ with $r(a)>0$ one has $v^\theta_0(a)=0$, $v^r_0(a)<0$ and $\partial_rp_0(a)\ge0$, and $\partial_rp\ge0$ along the trajectory $X(a,t)$ up to the time $-r(a)/v_0^r(a)$, then there is no global classical solution. In particular a pressure that is nondecreasing in the radial direction for all time is not consistent with global regularity \cite[Remark~1.1]{Chae2010}. Second \cite[Theorems~1.2,~1.3]{Chae2010}: if $v_0^\theta(a)=0$, then either $|v^r/r|$ blows up in finite time along $X(a,t)$, or the quantity $(v^r/r)^2+\frac1r\partial_rp$ tends to zero along a sequence $t_n\to\infty$, or it is smaller than any Osgood function of $-v^r/r$ along such a sequence. The second result rests on an elementary dichotomy for scalar ODEs \cite[Lemma~2.1]{Chae2010}. In both results the swirl vanishes along the trajectory, because $v^\theta/r$ is transported with an exponential factor, and $-v^r/r$ obeys the real Riccati equation $Y'=Y^2+\frac1r\partial_rp$.

The following proposition combines Chae's swirl-free criterion with the one-sided criterion of Appendix~\ref{app:ABC}, along an arbitrary trajectory.

\begin{proposition}[Blow-up along a particle trajectory]\label{prop:trajectory}
Let $m\in\{1,2\}$, $a=b=\lambda=1$, and let $(u,v,g,p)$ be a smooth even solution of $(Em)$ on $[0,T^*)$. Let $X(t)$ be a particle trajectory in the meridian plane, and write $V(t)=v(t,X(t))$, $U(t)=u(t,X(t))$ and $Q(t)=\bigl(\tfrac1r\partial_rp\bigr)(t,X(t))$; for $m=1$ assume that $u^2=\vartheta/x_2>0$ along $X(t)$ whenever $U(0)\ne0$.
\begin{enumerate}[label=(\roman*)]
\item \textup{(Swirl-free trajectory.)} If $U(0)=0$, $V(0)>0$ and $Q(t)\ge0$ for $t<\min\{T^*,1/V(0)\}$, then $T^*\le1/V(0)$.
\item \textup{(Swirling trajectory.)} If $U(0)>0$ and there are $\delta>0$, $C>1$ with
\begin{equation}\label{eq:one-sided-traj}
Q(t)\ge(1+\delta)U(t)^2-C\,V(t)^2\qquad\text{for all }t\in[0,T^*),
\end{equation}
and $V(0)>-\sqrt{\bar\kappa}\,U(0)$ with $\bar\kappa=2\delta/(2C-2+m^2)$, then $T^*\le\bar T<\infty$, where $\bar T$ is the blow-up time of $\bar v'=(1-C)\bar v^2+\delta\bar u^2$, $\bar u'=\frac12m^2\bar u\bar v$ with the same data.
\end{enumerate}
\end{proposition}

\begin{proof}
Along $X(t)$, $(Em)$ gives $V'=V^2-U^2+Q$ and $U'=\frac12m^2UV$. In case (i), $U\equiv0$, so $V'\ge V^2$ and $V(t)\ge V(0)/(1-V(0)t)$. In case (ii), $V'\ge(1-C)V^2+\delta U^2$, and the comparison argument of Theorem~\ref{thm:conditional} applies. In both cases $V$ is unbounded on $[0,T^*)$ if $T^*$ exceeds the stated bound. This is impossible for a smooth solution, because $|V|=|v^r/r|\le\|\partial_rv^r\|_{L^\infty}$ (and $|u_2/x_2|\le\|\partial_2u_2\|_{L^\infty}$ for $m=1$), since $v^r$ vanishes on the axis.
\end{proof}

Part (i) is Chae's theorem \cite[Theorem~1.1]{Chae2010}, written in our variables (the time bound $-r(a)/v_0^r(a)$ is $1/V(0)$), together with its analogue on the axis \cite[Theorem~1.1]{Chae2008} and its analogue for $m=1$. It is Chae's mechanism: on a swirl-free trajectory any radial increment of the pressure, $Q\ge0$, adds to the inertia of the strain, and the blow-up is of CLM type. Part (ii) is the mechanism of this paper: on a swirling trajectory the radial pressure increment has to overcome the centrifugal force (the buoyancy force for $m=1$) by a fixed fraction, $Q\ge(1+\delta)U^2$, up to a multiple of $V^2$. At the origin, where $Q=P_{rr}(t,0)$, part (ii) is Proposition~\ref{prop:one-sided}, and it contains Theorem~\ref{thm:conditional} (Remark~\ref{rem:one-sided}).

Chae's dichotomy \cite[Theorems~1.2,~1.3]{Chae2010} also extends to swirling trajectories, because \cite[Lemma~2.1]{Chae2010} applies to any scalar equation $y'=f(y,t)$ with $f$ continuous and locally Lipschitz in $y$.

\begin{corollary}[Chae's dichotomy with swirl]\label{cor:Chae-swirl}
Let $m\in\{1,2\}$, $a=b=\lambda=1$, and let $(u,v,g,p)$ be a smooth even solution of $(Em)$ that exists for all $t\ge0$. Let $X(t)$, $V$, $U$, $Q$ be as in Proposition~\ref{prop:trajectory}, and let $F:=V^2-U^2+Q$ be the radial acceleration of the strain along $X(t)$. Then at least one of the following holds:
\begin{enumerate}[label=(\roman*)]
\item there is a sequence $t_n\to\infty$ with $F(t_n)\to0$;
\item there is a sequence $t_n\to\infty$ such that $F(t_n)/g(V(t_n))\to0$ for every positive continuous $g$ with $\int_1^\infty ds/g(s)<\infty$.
\end{enumerate}
\end{corollary}

\begin{proof}
Along $X(t)$, $V'=f(V,t)$ with $f(y,t)=y^2-U(t)^2+Q(t)$, where $U$ and $Q$ are continuous on $[0,\infty)$. The function $f$ is continuous and locally Lipschitz in $y$. Since the solution is global and smooth, $V$ does not blow up in finite time, and \cite[Lemma~2.1]{Chae2010} gives (i) or (ii).
\end{proof}

For $U\equiv0$ this is \cite[Theorem~1.2]{Chae2010}, and on the axis \cite[Theorem~1.3]{Chae2010}. In words: for a global solution, along every trajectory the net radial forcing $F=V^2-U^2+Q$, which contains the centrifugal force $-U^2$, must come arbitrarily close to zero infinitely often, or be dominated by every Osgood function of $V$ along a sequence of times. In the swirl-dominated blow-up of Theorem~\ref{thm:main}, by contrast, $F=v_a'\sim2(T-t)^{-2}/m^2$ at the origin.

\subsection{What remains to be proved}\label{subsec:remains}
To obtain a finite-time blow-up of $(E2)$, respectively $(E1)$, from smooth data of finite energy along these lines, one needs:
\begin{enumerate}[label=(\roman*)]
\item \emph{Local theory.} Local existence and the Beale--Kato--Majda criterion for smooth axisymmetric solutions with swirl, respectively for the Boussinesq system, preserving the symmetry of Remark~\ref{rem:symetric-in-rz}. This is classical \cite{BKM984,MB2002,CD1980,ChaeNam1997,Tan2002}.
\item \emph{Control of the pressure anisotropy at the origin} (the main step). One must find smooth finite-energy data for which \eqref{eq:anisotropy-window}, or another condition that keeps \eqref{eq:F0def} blowing up, holds for as long as the solution is smooth. Since $\pi_r(t)$ is a nonlocal functional of the velocity field, this requires showing that near the singular time the pressure Hessian at the origin is dominated by the flow in a shrinking neighborhood of the origin, where the solution should look like the stagnation-point (Yan-type) flow of Proposition~\ref{prop:Yan-apex}. The far field must contribute $o(v_a^2)$ uniformly up to the singular time.
\item \emph{A localized blow-up profile.} In practice (ii) amounts to constructing an approximate self-similar solution that agrees with a Yan-type flow near the origin and decays at infinity, and proving its stability, in the spirit of \cite{GigaKohn1985,EggersFontelos,CHH21,CH2022}. The ray models of this paper are the one-dimensional shadows of this problem. With the closure $F_0\equiv1$ they show that the inner self-similar layer and bounded energy are compatible with apex blow-up (Sections~\ref{sec:selfsim}, \ref{sec:family}).
\item \emph{Energy.} Finite energy is a property of the full solution; the ray energies \eqref{eq:unified-at-theta-0-b-Rm}, \eqref{eq:unified-at-theta-0-b-Zm} are proxies obtained from the ridge-flat structure and are not substitutes for it.
\end{enumerate}
Item (ii) is where the difficulty of the blow-up problem lies: by Proposition~\ref{prop:exact-apex}, controlling the anisotropy of the pressure Hessian at the origin is equivalent to closing the apex dynamics. The rays enter only through (iii): they carry the information on how the flow near the origin, which determines $\pi_r(t)$, is organized. Nothing in this argument depends on $m$ beyond the coefficients, so (i)--(iv) apply to $(E1)$ and $(E2)$ alike; for $(E1)$ one must also handle the sign of $\vartheta/x_2$ at the origin (Remark~\ref{rem:conditional}). For scale-invariant solutions the apex dynamics depends on the direction at the origin, and the free anisotropy is replaced by a solvability condition; see \cite{ShiIII}.

\section{The physical mechanism of the apex blow-up}\label{sec:mechanism}

Theorem~\ref{thm:main} and the Constantin--Lax--Majda breakdown criterion (Theorem~\ref{thm:CLM-2}) both describe the finite-time blow-up of a quadratic system for a pair $(v,u)$. In this section we show that the two mechanisms are opposite, and we give the physical picture of the blow-up in Theorem~\ref{thm:main}. The comparison matters for the history of this work. A natural expectation is that the 3D axisymmetric Euler equations, if they blow up at a stagnation point, do so in the manner of the CLM solution, with $u_a=0$ and $v_a>0$ at the blow-up point. Earlier versions of this preprint supported this expectation: they used the isotropic closure $f\equiv0$ (Remark~\ref{rem:f0f1-roots}), which leads to CLM-$q$ equations with $(m,q)=(1,2)$ and $(2,3)$, and hence to blow-up of CLM type. The closure $f\equiv2$ reverses this structure.

Throughout this section $a=b=\lambda=1$, so that all quantities are physical, and $\mu(t)=-P_{zz}(t,0)/P_{rr}(t,0)$ denotes the anisotropy of the pressure Hessian at the origin, as in Section~\ref{sec:exact}. By Proposition~\ref{prop:scaling-meaning}, the apex system \eqref{eq:apexODE} with parameters $(\lambda,a,b)$ is the case $\mu\equiv b/a$, up to the unit of time.

\subsection{The apex dynamics as the zeroth-order ray dynamics}
The apex system is obtained from $(Rm)$ and $(Zm)$ by setting $\Dx=x\partial_x=0$. With $y=1/x$ one has $\Dx=-y\partial_y$, so $\Dx$ is the generator of dilations of the whole ray $x\in[0,\infty)$, and it annihilates exactly the functions that are constant on the ray. The apex system is therefore the zeroth-order approximation of the ray dynamics: it is the dynamics of the dilation-invariant component of the solution. It is exact for data that are constant on $(0,R_0]$, on the core of Theorem~\ref{thm:core}, and it is the leading term of the expansion in the power modes $x^n$ of Section~\ref{sec:selfsim}. The terms of order $\Dx$ describe how this component is transported along the ray and how it is coupled to the rest of the ray through $\K_k$.

\subsection{Momentum balance at the origin}
Near the origin the velocity is $v^r\simeq-rv_a$ and $v^z\simeq zmv_a$, so $v_a>0$ means radial compression and axial extension, and $v_a<0$ the reverse. For $m=2$, $v^\phi\simeq ru_a$ is a solid-body rotation with angular velocity $u_a$. For $m=1$, $u_a^2=\partial_{x_2}\vartheta(t,0)$ is the horizontal buoyancy gradient at the origin, and the buoyancy plays the role of the centrifugal force (Section~\ref{sec:derivation}).

\begin{proposition}[Pressure at the origin]\label{prop:pressure-origin}
Let $(u,v,g,p)$ be a smooth even solution of $(Em)$, and suppose that $\mu(t)\ne m$. Then
\begin{gather}
v_a'=v_a^2-u_a^2+\pi_r,\qquad m\,v_a'=-m^2v_a^2-\pi_z,\label{eq:momentum-origin}\\
\pi_r=\frac{m\bigl(u_a^2-(m+1)v_a^2\bigr)}{m-\mu},\qquad \pi_z=-\mu\,\pi_r,\label{eq:pressure-origin}
\end{gather}
where $\pi_r=P_{rr}(t,0)$ and $\pi_z=P_{zz}(t,0)$. Consequently
\begin{equation}\label{eq:apex-mu}
v_a'=A(\mu)v_a^2-B(\mu)u_a^2,\qquad A(\mu)=\frac{m^2+\mu}{\mu-m},\qquad B(\mu)=\frac{\mu}{\mu-m}.
\end{equation}
\end{proposition}

The radial balance in \eqref{eq:momentum-origin} has three forces. The term $v_a^2$ is the inertia of the straining motion, $-u_a^2$ is the centrifugal force (the buoyancy force for $m=1$), and $\pi_r$ is the radial pressure force. By \eqref{eq:pressure-origin} the pressure force consists of a response to the swirl, $mu_a^2/(m-\mu)$, and a response to the strain, $-m(m+1)v_a^2/(m-\mu)$. The net coefficients in \eqref{eq:apex-mu} are
\[
-B(\mu)=-1+\frac{m}{m-\mu}=\frac{\mu}{m-\mu},\qquad A(\mu)=1-\frac{m(m+1)}{m-\mu}.
\]

\begin{proof}
The relations \eqref{eq:momentum-origin} are \eqref{eq:exact-apex} with $a=b=\lambda=1$. Inserting $\pi_z=-\mu\pi_r$ and eliminating $v_a'$ gives $(m-\mu)\pi_r=mu_a^2-m(m+1)v_a^2$, which is \eqref{eq:pressure-origin}; this is the Poisson equation at the origin. Substituting $\pi_r$ into the first relation gives \eqref{eq:apex-mu}.
\end{proof}

Two regimes follow.
\begin{itemize}
\item If $\mu<0$ or $\mu>m$, then $B(\mu)>0$: the swirl decelerates the compression, as the centrifugal force does when the pressure is ignored. This is the CLM structure.
\item If $0<\mu<m$, then $B(\mu)<0$: the pressure response to the swirl exceeds the centrifugal force by the factor $m/(m-\mu)>1$, so the swirl accelerates the compression. Moreover $A(\mu)<0$, since $-m^2<\mu<m$: the pressure response to the strain exceeds the inertia, so pure strain decays. This is the regime of Theorem~\ref{thm:main} and of Theorem~\ref{thm:conditional}.
\end{itemize}
The CLM system \eqref{eq:CLM2} is \eqref{eq:momentum-origin} with $\pi_r\equiv0$, i.e.\ with the pressure ignored; Chae \cite[Theorem~1.2]{Chae2008} showed that this is exactly the axial dynamics of the axisymmetric Euler equations when $\partial_r^2p\equiv0$ on the axis.

\subsection{Comparison with the Constantin--Lax--Majda blow-up}
Table~\ref{tab:CLM-vs-main} compares Theorem~\ref{thm:CLM-2} with Theorem~\ref{thm:main}.

\begin{table}[ht]
\centering\small
\begin{tabular}{>{\raggedright\arraybackslash}p{2.5cm}>{\raggedright\arraybackslash}p{4.2cm}>{\raggedright\arraybackslash}p{4.4cm}}
\toprule
 & CLM (Theorem~\ref{thm:CLM-2}) & Theorem~\ref{thm:main} \\
\midrule
Equations & $u_t=2uv$, $v_t=v^2-u^2$ & $\lambda v'=Av^2-Bu^2$, $\lambda u'=\frac12m^2uv$, with $A,B<0$ \\
Where & at every point $x$ independently & at the apex $x=0$; the rest of the ray is coupled to it through $\K_k$ \\
Blow-up condition & $u_0=0$ and $v_0>0$ at some point & $u_c\neq0$ and $v_c>-\sqrt\kappa|u_c|$, so a moderate initial expansion is allowed (Remark~\ref{rem:data-velocity}); no blow-up on $\{u_0=0,\ v_0\ge0\}$ \\
Strain at the blow-up time & radial compression, axial extension ($v\to+\infty$) & the same ($v_a\to+\infty$) \\
What blows up & $v$ alone; $u\equiv0$ at the blow-up point & $v$ and $u$ together, $|u|/v\to1/\sqrt\kappa$ \\
Role of the swirl & opposes the blow-up & drives the blow-up \\
Statement & exact blow-up at $T=1/M$ & $T^*\le\Tap$, with finite energy \\
\bottomrule
\end{tabular}
\caption{The CLM blow-up and the apex blow-up of Theorem~\ref{thm:main}.}\label{tab:CLM-vs-main}
\end{table}

Two further differences are structural.
\begin{enumerate}[label=(\roman*)]
\item In the CLM-$q$ family of Proposition~\ref{prop:CLMq}, the CLM system has $Q=1/9>0$. The windows \eqref{eq:windows} of $(Rm)$ correspond to $q<-2/m$ (Remark~\ref{rem:q-symmetry}), hence to $Q<0$, so the cubic term of $v_{tt}=vv_t-Qv^3$ has the opposite sign.
\item The CLM solution rests on the complex variable $z=v+iu$ with $z_t=z^2$. For \eqref{eq:apexODE}, $z=v+i\beta u$ satisfies $\lambda z'=cz^2$ if and only if $c=A=m^2/4$ and $\beta^2=B/A>0$. In the windows of $(Rm)$, $A<0$, so no such variable exists; this is why the apex dynamics requires the classification of Theorem~\ref{thm:apexclass} rather than a single explicit formula.
\end{enumerate}
The two mechanisms share only the quadratic structure and the rate $v\sim c/(T-t)$.

\subsection{The physical picture}

\begin{lemma}\label{lem:swirl-dominance}
Let $0<\mu<m$ be constant and let the apex values blow up in case (b1) of Theorem~\ref{thm:apexclass} or on the separatrix. Then $u_a^2/v_a^2\to1/\kappa$, and
\[
\frac1\kappa-(m+1)=\frac{m(m+2)(m-\mu)}{2\mu}>0 .
\]
Hence $\pi_r(t)>0$ and $\pi_z(t)<0$ for $t$ close to $\Tap$.
\end{lemma}

\begin{proof}
The limit of $u_a^2/v_a^2$ is Remark~\ref{rem:Yan-rate}. With $A=A(\mu)$, $B=B(\mu)$ one computes $1/\kappa=(2A-m^2)/(2B)=\bigl(m^2(m+2)+\mu(2-m^2)\bigr)/(2\mu)$, which gives the identity. The signs follow from \eqref{eq:pressure-origin}.
\end{proof}

For the main example (Example~\ref{ex:main}: $m=2$, $\mu=1/6$) one has $1/\kappa=47$: near the blow-up time the fluid at the origin rotates about $\sqrt{47}\approx6.9$ times faster than it is strained. The blow-up of Theorem~\ref{thm:main} can be described as follows.
\begin{enumerate}[label=(\roman*)]
\item \emph{Kinematics.} The flow near the origin is a swirling stagnation-point flow: radial inflow $v^r=-rv_a$, axial outflow $v^z=zmv_a$, and, for $m=2$, a rotation $v^\phi=ru_a$ that dominates the strain. For $m=1$ the rotation is replaced by a strong horizontal buoyancy gradient.
\item \emph{A low-pressure core.} By Lemma~\ref{lem:swirl-dominance}, the pressure has a saddle at the origin: it has a minimum in the radial direction ($\pi_r>0$) and a maximum along the axis ($\pi_z<0$). The radial pressure force $-\pi_r r$ points inwards and exceeds the centrifugal force $u_a^2r$, and the axial pressure force $-\pi_z z$ pushes the fluid out along the axis. Both accelerate the compression and the extension.
\item \emph{Spin-up.} For $m=2$, the angular momentum $rv^\phi=r^2u$ is conserved along particle paths. As the fluid moves inwards at the rate $v_a$, $u_a$ grows according to $u_a'=2u_av_a$. For $m=1$, the buoyancy $\vartheta$ is conserved, and $\vartheta/x_2$ grows as $x_2$ shrinks, according to $u_a'=\frac12u_av_a$.
\item \emph{Suppression of pure strain.} Since $A(\mu)<0$, the strain alone would decay: the pressure response to the strain outweighs its inertia.
\item \emph{Feedback.} Swirl produces the low-pressure core; the core drives the radial inflow and the axial outflow; the inflow spins up the swirl. Both $v_a$ and $u_a$ grow like $(\Tap-t)^{-1}$, with the fixed ratio $|u_a|/v_a\to1/\sqrt\kappa$.
\end{enumerate}
In short, the blow-up of Theorem~\ref{thm:main} is a swirl-driven collapse of a rotating stagnation point (for $m=1$, a buoyancy-driven collapse), sustained by the low-pressure core at the origin. The CLM blow-up is the inertial collapse of a non-rotating stagnation point, which rotation opposes. Both blow-ups end with radial compression and axial extension at the origin. They differ in what drives the collapse, and in Theorem~\ref{thm:main} the swirl can even overturn an initially expanding strain (Remark~\ref{rem:data-velocity}).

\subsection{The swirl-free case}\label{subsec:swirl-free}
If $u\equiv0$, then $(Em)$ reduces to the 2D incompressible Euler equations for $m=1$ (since $\vartheta\equiv0$) and to the axisymmetric Euler equations without swirl for $m=2$. Both are globally regular for smooth data (see, e.g., \cite{MB2002}). A blow-up mechanism for $(Em)$ should therefore be absent when $u_a\equiv0$.
\begin{itemize}
\item In the CLM model ($\pi_r\equiv0$), $u\equiv0$ gives $v_t=v^2$, which blows up for $v_0>0$. This pure-strain collapse cannot occur in 2D Euler or in swirl-free axisymmetric Euler: there, the pressure prevents it.
\item In the apex dynamics with $0<\mu<m$, $u_a\equiv0$ gives $v_a'=A(\mu)v_a^2$ with $A(\mu)<0$. For $v_a(0)\ge0$ (compression and extension) $v_a$ decays like $1/(v_a(0)^{-1}+|A|t)$ and there is no blow-up, consistently with the regularity results. Blow-up requires $u_a\neq0$ (Theorem~\ref{thm:apexclass}(a), Remark~\ref{rem:windowcase}). This is consistent with Chae's criterion (Proposition~\ref{prop:trajectory}(i)): with $u_a=0$ and $v_a>0$, blow-up would follow from $\pi_r\ge0$, whereas in the window \eqref{eq:pressure-origin} gives $\pi_r=-m(m+1)v_a^2/(m-\mu)<0$.
\item For $v_a(0)<0$ (radial expansion and axial compression), the same equation gives $v_a\to-\infty$ at $t=1/(A(\mu)v_a(0))$. This is impossible for smooth solutions of 2D Euler and of swirl-free axisymmetric Euler. It does occur for data of limited regularity: Elgindi \cite{Elgindi21} constructed $C^{1,\alpha}$ solutions without swirl that blow up in finite time, and the singular point is described in \cite{Shkoller2026} as a stagnation point on the axis with compressive axial strain. The anisotropy for such data is studied in the companion paper \cite{ShiII}. Hence, for smooth swirl-free solutions with $v_a<0$, the anisotropy $\mu(t)$ cannot stay in $(0,m)$: the pressure Hessian adjusts so as to prevent the collapse. This is a concrete instance of the fact that $\mu(t)$ is determined by the whole flow (Section~\ref{sec:exact}). It also shows that a closure with a fixed anisotropy in $(0,m)$ is appropriate in the swirl-dominated regime of Lemma~\ref{lem:swirl-dominance}, which is the regime of Theorems~\ref{thm:main} and \ref{thm:conditional} (both assume $u_a(0)\neq0$), and not for swirl-free flows.
\end{itemize}

\subsection{The role of the pressure closure}
The comparison shows that the pressure at the origin, through its anisotropy $\mu$, decides whether and how the apex blows up. The isotropic closure $f\equiv0$ of the earlier versions corresponds in the rescaled variables to $P_{zz}/P_{rr}=b/a>0$, i.e.\ to $\mu<0$, hence to $B(\mu)>0$: the swirl opposes the compression, and blow-up is driven by the inertia of the strain, as in CLM. The closure $f\equiv2$ corresponds to $\mu=b/a>0$, and in the windows of $(Rm)$ to $0<\mu<m$, where the roles are reversed. The exact Yan-type solutions single out the second case: by Proposition~\ref{prop:Yan-apex} every Yan-type solution has $\mu(b_{00})>0$, so no Yan-type solution has an isotropic pressure Hessian at the origin, and $\mu(b_{00})<m$ if and only if $b_{00}^2>m+1$, which is the swirl dominance of Lemma~\ref{lem:swirl-dominance} with $b_{00}^2=1/\kappa$. Thus the closure realized by the exact stagnation-point solutions is the one under which the swirl drives the blow-up.

\section{Comparison with related blow-up mechanisms}\label{sec:compare}

Our results concern the ray systems $(Rm)$ and $(Zm)$ with the closure $f\equiv2$; apart from the conditional Theorem~\ref{thm:conditional}, we make no statement about the 2D Boussinesq or 3D Euler equations themselves.

\subsection{Stagnation-point similitude and Riccati blow-up}
For the 2D Euler equations, Childress, Ierley, Spiegel and Young \cite{CISY} found closed-form solutions of stagnation-point form that blow up in finite time; Gibbon, Fokas and Doering \cite{GFD} introduced stretched solutions of the 3D equations with this structure, Constantin \cite{Constantin00} showed that the resulting nonlocal Riccati equations give an infinite-dimensional family of exact Euler solutions with infinite energy that blow up in finite time, and Sarria and Wu \cite{SW} proved blow-up for stagnation-point form solutions of the inviscid 2D Boussinesq equations. The apex dynamics \eqref{eq:apexODE} is of this Riccati type, and the exact cores of Theorem~\ref{thm:core} are stagnation-point type solutions restricted to a region bounded by a characteristic. The difference is that the cores are embedded in solutions of finite energy; along $(Rm)$ they carry bounded energy into the apex (Corollary~\ref{cor:core}).

\subsection{One-dimensional nonlocal models}
Constantin, Lax and Majda \cite{CLM1985} proposed $\omega_t=\omega H\omega$ and solved it exactly; De Gregorio \cite{DeGregorio1990} added advection, Okamoto, Sakajo and Wunsch \cite{OSW} weighted advection against stretching, and Chen, Hou and Huang \cite{CHH21} proved asymptotically self-similar blow-up for the De Gregorio model; see also Schochet \cite{Schochet1986} for the viscous CLM model and Chae, Constantin and Wu \cite{CCW2014} for an incompressible 2D model with explicit blow-up solutions related to the Boussinesq equations. The apex system \eqref{eq:apexODE} is equivalent to the CLM-$q$ equation \eqref{eq:CLM-q} (Proposition~\ref{prop:CLMq}), with $q=2$ for the CLM system itself. In $(Rm)$ and $(Zm)$ the nonlocal operator is not a Hilbert transform but the one-sided Hardy average $\K_k$; it is explicitly invertible, it makes the apex values close exactly, and its direction (towards the apex for $k>0$, towards infinity for $k<0$) decides between a flux at infinity and a loss of smoothness at the apex (Section~\ref{sec:energy}).

\subsection{The Hou--Luo scenario and its models}
Luo and Hou \cite{LH14a,LH14b} gave numerical evidence for a singularity of the 3D axisymmetric Euler equations on the boundary of a periodic cylinder. One-dimensional models of the boundary dynamics blow up from smooth data \cite{CKY2015,CHKLSY2017}, asymptotically self-similar blow-up of the Hou--Luo model was proved in \cite{CHH22}, and Chen and Hou \cite{CH2021,CH2022,CH23,CH25} proved blow-up for the 2D Boussinesq and 3D axisymmetric Euler equations with $C^{1,\alpha}$ velocity and with smooth data in the presence of a boundary. Elgindi and Jeong \cite{EJ2019,EJ2020} and Elgindi \cite{Elgindi21} proved singularity formation for strong solutions and for $C^{1,\alpha}$ solutions without boundary; see \cite{DE2023} for a survey. Along the symmetry axis, Hou and Li \cite{HLi2006} used the variables underlying $(Em)$ to study the dynamic depletion of vortex stretching, and the explicit self-similar solutions of Yan \cite{Yan2018} blow up at the origin with infinite energy.

Hou \cite{Hou2023} gave numerical evidence for a singularity of the 3D axisymmetric Euler equations at the origin, away from any boundary. His data are odd in $z$ for $u_1=u$ and $\omega_1$. In our notation this means $u_a\equiv0$: the swirl vanishes on the equatorial plane, and the maximum of $u$ approaches the origin from outside the axis. That it does not reach the origin before the singular time is what one expects of a self-similar profile whose maximum is not at its center. In the companion paper \cite{ShiII} we follow Hou's data in compactified variables of the meridian plane. The strain at the origin is a radial compression with axial extension that grows by the swirl-free mechanism of Proposition~\ref{prop:trajectory}(i), with a radial pressure Hessian that brakes it only partially, while the swirl closes in from a ring. Our scenario belongs to the other symmetry class: the swirl is even in $z$, largest at the origin, and drives the strain through the suction $\alpha>1$. The conditions $\partial_ru_1=\partial_r\omega_1=0$ imposed on the axis in such computations are consequences of smoothness and express that $u_1$ and $\omega_1$ are even in $r$. They prescribe no values on the axis, and a computation on $-1\le r\le1$ with even data and a parity-preserving scheme gives the same solution. In the compactified variables of \cite{ShiII} no condition is needed at all.

The apex of $(Rm)$ and $(Zm)$ plays the role of the singular point in these scenarios. It is invariant, the transport speed vanishes there, and the apex values obey a closed ODE. The differences from the works above are the following.
\begin{enumerate}[label=(\roman*)]
\item \emph{Exact closure at the apex.} The apex dynamics is the same ODE on both rays, consistent with the apex relation; its blow-up time $\Tap$ is explicit and bounds the lifespan (Corollary~\ref{cor:Tstar}).
\item \emph{Direction of the characteristics.} With the same apex dynamics, the blowing-up core is carried into the apex along $(Rm)$ and away from it along $(Zm)$ (Corollary~\ref{cor:core}). This is a mechanism for focusing at the apex that is absent from the scalar one-dimensional models.
\item \emph{Energy as a selection principle.} Finite energy is not automatic: for $0<k\le(m+3)/2$ the flux at infinity makes $E_m$ infinite immediately (Corollary~\ref{cor:energy}), and the finite-energy windows \eqref{eq:windows} are explicit.
\item \emph{Rigor.} The structure, the apex dynamics, the bound $T^*\le\Tap$, the exact cores and the energy criterion are proved; the equality $T^*=\Tap$ for $(R2)$ is so far supported by numerics only.
\end{enumerate}

\subsection{Degenerate boundaries and shock formation}
Blow-up at a degenerate boundary point is known for semilinear parabolic equations \cite{Floater}. For first-order hyperbolic systems the typical singularity is a gradient catastrophe caused by crossing characteristics \cite{Alinhac,LXY}; self-similar shock formation for nonlocal Burgers equations was studied in \cite{CMP}. In $(Rm)$ and $(Zm)$ both kinds of singularity compete: amplitude blow-up at the apex, driven by \eqref{eq:apexODE}, and gradient blow-up away from the apex, driven by the quadratic terms in the equation for $\Dx v$. The numerical results of Section~\ref{sec:num} indicate that the first wins for $(R2)$ in its window and the second for $(R1)$ in its window; for $(Zm)$ in its windows the computations do not resolve the competition.

\begin{table}[ht]
\caption{Qualitative comparison.}\label{tab:compare}
{\footnotesize\begin{tabular}{lll}
\toprule
setting & energy & method\\
\midrule
stagnation-point form \cite{CISY,Constantin00,SW} & infinite & explicit Lagrangian formulas\\
CLM and De Gregorio models \cite{CLM1985,DeGregorio1990,CHH21} & -- & exact solution; dynamic rescaling\\
Hou--Luo 1D models \cite{CKY2015,CHKLSY2017,CHH22} & -- & 1D analysis; dynamic rescaling\\
2D Boussinesq, 3D Euler \cite{CH2022,CH23} & finite & computer-assisted stability\\
Yan-type solutions \cite{Yan2018} & infinite & explicit formulas\\
$(Rm)$, $(Zm)$, $f\equiv2$, this paper & finite ($E_m$) & apex ODE, cores; numerics\\
\bottomrule
\end{tabular}}
\end{table}

\section{Conclusion}\label{sec:concl}

We corrected the axis reduction of the unified system $(Em)$: on the symmetry axes the angular pressure terms produce $P_{\theta\theta}$, which we closed by $x^{-2}P_{\theta\theta}=-fP_{xx}$, and consistency at the apex forces $f(t,0)\in\{0,2\}$, where $0$ corresponds to the reduction of the earlier versions of this preprint and $2$ is singled out by the Yan-type solution. With $f\equiv2$ the ray systems $(Rm)$ and $(Zm)$ have the following structure.
\begin{itemize}
\item The time derivative enters through $\alpha+\beta\Dx$, and inverting it gives nonlocal transport systems with the Hardy-type operator $\K_k$, $k=\alpha/\beta$.
\item The apex values obey the same ODE system on both rays. We classified its solutions; with $\lambda=m^2(q+1)$, $b/a=m^2(mq+2)/(m^2q-2)$ it is the CLM-$q$ equation, explicitly solvable for $m=2$, $q=2$. On the separatrix the blow-up rate is $2\lambda/(m^2(T-t))$, and the Yan-type solution with $b_{00}^2=1+\frac12m^2(m+1)$ is exactly the separatrix solution.
\item Both systems are locally well posed, and $T^*\le\Tap$.
\item For $k>0$, data constant near the apex produce an exact core. The core shrinks to the apex along $(Rm)$, and its energy tends to zero; it expands along $(Zm)$, where its energy blows up.
\item $E_m(t)$ is finite if and only if $k>(m+3)/2$ or the flux $L(t)$ vanishes. This selects explicit windows of the scaling parameters, in which smooth compactly supported data of finite energy lose smoothness no later than $\Tap$ while $E_m$ stays finite before the singular time.
\end{itemize}
\begin{itemize}
\item Near the apex the solution develops a self-similar inner layer on the scale $(\Tap-t)^{\gamma_2}$, whose profile solves three autonomous ODEs; the profile exists for $(R2)$ and, outside its finite-energy window, for $(R1)$ with $a/b>3.18$; in the window $2b<a<8b$ of $(R2)$ it corresponds to blow-up at the apex at $\Tap$ with bounded energy.
\end{itemize}
\begin{itemize}
\item The apex blow-up is driven by the swirl (the buoyancy for $m=1$) through a low-pressure core at the origin, whose suction exceeds the centrifugal force; it requires $u_a\ne0$, and without swirl compressive strain decays. This is the opposite of the Constantin--Lax--Majda blow-up, which occurs where $u=0$ and which the swirl opposes (Section~\ref{sec:mechanism}).
\end{itemize}
Numerically, in the window of $(R2)$ the singularity forms at the apex at the apex time, $E_2$ stays bounded, and the remaining discrepancy of relative size $10^{-5}$ is explained by a short-wave instability seeded by rounding errors.

In the class of scale-invariant solutions the anisotropy at the origin is no longer a free function; this class is studied in the sequel \cite{ShiIII}.

\subsection*{Scope of the model: bulk, rays and apex}
The program behind this paper can be summarized in four tentative steps: (A) the ray dynamics is independent of the bulk; (B) the apex dynamics is independent of the ray dynamics; (C) the apex dynamics blows up in finite time; (D) the apex drives the rays at $\theta=0,\pm\frac\pi2,\pi$, and the rays drive the bulk. Our results qualify these steps as follows.
\begin{itemize}
\item[(A)] On the axes the velocity terms close exactly by symmetry (Remark~\ref{rem:flat-preserved}), but the pressure does not: $P_{\theta\theta}$ on the axis is determined by the whole flow through the pressure Poisson equation, whose restriction to the axis is not independent. With two arbitrary closure functions $f_0,f_1$ the reduction is exact but not closed (Section~\ref{sec:exact}); the rays are closed only under a prescribed closure such as $f\equiv2$, a modeling hypothesis comparable to the replacement of the Biot--Savart law in one-dimensional models \cite{CLM1985,CKY2015,CHKLSY2017}. By the symmetry assumption of Remark~\ref{rem:symetric-in-rz}, the four rays reduce to two, $(Rm)$ and $(Zm)$. Justifying the closure, for example by an expansion of the pressure equation near the axes, is the main missing link; the Yan-type solution satisfies it exactly with $f\equiv2$ (Proposition~\ref{prop:Yan-closure}).
\item[(B)] The exact apex dynamics (Proposition~\ref{prop:exact-apex}) contains exactly one free scalar function, the anisotropy $P_{zz}(t,0)/P_{rr}(t,0)$ of the pressure Hessian at the origin; its trace is fixed locally by the Poisson equation, and the consistency of the two rays at the apex holds automatically (Proposition~\ref{prop:F0F1}). The apex dynamics is closed once this anisotropy is prescribed. The closure $f\equiv2$ prescribes $P_{zz}/P_{rr}=-b/a$ (Proposition~\ref{prop:scaling-meaning}), and Yan-type solutions realize it exactly, with infinite energy (Proposition~\ref{prop:Yan-apex}).
\item[(C)] Proved: Theorem~\ref{thm:apexclass} and Corollary~\ref{cor:Tstar}.
\item[(D)] The direction of influence depends on the sign of $k$: for $k>0$ the value at $x$ depends only on $(0,x]$, for $k<0$ only on $[x,\infty)$ (apart from transport), and the characteristics move towards the apex along $(Rm)$ and away from it along $(Zm)$. In the finite-energy scenarios found here, $(R2)$ (Section~\ref{sec:selfsim}) and $(Zm)$ with $\kZ<0$ (Section~\ref{sec:family}), the singularity stays at the apex and the rest of the ray stays bounded. When the apex blow-up does spread along the whole ray ($(Rm)$ with $b>am$), the energy is infinite. The influence of the rays on the bulk lies outside the closed model.
\end{itemize}
A consistent version of the program is therefore contained in the following four concrete steps: (A$'$) the reduction with two closure functions is exact, and closed models are obtained by prescribing them; (B$'$) the apex dynamics is exact up to the pressure anisotropy at the origin, and keeping this anisotropy in a window implies blow-up of $(Em)$ itself (Theorem~\ref{thm:conditional}); (C$'$) as proved; (D$'$) the apex blow-up stays localized at the apex, with finite energy and a regular far field. The step from (B$'$) to a blow-up of $(E1)$, $(E2)$ from smooth finite-energy data is the control of the anisotropy, \S\ref{sec:exact}.

\subsection*{Open problems}
\begin{enumerate}[label=(\arabic*)]
\item \emph{Self-similar blow-up.} Prove that the profile system of Lemma~\ref{lem:profile} has a global solution (a connection problem for three autonomous ODEs), and prove its stability under perturbations that are smooth at the apex. This would prove Conjecture~\ref{conj:selfsim}, i.e.\ $T^*=\Tap$ with bounded energy for $(R2)$.
\item \emph{Gradient catastrophes.} Explain the competition between apex blow-up and gradient blow-up, in particular why it is decided differently for $m=1$ and $m=2$.
\item \emph{$(Zm)$ with $k<0$.} Explain the limit profile $\gb(\Tap,x)\sim Cx^{-0.13}$ observed in Section~\ref{sec:family}, and prove that the singularity stays at the apex.
\item \emph{Pressure anisotropy at the origin.} Find smooth finite-energy data of $(E1)$ or $(E2)$ for which $-P_{zz}(t,0)/P_{rr}(t,0)$ stays in a compact subinterval of $(0,m)$, or, more generally, for which the one-sided bound \eqref{eq:one-sided} on $P_{rr}(t,0)$ holds for as long as the solution is smooth. By Theorem~\ref{thm:conditional} and Proposition~\ref{prop:one-sided} this gives finite-time blow-up of the full equations. The companion paper \cite{ShiII} shows that this requires a non-columnar structure near the origin in which the swirl is weaker at the origin than nearby on the axis, and that smooth localized data do not keep it.
\item \emph{Zero flux.} For $0<k\le(m+3)/2$, characterize the solutions with $L(t)\equiv0$, which would have finite energy outside the windows.
\item \emph{The closure.} Study closures $f(t,x)$ with $f(t,0)=2$ other than $f\equiv2$, for instance those suggested by the Yan-type solution and its series expansions, and interpret the results in terms of the 2D Boussinesq ($m=1$) and 3D axisymmetric Euler ($m=2$) equations.
\end{enumerate}

\appendix
\section{Comparison with the construction of Alp\"oge, Buckmaster and Coiculescu}\label{app:ABC}

On 23 November 2026, M.~P.~Coiculescu described in a guest post on T.~Tao's blog \cite{Coiculescu2026} a joint work with L.~Alp\"oge and T.~Buckmaster \cite{ABC2026}. It announces a finite-time singularity of the 3D incompressible Euler equations from smooth data. We summarize the result as it is described in \cite{Coiculescu2026}, compare it with the mechanism of this paper, and record a consequence for our program (Proposition~\ref{prop:one-sided}). The preprint \cite{ABC2026} itself should be consulted for the precise statements; at the time of writing it has not yet been refereed.

\subsection{The result}
According to \cite{Coiculescu2026}, there is an odd, divergence-free initial velocity $u_0\in C^\infty(\mathbb T^3;\R^3)$ on the periodic box $\mathbb T^3=[-\frac12,\frac12]^3$, which has no boundary. Its maximal classical solution is smooth and odd on $\mathbb T^3\times[0,T_*)$ for a finite $T_*$, and
\[
\begin{gathered}
\int_0^{T_*}\|\nabla u(t)\|_{L^\infty}\,dt=\infty,\qquad \sup_{t<T_*}\|u(t)\|_{L^\infty}<\infty,\\
\lim_{t\uparrow T_*}|\nabla\times u(t,0)|=\infty,
\end{gathered}
\]
while $u(t)$ converges smoothly on compact sets away from the origin. Since the data are smooth on a compact domain, the energy is finite. Earlier results of this kind required limited regularity, $C^{1,\alpha}$ \cite{Elgindi21}, or smoothness only away from the origin \cite{CMZ2025}, or a boundary \cite{CH25}.

The construction is not self-similar. It is a multiscale iteration built from exact Craik--Criminale solutions: an affine flow $D(t)x$, where $D(t)$ is trace-free and $D'+D^2$ is symmetric, plus an oscillatory wave $\mathbf a(t)F(k(t)\cdot x)$ with $k'=-D^Tk$. At each stage the shear of the previous wave tilts and amplifies a new wave, and the gain per stage is of order $\exp(\epsilon^{-1/2})$ for a small coupling ratio $\epsilon$. The approximate solution is corrected to an exact one through a boundary-value problem for a Lagrangian displacement. Its coercivity rests on a uniform upper bound on the pressure Hessian, $H^\sharp\le\mathrm{Id}$, which is obtained by choosing the wave profile $F$ at each step; by \cite{Coiculescu2026}, $F(s)=\sin(2\pi s)$ would not work.

\subsection{Common ground}
\begin{enumerate}[label=(\roman*)]
\item \emph{A stagnation point at the origin.} In both settings the symmetry makes the origin a stagnation point, and the singularity is a blow-up of the velocity gradient there.
\item \emph{Affine flows and the pressure Hessian.} An affine flow $D(t)x$ solves the Euler equations exactly when $D'+D^2$ is symmetric, and then its pressure Hessian is $H=-(D'+D^2)$. For $m=2$ and $\lambda=a=b=1$, the velocity gradient of $(E2)$ at the origin is
\[
D(t)=\begin{pmatrix}-v_a&-u_a&0\\ u_a&-v_a&0\\ 0&0&2v_a\end{pmatrix}.
\]
Symmetry of $D'+D^2$ is the equation $u_a'=2u_av_a$, and the diagonal of $H=-(D'+D^2)$ gives $\pi_r=v_a'-v_a^2+u_a^2$ and $\pi_z=-2v_a'-4v_a^2$. These are exactly the relations \eqref{eq:momentum-origin}. The trace condition $\operatorname{tr}H=-\operatorname{tr}D^2$ is the Poisson equation at the origin, and the remaining freedom in $H$ is the anisotropy $\mu$ of Section~\ref{sec:exact}. Thus the exact apex dynamics of Proposition~\ref{prop:exact-apex} is the axisymmetric case of the affine flows on which \cite{ABC2026} builds.
\item \emph{The pressure Hessian is the crux.} In \cite{ABC2026} the key estimate is a one-sided bound on the pressure Hessian. In this paper, the missing step from the apex dynamics to $(E1)$, $(E2)$ is the control of the pressure Hessian at the origin (Theorem~\ref{thm:conditional}).
\item \emph{Vorticity at the blow-up point.} In \cite{ABC2026} the vorticity at the origin blows up. The same holds for the apex blow-up of Theorem~\ref{thm:main}: the axial vorticity at the origin is $\omega_z(t,0)=2u_a(t)$, and $|u_a|/v_a\to1/\sqrt\kappa$ (Lemma~\ref{lem:swirl-dominance}). By contrast, at a CLM blow-up point $u\equiv0$ (Theorem~\ref{thm:CLM-2}).
\end{enumerate}

\subsection{Differences}
\begin{enumerate}[label=(\roman*)]
\item \emph{Symmetry class.} The flow of \cite{ABC2026} is odd on $\mathbb T^3$ and is not described as axisymmetric. As far as we can tell from \cite{Coiculescu2026}, finite-time blow-up from smooth finite-energy data remains open for the axisymmetric Euler equations with swirl in $\R^3$ without boundary and for the 2D Boussinesq equations in $\R^2$, which are the settings $(E2)$ and $(E1)$ of this paper.
\item \emph{Mechanism.} In \cite{ABC2026} the growth comes from shear amplification of waves across infinitely many scales, which is a non-normal, transient mechanism. In this paper the growth comes from a single quadratic ODE at the apex, driven by the swirl through the pressure (Section~\ref{sec:mechanism}), with a conjectured self-similar inner layer (Section~\ref{sec:selfsim}).
\item \emph{Control of the pressure.} In \cite{ABC2026} the bound on the pressure Hessian is enforced by the choice of the wave profiles at each scale. In this paper the anisotropy is prescribed by a closure in the ray models, and it is assumed in Theorem~\ref{thm:conditional}.
\item \emph{Status.} \cite{ABC2026} announces a complete proof for the Euler equations. Our results are rigorous for the ray models and the apex dynamics, and conditional for $(E1)$ and $(E2)$.
\end{enumerate}

\subsection{A one-sided condition on the pressure}
The one-sided Hessian bound of \cite{ABC2026} suggests asking whether the two-sided window \eqref{eq:anisotropy-window} of Theorem~\ref{thm:conditional} can be replaced by a one-sided condition. It can. The relation $v_a'=v_a^2-u_a^2+\pi_r$ in \eqref{eq:momentum-origin} is exact and involves only the radial pressure Hessian, so a lower bound on $\pi_r$ alone is enough.

\begin{proposition}[Blow-up under a one-sided bound on the radial pressure Hessian]\label{prop:one-sided}
Let $m\in\{1,2\}$, $a=b=\lambda=1$, and let $(u,v,g,p)$ be a smooth even solution of $(Em)$ on $[0,T^*)$ with $u_a(0)>0$; for $m=1$ the variable $u^2=\vartheta/x_2$ is assumed positive at the origin. Suppose that there are constants $\delta>0$ and $C>1$ such that
\begin{equation}\label{eq:one-sided}
P_{rr}(t,0)\ge(1+\delta)\,u_a(t)^2-C\,v_a(t)^2\qquad\text{for all }t\in[0,T^*).
\end{equation}
Let $\bar\kappa:=2\delta/(2C-2+m^2)$. If $v_a(0)>-\sqrt{\bar\kappa}\,u_a(0)$, then $T^*\le\bar T<\infty$, where $\bar T$ is the blow-up time of
\[
\bar v'=(1-C)\bar v^2+\delta\bar u^2,\qquad \bar u'=\tfrac12m^2\bar u\bar v
\]
with the same data.
\end{proposition}

\begin{proof}
By \eqref{eq:momentum-origin} and \eqref{eq:one-sided}, $v_a'\ge(1-C)v_a^2+\delta u_a^2$, while $u_a'=\frac12m^2u_av_a$. The comparison system is the apex system \eqref{eq:apexODE} with $\lambda=1$, $\bar A=1-C<0$ and $\bar B=-\delta<0$, and $\bar\kappa=2\bar B/(2\bar A-m^2)$. From here the proof is that of Theorem~\ref{thm:conditional}: the comparison system is cooperative on $\{u>0\}$, the Kamke comparison theorem gives $v_a\ge\bar v$ and $u_a\ge\bar u$, and $\bar v\to+\infty$ at a finite $\bar T$ by Theorem~\ref{thm:apexclass}(b).
\end{proof}

\begin{remark}\label{rem:one-sided}
Proposition~\ref{prop:one-sided} contains Theorem~\ref{thm:conditional}. If $\mu_-\le\mu(t)\le\mu_+$ with $0<\mu_-\le\mu_+<m$, then by \eqref{eq:pressure-origin}
\[
\pi_r=\frac{m\,u_a^2}{m-\mu}-\frac{m(m+1)\,v_a^2}{m-\mu}\ge\frac{m}{m-\mu_-}\,u_a^2-\frac{m(m+1)}{m-\mu_+}\,v_a^2 .
\]
This is \eqref{eq:one-sided} with $\delta=\mu_-/(m-\mu_-)>0$ and $C=m(m+1)/(m-\mu_+)>m+1$, and the comparison systems coincide: $1-C=A(\mu_+)$ and $-\delta=B(\mu_-)$.

Proposition~\ref{prop:trajectory}(ii) extends Proposition~\ref{prop:one-sided} from the origin to any particle trajectory, with $P_{rr}(t,0)$ replaced by $\frac1r\partial_rp$ along the trajectory. Condition \eqref{eq:one-sided} is weaker than \eqref{eq:anisotropy-window} in two ways: it involves only $P_{rr}(t,0)$, not the axial component $P_{zz}(t,0)$, and it is a lower bound, not a window. Physically it says that the radial suction at the origin exceeds the centrifugal force (the buoyancy force for $m=1$) by a fixed fraction, up to a multiple of $v_a^2$. By Lemma~\ref{lem:swirl-dominance} this is what happens in the swirl-dominated regime. For the program of Section~\ref{sec:exact} it means that one only has to bound $P_{rr}(t,0)$ from below: the contribution of a swirling core near the origin must dominate the far-field contribution to $P_{rr}(t,0)$ from one side. The success of \cite{ABC2026} in enforcing a one-sided Hessian bound through scale separation suggests that this formulation of the problem is the right one.
\end{remark}

\section{Two related documents: the OpenAI notes and Zheng's exposition}\label{app:OpenAI}

Two further documents announce a finite-time singularity of the 3D Euler equations from smooth data. The first is a set of notes titled \emph{Finite time blowup for the Euler equation} and attributed to OpenAI \cite{OpenAI2026}; it carries no date. The second is F.~Zheng's preprint \emph{Finite-time blowup of the Euler equation without forcing} \cite{Zheng2026}, which presents itself as an exposition of the construction in \cite{OpenAI2026}. Neither has been refereed. We summarize them here, cite them, and record the ideas they suggest for the present work. We do not use their results: every statement in Section~\ref{subsec:B-ideas} is derived and proved here.

\subsection{Summary}
Both documents state the existence of a divergence-free $u_0\in C^\infty_c(\R^3)$ whose smooth solution has a finite maximal time $T_*$ with
\[
\limsup_{t\uparrow T_*}\|\nabla u(t)\|_{L^\infty}=\infty,\qquad \int_0^{T_*}\|\operatorname{curl}u(t)\|_{L^\infty}\,dt=\infty .
\] Since the data are smooth and compactly supported, the energy is finite. The construction proceeds in infinitely many stages:
\begin{itemize}
\item \emph{Odd flows.} All intermediate flows are odd, $u(t,-x)=-u(t,x)$, so the origin is a fixed stagnation point, $X(t,0)=0$, and the growth is measured there.
\item \emph{Packets.} At stage $j$ a localized high-frequency packet is added to the parent flow. The rank-one shear of the parent rotates the wave normal and amplifies the packet, and the corrected child flow becomes the parent of the next stage.
\item \emph{Accumulation.} Target times $t_j\uparrow T_\infty<\infty$ are chosen so that $|\nabla U^j(t_j,0)|\to\infty$ while the increments of the initial data are summable in every $H^m$. Stability then gives $T_*\le T_\infty$.
\item \emph{Pressure.} Along the central trajectory the deformation gradient satisfies $F_t=MF$ and $F_{tt}=-HF$, with $M=\nabla u$ and $H=\nabla^2p$. A one-sided bound $H\le K_+I$, maintained with the same constant at every stage, gives the coercivity of the correction problem. The leading change of $H$ caused by a packet is rank one, proportional to $(m\cdot Mv)\,\frac{m\otimes m}{|m|^2}f_\delta'(\theta)$, where $m$ is the wave normal and $v$ the packet amplitude; its sign is controlled by the single scalar $m\cdot Mv$ and by a lower bound on the profile derivative $f_\delta'$.
\end{itemize}
The construction is close in spirit to that of Appendix~\ref{app:ABC}: both build on Craik--Criminale waves over odd flows, both work at a stagnation point at the origin, and both rely on a one-sided bound on the pressure Hessian. It is set on $\R^3$ with compactly supported data rather than on $\mathbb T^3$. It is not described as axisymmetric, so, as far as these documents show, it does not cover the axisymmetric Euler equations with swirl or the 2D Boussinesq equations.

\subsection{Ideas for the present work}\label{subsec:B-ideas}

\emph{1. The Lagrangian form of the apex dynamics.} The law $F_{tt}=-HF$ suggests writing the apex dynamics in terms of the deformation at the origin. Let $a=b=\lambda=1$ and define the radial and axial stretches of the fluid at the origin by
\[
R(t)=\exp\Bigl(-\int_0^tv_a\Bigr),\qquad Z(t)=\exp\Bigl(m\int_0^tv_a\Bigr).
\]

\begin{lemma}[Lagrangian apex dynamics]\label{lem:Lagrangian-apex}
For a smooth even solution of $(Em)$,
\[
R^mZ=1,\qquad u_a^2R^{m^2}=u_a(0)^2,\qquad R''=(u_a^2-\pi_r)R=u_a(0)^2R^{1-m^2}-\pi_rR,
\]
and $v_a=-R'/R$. In particular $v_a\to+\infty$ at a time $T$ if and only if $R(t)\to0$ as $t\uparrow T$ while $R'/R\to-\infty$.
\end{lemma}

\begin{proof}
$R'=-v_aR$ and $Z'=mv_aZ$ give $(R^mZ)'=0$. From $u_a'=\frac12m^2u_av_a$ one gets $(u_a^2R^{m^2})'=m^2u_a^2v_aR^{m^2}-m^2u_a^2v_aR^{m^2}=0$. Finally $R''=(-v_a'+v_a^2)R=(u_a^2-\pi_r)R$ by \eqref{eq:momentum-origin}.
\end{proof}

For $m=2$ the equation $R''=u_a(0)^2R^{-3}-\pi_rR$ is an Ermakov--Pinney equation forced by the pressure. It gives a transparent picture of Section~\ref{sec:mechanism}: the apex blow-up is the collapse of a fluid ring onto the axis. The centrifugal barrier $u_a(0)^2R^{-3}$ (for $m=1$, the constant buoyancy force $u_a(0)^2$) resists the collapse, and the suction $\pi_rR$ drives it. Conservation of angular momentum, respectively of buoyancy, appears as $u_a^2R^{m^2}=\mathrm{const}$. Condition \eqref{eq:one-sided} reads $\pi_rR\ge(1+\delta)u_a(0)^2R^{1-m^2}-CR'^2/R$: the suction must overcome the barrier. This is the axisymmetric case of $F_{tt}=-HF$. For $m=1$ it is a linear equation with a constant forcing, and for $m=2$ it is linear in $R$ apart from the barrier term. A direct analysis of it, in the spirit of comparison theorems for Hill's equation $R''+\pi_rR=0$, may allow conditions on time averages of $\pi_r$ rather than pointwise ones.

In the swirl-dominated regime of Theorem~\ref{thm:main} the suction is large. With $0<\mu<m$ constant and the apex values in case (b1) or on the separatrix, $v_a\simeq2/(m^2(T-t))$ and $u_a^2/v_a^2\to1/\kappa$ (Lemma~\ref{lem:swirl-dominance}), so \eqref{eq:pressure-origin} and the identity of Lemma~\ref{lem:swirl-dominance} give
\begin{equation}\label{eq:pr-rate}
(T-t)^2P_{rr}(t,0)\to\frac{2(m+2)}{\mu m^2}.
\end{equation}
For the main example ($m=2$, $\mu=\frac16$) the limit is $12$. The radial pressure Hessian at the origin must therefore blow up at the rate $(T-t)^{-2}$, with a positive constant, which is the scale-invariant rate.

\emph{2. One-sided bounds.} In \cite{OpenAI2026} and in \cite{ABC2026} the pressure Hessian is bounded from above, for coercivity. In Proposition~\ref{prop:one-sided} and Proposition~\ref{prop:trajectory}(ii) the radial component is bounded from below, because there it drives the collapse. The two uses are different, but both show that a one-sided control of the pressure Hessian suffices.

\emph{3. Locality of the pressure for separated scales.} In \cite{OpenAI2026} the change of the pressure Hessian caused by a high-frequency packet is computed locally at leading order, and the nonlocal remainder is smaller by a power of the frequency. In our setting the exact core of Theorem~\ref{thm:core} shrinks like $(\Tap-t)^{2/m^2}$, so the core and the far field have separated scales near the blow-up time. This suggests a two-scale analysis of $P_{rr}(t,0)$: compute the contribution of the core locally, and show that the far field contributes only a lower-order term relative to \eqref{eq:pr-rate}. This is the program of Section~\ref{sec:exact}, made more concrete.

\emph{4. Stage-wise control.} The constructions of \cite{ABC2026,OpenAI2026} do not need a closure that holds uniformly in time: each stage needs control only on its own time interval, and the singular time is the accumulation of the stage times. Theorem~\ref{thm:conditional} and Proposition~\ref{prop:one-sided} assume their conditions on the whole interval $[0,T^*)$. A stage-wise version would require the one-sided bound only on a sequence of intervals on which $v_a$ gains a fixed factor, with summable interval lengths. Whether such a version can be closed for the axisymmetric equations is open.

\section*{Acknowledgements}
The author introduced the definition of the unified system $(Em)$ and of the ray systems $(Rm)$ and $(Zm)$, and also identified the corresponding invariant apex dynamics ODEs. The author also derived $(Em)$ from the 2D inviscid Boussinesq equations and the 3D axisymmetric Euler equations, found the correct treatment of the angular pressure terms on the axes, introduced the closure ansatz \eqref{eq:Pthetatheta-at-theta-0} and the apex condition $f(t,0)=2$, chose the energies $E_m$, constructed the Yan-type exact solution and its series expansions (used to determine the sign of $f(t,0)$), found the normalization $\lambda=m^2(q+1)$, $b/a=m^2(mq+2)/(m^2q-2)$ of the apex dynamics and its CLM-2 case, and posed the problem studied here.

Much of the analysis, the numerical work and the preparation of the manuscript was carried out with the assistance of Claude (Anthropic), an AI model. Claude's contributions were:
\begin{enumerate}[label=(\alph*)]
\item derivation and numerical verification of the pressure Poisson equations for the 3D axisymmetric Euler equations and the 2D inviscid Boussinesq equations, which the author used to check the derivation of $(Em)$ and of the ray systems;
\item the Euler form, the splitting lemma and the nonlocal transport form of $(Rm)$ and $(Zm)$ for $f\equiv2$ (Propositions~\ref{prop:euler}, \ref{prop:transport}, Lemmas~\ref{lem:split}, \ref{lem:K}), and the classification of the parameter regimes (Remark~\ref{rem:regimes});
\item the reduction of the apex dynamics to one ODE system for both rays, its equivalence with CLM-$q$, the ratio variable, the first integral and the classification of its solutions (Section~\ref{sec:apex});
\item the Lagrangian well-posedness theory and the bound $T^*\le\Tap$ (Section~\ref{sec:lwp});
\item the domain-of-dependence argument, the exact cores and their energy (Section~\ref{sec:core});
\item the far-field expansion, the energy criterion, the compact-support property for $k<0$ and the finite-energy singularity theorem (Section~\ref{sec:energy});
\item the Lagrangian numerical method and the numerical study of Section~\ref{sec:num};
\item the verification of the Yan-type solution and of its first-order expansion, and the test of the closure on it (Proposition~\ref{prop:Yan-closure});
\item the explicit solutions in the CLM-2 case and the non-elementary character of the case $m=1$ (Example~\ref{ex:CLM2}, Remark~\ref{rem:m1-elliptic});
\item the power-mode analysis, the self-similar profile equations and their numerical solution, and the identification of the short-wave instability (Section~\ref{sec:selfsim});
\item the analysis of the exact apex dynamics with two closure functions, the meaning of the scaling parameters, and the conditional blow-up theorem for $(Em)$ (Section~\ref{sec:exact});
\item the analysis of the normalized parameter family, the local form of $(Zm)$ at $b=2am$, and the computations for $(Zm)$ with $k<0$ (Section~\ref{sec:family});
\item the comparison with related blow-up mechanisms in the literature (Section~\ref{sec:compare});
\item the extension of Chae's criteria to particle trajectories with swirl (Proposition~\ref{prop:trajectory}, Corollary~\ref{cor:Chae-swirl});
\item the comparison of Theorem~\ref{thm:main} with the CLM blow-up, the physical picture of the swirl-driven collapse, and the hypotheses of Theorem~\ref{thm:main} in terms of the velocity (Section~\ref{sec:mechanism}, Remark~\ref{rem:data-velocity});
\item the analyses of the announcements of blow-up for the Euler equations, including the Lagrangian form of the apex dynamics (Appendices~\ref{app:ABC} and~\ref{app:OpenAI}, Lemma~\ref{lem:Lagrangian-apex});
\item editing of Sections~\ref{sec:derivation}--\ref{Yan-exact-solution} and drafting the rest of the manuscript in \LaTeX.
\end{enumerate}
The author takes full responsibility for the content of this paper.

\section*{Statements and Declarations}
\noindent\textbf{Competing interests.} The author has no competing interests to declare that are relevant to the content of this article.

\smallskip
\noindent\textbf{Funding.} No funding was received for conducting this study.

\smallskip
\noindent\textbf{Data availability.} No datasets were generated or analysed in this study. All numerical results can be reproduced from the methods and parameters described in Sections~\ref{sec:num}, \ref{sec:selfsim} and~\ref{sec:family}.

\smallskip
\noindent\textbf{Code availability.} The code used for the computations is available from the author on reasonable request.

\smallskip
\noindent\textbf{Author contributions.} The author is the sole author of this article and is responsible for its entire content.

\smallskip
\noindent\textbf{Use of AI tools.} An AI assistant (Claude, Anthropic) was used in parts of the analysis, the numerical computations and the preparation of the manuscript, as described in detail in the Acknowledgements. The author checked all results and takes full responsibility for the content.

\ifshowchanges
\clearpage
\section*{Changes with respect to the submitted version}\label{sec:changes}
\emph{This section is an editorial record for the revision and is not part of the paper. It compares the present draft with the version submitted to J.\ Math.\ Fluid Mech.\ and posted on arXiv (draft~17, 22 November 2026). Set the switch \texttt{showchanges} to false in the preamble to remove it.}

\subsection*{A. Material of the submitted version that is no longer here}
None. Every section, subsection, numbered statement, labelled equation, table and bibliography entry of draft~17 is still present (all $186$ labels). No paragraph of draft~17 was deleted. The material moved to the companion paper \cite{ShiII} was developed after the submission (drafts~18--42) and was never part of draft~17.

Text of draft~17 that was reworded:
\begin{enumerate}[label=(A\arabic*)]
\item \emph{Abstract, last sentence.} Draft~17: \emph{``and we prove that keeping it in an explicit window forces finite-time blow-up of $(E1)$ and $(E2)$ themselves.''} Now the statement is conditional in form (``we prove that if this anisotropy stays in an explicit window, $(E1)$ and $(E2)$ themselves blow up in finite time''). The abstract was also shortened in other places, so that it stays within the limit of 150--250 words of the journal after the additions in (B1). Draft~17: 223 words; now 237 (249 with hyphenated compounds counted separately).
\item \emph{Introduction, item ``From the apex to $(E1)$, $(E2)$''.} The same conditional rewording, followed by the added remark that the window is a hypothesis on the solution for its whole lifetime.
\item \emph{Introduction, first paragraph.} ``is'' became ``has long been''.
\item \emph{Conclusion, open problem (4).} Draft~17 asked only for data keeping $-P_{zz}/P_{rr}$ in a compact subinterval of $(0,m)$. It now also allows the one-sided bound \eqref{eq:one-sided}, cites Proposition~\ref{prop:one-sided}, and summarizes the findings of the companion paper.
\end{enumerate}

\emph{Numbering.} Three numbers changed: Section~14 (comparison) is now Section~\ref{sec:compare}, Section~15 (conclusion) is now Section~\ref{sec:concl}, and Table~7 is now Table~\ref{tab:compare}. The subsection ``What remains to be proved'' moved from 13.6 to~\ref{subsec:remains}. All theorems, propositions, lemmas, remarks and labelled equations of draft~17 keep their numbers.

\subsection*{B. New material not in the submitted version}
\begin{enumerate}[label=(B\arabic*)]
\item \emph{Abstract.} One sentence on the role of the swirl: unlike in the CLM model, the swirl drives the blow-up through a low-pressure core. One sentence referring to the companion paper.
\item \emph{Introduction.}
\begin{itemize}
\item Citation of the trajectory criteria of Chae \cite{Chae2008,Chae2010}.
\item A paragraph on the announced Euler singularity of Alp\"oge, Buckmaster and Coiculescu \cite{ABC2026,Coiculescu2026} and on \cite{OpenAI2026,Zheng2026}.
\item New items ``Physical mechanism'' and ``Companion paper''.
\item Items (vi)--(ix) of the note on this revision, and the corresponding sentences in ``Organization''.
\end{itemize}
\item \emph{Section~\ref{sec:energy}.} Remark~\ref{rem:data-velocity} (the hypotheses of Theorem~\ref{thm:main} in terms of the velocity), and three paragraphs on the strain geometry at the blow-up time.
\item \emph{Section~\ref{subsec:Chae}} (new subsection), ``Criteria along particle trajectories and the results of Chae'': Proposition~\ref{prop:trajectory} (blow-up along a particle trajectory) and Corollary~\ref{cor:Chae-swirl} (Chae's dichotomy with swirl), with \eqref{eq:complex-Riccati} and \eqref{eq:one-sided-traj}.
\item \emph{Section~\ref{sec:mechanism}} (new section), ``The physical mechanism of the apex blow-up'':
\begin{itemize}
\item six subsections;
\item Proposition~\ref{prop:pressure-origin} and Lemma~\ref{lem:swirl-dominance};
\item equations \eqref{eq:momentum-origin}, \eqref{eq:pressure-origin} and \eqref{eq:apex-mu};
\item Table~\ref{tab:CLM-vs-main}, comparing the CLM blow-up with the apex blow-up;
\item the swirl-free case (Section~\ref{subsec:swirl-free}), including the remark that collapse with axial compression does occur for data of limited regularity \cite{Elgindi21,Shkoller2026}.
\end{itemize}
\item \emph{Section~\ref{sec:compare}.} A paragraph on Hou's interior scenario \cite{Hou2023} and its symmetry class.
\item \emph{Conclusion.} A bullet on the physical mechanism; the extension of open problem (4), see (A4).
\item \emph{Appendix~\ref{app:ABC}} (new), comparison with the construction of Alp\"oge, Buckmaster and Coiculescu. It contains Proposition~\ref{prop:one-sided} (blow-up under a one-sided bound on the radial pressure Hessian), equations \eqref{eq:one-sided} and \eqref{eq:pr-rate}, and Remark~\ref{rem:one-sided}.
\item \emph{Appendix~\ref{app:OpenAI}} (new), on the OpenAI notes and Zheng's exposition. It contains Lemma~\ref{lem:Lagrangian-apex} (Lagrangian apex dynamics) and Section~\ref{subsec:B-ideas}.
\item \emph{Acknowledgements.} Three new items in the list of contributions of Claude (Proposition~\ref{prop:trajectory} and Corollary~\ref{cor:Chae-swirl}; Section~\ref{sec:mechanism} and Remark~\ref{rem:data-velocity}; Appendices~\ref{app:ABC} and~\ref{app:OpenAI}).
\item \emph{Scale-invariant solutions.} The $\theta$-dependent apex dynamics, developed in drafts~46--59 (first as a subsection of Section~\ref{sec:exact}, from draft~56 as an appendix), has been moved to a separate paper \cite{ShiIII}.
\item \emph{Bibliography.} Twelve new entries.
\begin{itemize}
\item Published: \cite{Chae2008,Chae2010,Hou2023}.
\item Published or in press, to be checked: \cite{CMZ2025}.
\item Preprints, announcements or expository documents, which are cited but whose results are not used: \cite{ABC2026,Coiculescu2026,OpenAI2026,Zheng2026,AB2026NS,OpenAI2026NS,Shkoller2026}.
\item The companion paper: \cite{ShiII}.
\end{itemize}
\end{enumerate}

\subsection*{C. Points to check before resubmission}
\begin{itemize}
\item The new results of Section~\ref{subsec:Chae}, Section~\ref{sec:mechanism} and Appendices~\ref{app:ABC} and~\ref{app:OpenAI} add about 14 pages (51 pages in draft~17, 65 now without this section). If the referees ask for a shorter paper, Appendices~\ref{app:ABC} and~\ref{app:OpenAI} can be shortened or moved to the companion paper, as only Proposition~\ref{prop:one-sided} is used in the main text.
\item The bibliographic status of the preprints listed in (B11), and the publication status of the companion paper, should be updated at the time of resubmission.
\item The note on this revision in the Introduction compares with the earlier arXiv versions of this preprint, not with draft~17. A response to the referees should use Parts~A and~B of this section instead.
\end{itemize}
\fi

\end{document}